\pdfoutput=1
\documentclass[a4paper,11pt]{amsart}
\usepackage[hmarginratio={1:1},vmarginratio={1:1},lmargin=80.4pt,tmargin=90.0pt]{geometry}

\usepackage{ifdraft}
\ifdraft{\usepackage[draft]{showkeys}}{\usepackage[final]{showkeys}}
\usepackage{calligra}
\usepackage[T1]{fontenc}
\usepackage[ugly]{nicefrac}

\usepackage{mathtools}
\mathtoolsset{mathic}
\usepackage[final]{microtype}
\usepackage{multicol}
\usepackage{tabularx,tabu}
\usepackage{latexsym,exscale,enumitem,amsfonts,amssymb}
\usepackage[leqno]{amsmath}
\usepackage{amsthm,amsfonts,amssymb,amscd,textcomp,bbm}
\usepackage{stmaryrd}
\SetSymbolFont{stmry}{bold}{U}{stmry}{m}{n}
\usepackage[normalem]{ulem}
\usepackage{thmtools}
\usepackage{etex}
\usepackage{fancybox}
\usepackage[table]{xcolor}
\usepackage{mathrsfs}
\DeclareMathAlphabet{\mathscrbf}{OMS}{mdugm}{b}{n}
\DeclareFontEncoding{LS1}{}{}
\DeclareFontSubstitution{LS1}{stix}{m}{n}
\DeclareMathAlphabet{\mathpzc}{LS1}{stixscr}{m}{n}

\usepackage{caption} 
\allowdisplaybreaks[4]

\makeatletter
\newcommand{\leqnomode}{\tagsleft@true\let\veqno\@@leqno}
\newcommand{\reqnomode}{\tagsleft@false\let\veqno\@@eqno}
\makeatother

\definecolor{mygray}{gray}{0.6}
\definecolor{mygraydark}{gray}{0.4}
\definecolor{mygraylight}{gray}{0.8}

\definecolor{cherry}{RGB}{135,135,135}
\definecolor{infinity}{RGB}{110,110,110}
\definecolor{cream}{RGB}{255,253,208}
\definecolor{corn}{RGB}{251,236,93}
\definecolor{citron}{RGB}{190,180,90}

\definecolor{spinach}{RGB}{46,139,87}
\definecolor{tomato}{RGB}{255,99,71}
\definecolor{pumpkin}{RGB}{224,180,80}

\definecolor{orchid}{RGB}{143,40,194}
\definecolor{lava}{RGB}{207,16,32}
\definecolor{mydarkblue}{RGB}{10,10,150}

\definecolor{posbullet}{RGB}{170,10,10}
\definecolor{negbullet}{RGB}{10,10,170}

\definecolor{myorange}{RGB}{225,127,0}
\definecolor{mygreen}{RGB}{0,225,0}
\definecolor{mypurple}{RGB}{128,0,128}
\definecolor{myred}{RGB}{255,0,0}
\definecolor{myblue}{RGB}{0,0,195}
\definecolor{myyellow}{RGB}{210,210,0}

\usepackage[all]{xy}
\SelectTips{cm}{}

\usepackage{tikz}
\usetikzlibrary{cd}
\usetikzlibrary{decorations}
\usetikzlibrary{decorations.markings}
\usetikzlibrary{decorations.pathreplacing}
\usetikzlibrary{decorations.pathmorphing}
\usetikzlibrary{arrows.meta,shapes,positioning,matrix,calc}
\usetikzlibrary{shapes.callouts}
\usetikzlibrary{patterns}
\tikzset{anchorbase/.style={baseline={([yshift=-0.5ex]current bounding box.center)}},
  tinynodes/.style={font=\tiny,text height=0.5ex,text depth=0.1ex},
  smallnodes/.style={font=\scriptsize,text height=0.75ex,text depth=0.15ex},
  pole/.style={line width=3.0,color=cherry},
  usual/.style={line width=.9,color=black},
  cpole/.style={line width=3.0,color=cherry,densely dashed},
  ccpole/.style={line width=2.0,color=cherry,densely dashed},
  cusual/.style={line width=.9,densely dashed},
  crossline/.style={preaction={draw=white,line width=4.75pt,-},preaction={draw=black,line width=.9pt,-}},
  crosspole/.style={preaction={draw=white,line width=6.0pt,-},preaction={draw=cherry,line width=3.0pt,-}},
  ssoergel/.style={line width=1.0,color=black},
  polebs/.style={line width=2.0,color=cherry},
  usualbs/.style={line width=.9,color=black},
  crosspolebs/.style={preaction={draw=white,line width=5.0pt,-},preaction={draw=cherry,line width=2.0pt,-}},
  moviestyle/.style={line width=.75,color=black},
  movieframe/.style={line width=.75,color=black,pattern=horizontal lines,pattern color=black},
  specialsymbol/.style={draw=none,every to/.append style={edge node={node [sloped, allow upside down, auto=false]{$#1$}}}},
}
\tikzstyle directed=[postaction={decorate,decoration={markings,
    mark=at position #1 with {\arrow[line width=0.225mm, black]{>}}}}]
\tikzstyle rdirected=[postaction={decorate,decoration={markings,
    mark=at position #1 with {\arrow[line width=0.225mm, black]{<}}}}]
\tikzstyle sdirected=[postaction={decorate,decoration={markings,
    mark=at position #1 with {\arrow[line width=0.225mm, black]{>}}}}]
\tikzstyle rsdirected=[postaction={decorate,decoration={markings,
    mark=at position #1 with {\arrow[line width=0.225mm, black]{<}}}}]
    
\allowdisplaybreaks

\makeatletter
\newcommand{\setword}[2]{%
  \phantomsection
  #1\def\@currentlabel{\unexpanded{#1}}\label{#2}%
}
\makeatother

\newcommand{\tp}{\text{+}}
\newcommand{\tm}{\text{-}}

\newcommand{\ie}{\textsl{i.e. }}
\newcommand{\eg}{\textsl{e.g. }}
\newcommand{\cf}{\textsl{cf. }}
\newcommand{\etc}{\textsl{etc.}}

\newcommand{\ver}{\textsl{verbatim }}

\newcommand{\loccit}{\textsl{loc. cit. }}

\newcommand{\C}{\mathbb{C}}
\newcommand{\Z}{\mathbb{Z}}

\newcommand{\K}{\mathbb{K}}
\newcommand{\N}[1][0]{\mathbb{N}_{#1}}

\newcommand{\slm}[1][m]{\mathfrak{sl}_{#1}}

\newcommand{\placeholder}{{-}}

\newcommand{\acts}{\cdot}

\newcommand{\topstuff}[1]{\mathpzc{#1}}
\newcommand{\elstuff}[1]{\mathtt{#1}}

\newcommand{\setstuff}[1]{\mathrm{#1}}
\newcommand{\catstuff}[1]{\mathcal{#1}}
\newcommand{\twocatstuff}[1]{\mathscrbf{#1}}

\newcommand{\obstuff}[1]{\mathtt{#1}}
\newcommand{\morstuff}[1]{\mathrm{#1}}

\newcommand{\Hom}{\mathrm{Hom}}
\newcommand{\HOM}{\mathrm{HOM}}
\newcommand{\Ext}{\mathrm{Ext}}
\newcommand{\EXT}{\mathrm{EXT}}
\newcommand{\iHOM}{\mathcal{HOM}^{\bullet}}

\newcommand{\twoEnd}{\mathbf{End}}

\newcommand{\ouriso}[1][\cong]{#1}

\newcommand{\typeA}[1][]{\mathsf{A}_{#1}}
\newcommand{\typeB}[1][]{\mathsf{B}_{#1}}
\newcommand{\typeC}[1][]{\mathsf{C}_{#1}}
\newcommand{\atypeA}[1][]{\tilde{\mathsf{A}}_{#1}}
\newcommand{\etypeA}[1][]{\hat{\mathsf{A}}_{#1}}
\newcommand{\atypeC}[1][]{\tilde{\mathsf{C}}_{#1}}

\newcommand{\hbody}[1][g]{\topstuff{H}_{#1}}
\newcommand{\gsurface}[1][g]{\Sigma_{#1}}
\newcommand{\thesphere}[1][3]{\topstuff{S}^{#1}}
\newcommand{\thedisk}[1][2]{\topstuff{D}^{#1}}

\newcommand{\braid}[1][b]{\topstuff{#1}}
\newcommand{\onebraid}{\raisebox{.025cm}{$\uparrow$}}
\newcommand{\bclosure}[1][\braid]{\overline{#1}}
\newcommand{\link}[1][l]{\topstuff{#1}}

\newcommand{\braidg}[1][]{\topstuff{B}\mathrm{r}(#1)}
\newcommand{\braidd}[1][]{\setstuff{B}\mathrm{r}(#1)}
\newcommand{\braidq}[1][]{\topstuff{B}\mathrm{r}(#1)^{\mathrm{Ma}}}

\newcommand{\tgen}[1][k]{\topstuff{t}_{#1}}
\newcommand{\bgen}[1][i]{\topstuff{b}_{#1}}

\newcommand{\coxdia}{\Gamma}

\newcommand{\artintits}[1][\coxdia]{\setstuff{AT}(#1)}
\newcommand{\coxgroup}[1][\coxdia]{\setstuff{W}(#1)}

\newcommand{\atgen}[1][i]{\beta_{#1}}

\newcommand{\coxgen}[1][i]{\sigma_{#1}}

\newcommand{\symgroup}[1][\coxdia]{\setstuff{S}(#1)}
\newcommand{\psymgroup}[2]{\setstuff{S}_{#2}(#1)}

\newcommand{\qpar}{\mathbf{q}}
\newcommand{\tpar}{\mathbf{t}}
\newcommand{\apar}{\mathbf{a}}
\newcommand{\hpar}{\mathbf{h}}
\newcommand{\xpar}{\mathbf{x}}

\newcommand{\qdeg}[1][\qpar]{{#1}\mathrm{deg}}
\newcommand{\qdim}[1][\qpar]{{#1}\mathrm{dim}_{\K}}

\newcommand{\sSbim}[1][\coxdia]{\mathbf{SS}^{\qpar}(#1)}

\newcommand{\hcat}[1]{\catstuff{K}^{b}\big({#1}\big)}

\newcommand{\dcat}[1]{\catstuff{D}^{b}\big({#1}\big)}
\newcommand{\ddcat}[1]{\catstuff{D}^{b}({#1})}
\newcommand{\dtimes}{\otimes^{\mathsf{L}}}

\newcommand{\rring}[1][]{\setstuff{R}^{#1}}
\newcommand{\rooty}[1][i]{\alpha_{#1}}

\newcommand{\dotup}[2]{\Delta_{#1}^{#2}}

\newcommand{\dotdown}[2]{\mu_{#1}^{#2}}

\newcommand{\stuple}[1][I]{\obstuff{#1}}
\newcommand{\num}[1]{\#{#1}}

\newcommand{\psplit}[2]{{}_{#2}\morstuff{S}{}_{#1}}
\newcommand{\pmerge}[2]{{}_{#2}\morstuff{M}{}_{#1}}

\newcommand{\ualph}[1][1]{\mathbb{X}_{#1}}
\newcommand{\yalph}[1][1]{\mathbb{Y}_{#1}}
\newcommand{\efunc}{\elstuff{e}}

\newcommand{\frobel}{\elstuff{a}}

\newcommand{\mtuple}{\obstuff{M}}

\newcommand{\khbracket}[2]{\left\llbracket{#1}\right\rrbracket_{#2}}

\newcommand{\ering}[1]{{#1}\otimes{#1}^{\mathrm{op}}}
\newcommand{\eering}[1]{{#1}\otimes{#1}}

\newcommand{\ssbim}[1][\stuple]{\mathsf{SS}^{\qpar}(#1)}

\newcommand{\cHH}[1][\bullet]{\catstuff{HH}^{#1}}
\newcommand{\cHHH}[2]{\catstuff{HHH}^{\bullet}_{#2}\left(#1\right)}

\newcommand{\vecq}[1][\K]{{#1}\mathsf{Vec}^{\qpar}}

\newcommand{\vecaq}[1][\K]{{#1}\mathsf{Vec}^{\apar\qpar}}
\newcommand{\vecatq}[1][\K]{{#1}\mathsf{Vec}^{\apar\tpar\qpar}}

\newcommand{\extalg}[1][\bullet]{{\textstyle\bigwedge^{#1}}}

\newcommand{\bimodq}[1]{{#1}\mathsf{Bim}^{\qpar}}

\newcommand{\ifunctor}{\catstuff{I}}
\newcommand{\pfunctor}{\catstuff{T}}

\newtheorem{theoremm}{Theorem}[section]

\declaretheorem[style=plain,name=Theorem,numberlike=theoremm]{theorem}
\declaretheorem[style=plain,name=Lemma,numberlike=theoremm]{lemma}
\declaretheorem[style=plain,name=Proposition,numberlike=theoremm]{proposition}

\declaretheorem[style=plain,name=Theorem,qed=$\square$,numberlike=theoremm]{theoremqed}

\declaretheorem[style=plain,name=Corollary,qed=$\square$,numberlike=theoremm]{corollary}

\declaretheorem[style=definition,name=Definition,numberlike=theorem]{definition}

\declaretheorem[style=remark,name=Example,numberlike=theorem]{example}
\declaretheorem[style=remark,name=Remark,numberlike=theorem]{remark}
\declaretheorem[style=remark,name=Convention,numberlike=theorem]{convention}

\allowdisplaybreaks

\numberwithin{equation}{section}
\usepackage[hypertexnames=false]{hyperref}
\usepackage{cleveref,bookmark}
\hypersetup{
    pdftoolbar=true,        
    pdfmenubar=true,        
    pdffitwindow=false,     
    pdfstartview={FitH},    
    pdftitle={HOMFLYPT homologies for links in handlebodies},    
    pdfauthor={David E.V. Rose and Daniel Tubbenhauer},     
    pdfsubject={},   
    pdfcreator={David E.V. Rose and Daniel Tubbenhauer},   
    pdfproducer={David E.V. Rose and Daniel Tubbenhauer}, 
    pdfkeywords={}, 
    pdfnewwindow=true,      
    colorlinks=true,       
    linkcolor=mydarkblue,          
    citecolor=teal,        
    filecolor=magenta,      
    urlcolor=orchid,          
    linkbordercolor=lava,
    citebordercolor=teal,
    urlbordercolor=orchid,  
    linktocpage=true
}
\let\fullref\autoref
\def\makeautorefname#1#2{\expandafter\def\csname#1autorefname\endcsname{#2}}
\makeautorefname{equation}{equation}%
\makeautorefname{gather}{equation}%
\makeautorefname{align}{equation}%
\makeautorefname{footnote}{footnote}%
\makeautorefname{item}{item}%
\makeautorefname{figure}{Figure}%
\makeautorefname{table}{Table}%
\makeautorefname{part}{Part}%
\makeautorefname{appendix}{Appendix}%
\makeautorefname{chapter}{Chapter}%
\makeautorefname{section}{Section}%
\makeautorefname{subsection}{Section}%
\makeautorefname{subsubsection}{Section}%
\makeautorefname{paragraph}{Paragraph}%
\makeautorefname{subparagraph}{Paragraph}%
\makeautorefname{theorem}{Theorem}%
\makeautorefname{theo}{Theorem}%
\makeautorefname{thm}{Theorem}%
\makeautorefname{addendum}{Addendum}%
\makeautorefname{addend}{Addendum}%
\makeautorefname{add}{Addendum}%
\makeautorefname{maintheorem}{Main theorem}%
\makeautorefname{mainthm}{Main theorem}%
\makeautorefname{corollary}{Corollary}%
\makeautorefname{corol}{Corollary}%
\makeautorefname{coro}{Corollary}%
\makeautorefname{cor}{Corollary}%
\makeautorefname{lemma}{Lemma}%
\makeautorefname{lemm}{Lemma}%
\makeautorefname{lem}{Lemma}%
\makeautorefname{sublemma}{Sublemma}%
\makeautorefname{sublem}{Sublemma}%
\makeautorefname{subl}{Sublemma}%
\makeautorefname{proposition}{Proposition}%
\makeautorefname{proposit}{Proposition}%
\makeautorefname{propos}{Proposition}%
\makeautorefname{propo}{Proposition}%
\makeautorefname{prop}{Proposition}%
\makeautorefname{property}{Property}
\makeautorefname{proper}{Property}
\makeautorefname{scholium}{Scholium}%
\makeautorefname{step}{Step}%
\makeautorefname{conjecture}{Conjecture}%
\makeautorefname{conject}{Conjecture}%
\makeautorefname{conj}{Conjecture}%
\makeautorefname{question}{Question}
\makeautorefname{questn}{Question}
\makeautorefname{quest}{Question}
\makeautorefname{ques}{Question}
\makeautorefname{qn}{Question}
\makeautorefname{definition}{Definition}%
\makeautorefname{defin}{Definition}%
\makeautorefname{defi}{Definition}%
\makeautorefname{def}{Definition}%
\makeautorefname{dfn}{Definition}%
\makeautorefname{notation}{Notation}
\makeautorefname{nota}{Notation}
\makeautorefname{notn}{Notation}
\makeautorefname{remark}{Remark}%
\makeautorefname{rema}{Remark}%
\makeautorefname{rem}{Remark}%
\makeautorefname{rmk}{Remark}%
\makeautorefname{rk}{Remark}%
\makeautorefname{remarks}{Remarks}%
\makeautorefname{rems}{Remarks}%
\makeautorefname{rmks}{Remarks}%
\makeautorefname{rks}{Remarks}%
\makeautorefname{example}{Example}%
\makeautorefname{examp}{Example}%
\makeautorefname{exmp}{Example}%
\makeautorefname{exam}{Example}%
\makeautorefname{exa}{Example}%
\makeautorefname{algorithm}{Algorithm}%
\makeautorefname{algo}{Algorithm}%
\makeautorefname{alg}{Algorithm}%
\makeautorefname{axiom}{Axiom}%
\makeautorefname{axi}{Axiom}%
\makeautorefname{ax}{Axiom}%
\makeautorefname{case}{Case}%
\makeautorefname{claim}{Claim}%
\makeautorefname{clm}{Claim}%
\makeautorefname{assumption}{Assumption}%
\makeautorefname{assumpt}{Assumption}%
\makeautorefname{conclusion}{Conclusion}%
\makeautorefname{concl}{Conclusion}%
\makeautorefname{conc}{Conclusion}%
\makeautorefname{condition}{Condition}%
\makeautorefname{condit}{Condition}%
\makeautorefname{cond}{Condition}%
\makeautorefname{construction}{Construction}%
\makeautorefname{construct}{Construction}%
\makeautorefname{const}{Construction}%
\makeautorefname{cons}{Construction}%
\makeautorefname{criterion}{Criterion}%
\makeautorefname{criter}{Criterion}%
\makeautorefname{crit}{Criterion}%
\makeautorefname{exercise}{Exercise}%
\makeautorefname{exer}{Exercise}%
\makeautorefname{exe}{Exercise}%
\makeautorefname{problem}{Problem}%
\makeautorefname{problm}{Problem}%
\makeautorefname{probm}{Problem}%
\makeautorefname{prob}{Problem}%
\makeautorefname{solution}{Solution}%
\makeautorefname{soln}{Solution}%
\makeautorefname{sol}{Solution}%
\makeautorefname{summary}{Summary}%
\makeautorefname{summ}{Summary}%
\makeautorefname{sum}{Summary}%
\makeautorefname{operation}{Operation}%
\makeautorefname{oper}{Operation}%
\makeautorefname{observation}{Observation}%
\makeautorefname{observn}{Observation}%
\makeautorefname{obser}{Observation}%
\makeautorefname{obs}{Observation}%
\makeautorefname{ob}{Observation}%
\makeautorefname{convention}{Convention}%
\makeautorefname{convent}{Convention}%
\makeautorefname{conv}{Convention}%
\makeautorefname{cvn}{Convention}%
\makeautorefname{warning}{Warning}%
\makeautorefname{warn}{Warning}%
\makeautorefname{note}{Note}%
\makeautorefname{fact}{Fact}%

\begin{document}
\vbadness=10001
\hbadness=10001
\overfullrule=1mm
\title[HOMFLYPT homology for links in handlebodies]{HOMFLYPT homology for links in handlebodies via type $\typeA$ Soergel bimodules}
\author[David E.V. Rose and Daniel Tubbenhauer]{David E.V. Rose and Daniel Tubbenhauer}

\address{D.E.V.R.: Department of Mathematics, North Carolina at Chapel Hill, Phillips Hall CB3250, UNC-CH, Chapel Hill, NC 27599-3250, United states, \href{http://davidev.web.unc.edu/}{davidev.web.unc.edu}}
\email{davidrose@unc.edu}

\address{D.T.: Institut f{\"u}r Mathematik, Universit{\"a}t Z{\"u}rich, Winterthurerstrasse 190, Campus Irchel, Office Y27J32, CH-8057 Z{\"u}rich, Switzerland, \href{www.dtubbenhauer.com}{www.dtubbenhauer.com}}
\email{daniel.tubbenhauer@math.uzh.ch}

\begin{abstract}
We define a triply-graded invariant of links in a genus $g$ handlebody, 
generalizing the colored HOMFLYPT (co)homology of links in the $3$-ball. 
Our main tools are the description of these links in terms of a 
subgroup of the classical braid group, 
and a family of categorical actions built from complexes of 
(singular) Soergel bimodules.
\end{abstract}

\maketitle

\tableofcontents

\renewcommand{\theequation}{\thesection-\arabic{equation}}
\addtocontents{toc}{\protect\setcounter{tocdepth}{1}}

\section{Introduction}\label{section:introduction}

The HOMFLYPT polynomial is a classical invariant of links $\link\subset\thesphere$ 
in the $3$-sphere $\thesphere$
with interesting and deep connections to representation theory.
As pioneered by Jones \cite{Jo-hecke-homfly}, 
the HOMFLYPT polynomial may be defined using representations of the 
classical $n$ strand braid group $\braidg[n]$ on the type $\typeA$ Hecke algebra.
Indeed, we may use Alexander's theorem to present a link as a braid closure, 
and the HOMFLYPT polynomial then results by mapping the braid to 
the Hecke algebra and 
applying the so-called Jones--Ocneanu trace.
(For the duration of the introduction, if not explicitly stated otherwise, 
``Hecke algebra'' and related notions are always of type $\typeA$.)

This approach to the HOMFLYPT polynomial was categorified in work of Khovanov \cite{Kh-homfly-soergel}. 
In this work, Khovanov shows that the triply-graded Khovanov--Rozansky homology 
$\cHHH{\link}{}$ of $\link\subset\thesphere$, 
originally defined in \cite{KhRo-link-homologies-2}, 
admits a construction paralleling Jones's
approach at the categorical level. 
This approach proceeds by replacing the Hecke algebra by the 
corresponding Hecke category, \ie the category of Soergel bimodules.
The latter admits a categorical action of $\braidg[n]$ via so-called Rouquier complexes \cite{Ro-cat-braid-group}, 
and the link homology results by taking Hochschild (co)homology, 
which provides a categorical analogue
of the Jones--Ocneanu trace.

In addition to their triply-graded invariant, 
for each $m\geq 2$
Khovanov and Rozansky define a doubly-graded homology theory for links 
$\link\subset\thesphere$ \cite{KhRo-link-homologies} 
that categorifies the $\slm$ specialization of the HOMFLYPT polynomial.
In the $m=2$ case, which coincides with Khovanov's categorification of the Jones polynomial 
\cite{Kh-cat-jones}, 
Asaeda--Przytycki--Sikora have extended this link homology to links in $3$-manifolds 
$\topstuff{M}\neq\thesphere$ \cite{AsPrSi-link-homology-annulus}, namely to links in thickened surfaces. 
Of particular interest is the case of the thickened annulus, 
where the so-called annular Khovanov homology has deep connections to both 
Floer theory  
and representation theory, see \eg \cite{Rob}, \cite{GW} and \cite{GrLiWe-annular-schur-weyl}.
In \cite{QuRo-sutured-annular-homology}, 
an analogue of doubly-graded Khovanov--Rozansky homology was constructed for annular links, 
extending annular Khovanov homology, and its connection to representation theory, to general $m$.
Unfortunately, the above approaches to link homology in $3$-manifolds $\topstuff{M}\neq\thesphere$
do not extend to the triply-graded setting.

In this paper, we remedy this by constructing generalizations of the triply-graded link homology for links 
in $3$-manifolds distinct from the $3$-sphere, namely in genus $g$ handlebodies.
(For $g=1$, this is the case of links in the thickened annulus.)
Our key insights are: (1) to consider various generalizations of the classical braid group that 
are related to links in handlebodies, 
and (2) that certain structures in categorical representation theory model the topology of the handlebody.
We now detail our approach.

\subsection{An overview of our construction}\label{subsection:intro-b}

Throughout, we let $g,n\in\N$.
Recall that Khovanov's construction of $\cHHH{\link}{}$ for $\link\subset\thesphere$ requires the following.
\begin{enumerate}[label=$\bullet$]

\setlength\itemsep{.15cm}

\item Alexander's Theorem, which states that, up to isotopy, any link 
$\link\subset\thesphere$ 
can be presented as the closure of a braid $\braid$ in the classical 
$n$-strand braid group $\braidg[n]$.

\item Markov's Theorem, which gives necessary and sufficient conditions 
for two distinct braids to have isotopic closures.

\item A categorical action of $\braidg[n]$ on the Hecke category via Rouquier complexes, 
which allows for the assignment of a chain complex of Soergel bimodules to each $\braid\in\braidg[n]$.

\item Hochschild (co)homology, which produces a 
Markov invariant triply-graded vector space from this complex of 
Soergel bimodules.
\end{enumerate}

In \cite{HaOlLa-handlebodies} (see also \cite{La-typeB-braids}), it is shown that analogues of Alexander's 
and Markov's Theorems hold for links in the genus $g$ handlebody $\hbody$.
Playing the role of the classical braid group is the $n$-strand braid group $\braidg[g,n]$ 
of the $g$-times punctured disk $\thedisk_{g}$. 
The classical story here is the $g=0$ case, where $\braidg[n]=\braidg[0,n]$.

As we more fully detail in \fullref{subsec:links-alexander}, 
braids in $\braidg[g,n]$ can be pictured as classical braids in the presence of 
non-intersecting ``core strands''.
We then obtain a link in $\hbody$ by allowing the tops and bottoms of the core strands to 
meet at $\infty$, and by taking a closure of the ``usual strands''.
The latter then form a link in the complement of the (glued) core strands, 
which is a handlebody $\hbody$:
\begin{gather}\label{eq:braidexample}
\braid=
\begin{tikzpicture}[anchorbase,scale=.5,tinynodes]
	\draw[pole,crosspole] (0,0) to[out=90, in=270] (0,1.5);
	\draw[pole,crosspole] (1,0) to[out=90, in=270] (1,1.5);
	\draw[usual,crossline] (3,0) to[out=90, in=0] (2.75,.5) 
	to[out=180, in=0] (1.75,.5) to[out=180, in=270] (1.5,.75);
	\draw[pole,crosspole] (2,0) to[out=90, in=270] (2,1.5);
	\draw[usual,crossline] (1.5,.75) to[out=90, in=180] (1.75,1) 
	to[out=0, in=180] (2.75,1) to[out=0, in=270] (3,1.5);
	\draw[pole,crosspole] (0,1.5) to[out=90, in=270] (0,3);
	\draw[usual,crossline] (3,1.5) to[out=90, in=0] (2.75,2) 
	to[out=180, in=0] (.75,2) to[out=180, in=270] (.5,2.25);
	\draw[pole,crosspole] (1,1.5) to[out=90, in=270] (1,3);
	\draw[pole,crosspole] (2,1.5) to[out=90, in=270] (2,3);
	\draw[usual,crossline,directed=.99] (.5,2.25) 
	to[out=90, in=180] (.75,2.5) to[out=0, in=180] (2.75,2.5) to[out=0, in=270] (3,3);
	\draw[infinity,->] (1,-.6) to (0,-.35) to (0,-.1);
	\draw[infinity,->] (1,-.6) node[below,infinity]{core strands} to (1,-.35) to (1,-.1);
	\draw[infinity,->] (1,-.6) to (2,-.35) to (2,-.1);
	\draw[black,->] (1,3.6) node[above]{usual strand} to (1,3.35) to (3,3.35) to (3,3.1);
\end{tikzpicture}
\subset
\thedisk_{g}\times[0,1]
\xrightarrow{\text{closure}}
\begin{tikzpicture}[anchorbase,scale=.5,tinynodes]
	\draw[pole,crosspole] (0,0) to[out=90, in=270] (0,1.5);
	\draw[pole,crosspole] (1,0) to[out=90, in=270] (1,1.5);
	\draw[usual,crossline] (3,0) to[out=90, in=0] (2.75,.5) to[out=180, in=0] (1.75,.5) to[out=180, in=270] (1.5,.75);
	\draw[pole,crosspole] (2,0) to[out=90, in=270] (2,1.5);
	\draw[usual,crossline] (1.5,.75) to[out=90, in=180] (1.75,1) to[out=0, in=180] (2.75,1) to[out=0, in=270] (3,1.5);
	\draw[pole,crosspole] (0,1.5) to[out=90, in=270] (0,3);
	\draw[usual,crossline] (3,1.5) to[out=90, in=0] (2.75,2) to[out=180, in=0] (.75,2) to[out=180, in=270] (.5,2.25);
	\draw[pole,crosspole] (1,1.5) to[out=90, in=270] (1,3);
	\draw[pole,crosspole] (2,1.5) to[out=90, in=270] (2,3);
	\draw[usual,crossline,directed=.99] (.5,2.25) to[out=90, in=180] (.75,2.5) to[out=0, in=180] (2.75,2.5) to[out=0, in=270] (3,3);
	\draw[cpole] (0,0) to[out=270, in=180] (1,-.75);
	\draw[cpole] (1,0) to[out=270, in=90] (1,-.75);
	\draw[cpole] (2,0) to[out=270, in=0] (1,-.75);
	\draw[cpole] (0,3) to[out=90, in=180] (1,3.75);
	\draw[cpole] (1,3) to[out=90, in=270] (1,3.75);
	\draw[cpole] (2,3) to[out=90, in=0] (1,3.75);
	\draw node[pole] at (1,3.5) {{\color{infinity}\LARGE$\bullet$}};
	\draw node[pole, above] at (1,3.7) {\color{infinity}$\infty$};
	\draw node[pole] at (1,-1) {{\color{infinity}\LARGE$\bullet$}};
	\draw node[pole, below] at (1,-.8) {\color{infinity}$\infty$};
	\draw[cusual] (3,0) to[out=270, in=180] (3.5,-.5) 
	to[out=0, in=270] (4,0) to (4,3) to[out=90, in=0] (3.5,3.5) to[out=180, in=90] (3,3);
\end{tikzpicture}
\end{gather}
The analogue of the Alexander Theorem here shows that, up to isotopy, every link in $\hbody$ arises in this way, 
and the corresponding Markov Theorem characterizes when distinct braids give rise to isotopic closures.

Issues arise, however, when attempting to carry out the last two steps in the construction of 
triply-graded link homology in this setting. 
Indeed, for general $g$, the groups $\braidg[g,n]$ are not 
known to be 
Artin--Tits groups 
(see \fullref{subsection:intro-digression} for further discussion), 
so, to our knowledge, 
there are no associated Soergel bimodules and/or Rouquier complexes.
Further, the Markov Theorem has a weaker notion of conjugation than in the classical case, 
\eg we have
\begin{gather}\label{eq:not-equal}
\begin{tikzpicture}[anchorbase,scale=.5,tinynodes]
	\draw[pole,crosspole] (1,0) to[out=90, in=270] (1,1.5);
	\draw[usual,crossline] (3,0) to[out=90, in=0] (2.75,.5) 
	to[out=180, in=0] (1.75,.5) to[out=180, in=270] (1.5,.75);
	\draw[pole,crosspole] (2,0) to[out=90, in=270] (2,1.5);
	\draw[usual,crossline] (1.5,.75) to[out=90, in=180] (1.75,1) 
	to[out=0, in=180] (2.75,1) to[out=0, in=270] (3,1.5);
	\draw[pole,crosspole] (2,1.5) to[out=90, in=270] (2,3);
	\draw[usual,crossline] (3,1.5) to[out=90, in=0] (2.75,2) 
	to[out=180, in=0] (.75,2) to[out=180, in=270] (.5,2.25);
	\draw[pole,crosspole] (1,1.5) to[out=90, in=270] (1,3);
	\draw[usual,crossline,directed=1] (.5,2.25) to[out=90, in=180] (.75,2.5) 
	to[out=0, in=180] (2.75,2.5) to[out=0, in=270] (3,3);
	\draw[cpole] (1,0) to[out=270, in=180] (1.5,-.5);
	\draw[cpole] (2,0) to[out=270, in=0] (1.5,-.5);
	\draw[cpole] (1,3) to[out=90, in=180] (1.5,3.5);
	\draw[cpole] (2,3) to[out=90, in=0] (1.5,3.5);
	\draw node[pole] at (1.5,-.7) {{\color{infinity}\LARGE$\bullet$}};
	\draw node[pole, below] at (1.5,-.5) {\color{infinity}$\infty$};
	\draw node[pole] at (1.5,3.3) {{\color{infinity}\LARGE$\bullet$}};
	\draw node[pole, above] at (1.5,3.5) {\color{infinity}$\infty$};
	\draw[thick,mygray,fill=citron,fill opacity=.1] (0,3) 
	to (3.5,3) to (3.5,0) to (0,0) to (0,3);
	\draw[cusual] (3,0) to[out=270, in=180] (3.5,-.5) 
	to[out=0, in=270] (4,0) to (4,3) to[out=90, in=0] (3.5,3.5) to[out=180, in=90] (3,3);
	\node at (.5,1.5) {$\braid$};
\end{tikzpicture}
\xleftrightarrow{\text{not isotopic}}
\begin{tikzpicture}[anchorbase,scale=.5,tinynodes]
	\draw[pole,crosspole] (2,0) to[out=90, in=270] (2,1.5);
	\draw[usual,crossline] (3,0) to[out=90, in=0] (2.75,.5) 
	to[out=180, in=0] (.75,.5) to[out=180, in=270] (.5,.75);
	\draw[pole,crosspole] (1,0) to[out=90, in=270] (1,1.5);
	\draw[usual,crossline] (.5,.75) to[out=90, in=180] (.75,1) 
	to[out=0, in=180] (2.75,1) to[out=0, in=270] (3,1.5);
	\draw[pole,crosspole] (1,1.5) to[out=90, in=270] (1,3);
	\draw[usual,crossline] (3,1.5) to[out=90, in=0] (2.75,2) 
	to[out=180, in=0] (1.75,2) to[out=180, in=270] (1.5,2.25);
	\draw[pole,crosspole] (2,1.5) to[out=90, in=270] (2,3);
	\draw[usual,crossline,directed=1] (1.5,2.25) 
	to[out=90, in=180] (1.75,2.5) to[out=0, in=180] (2.75,2.5) to[out=0, in=270] (3,3);
	\draw[cpole] (1,0) to[out=270, in=180] (1.5,-.5);
	\draw[cpole] (2,0) to[out=270, in=0] (1.5,-.5);
	\draw[cpole] (1,3) to[out=90, in=180] (1.5,3.5);
	\draw[cpole] (2,3) to[out=90, in=0] (1.5,3.5);
	\draw node[pole] at (1.5,-.7) {{\color{infinity}\LARGE$\bullet$}};
	\draw node[pole, below] at (1.5,-.5) {\color{infinity}$\infty$};
	\draw node[pole] at (1.5,3.3) {{\color{infinity}\LARGE$\bullet$}};
	\draw node[pole, above] at (1.5,3.5) {\color{infinity}$\infty$};
	\draw[thick,mygray,fill=citron,fill opacity=.1] (0,3) 
	to (3.5,3) to (3.5,0) to (0,0) to (0,3);
	\draw[cusual] (3,0) to[out=270, in=180] (3.5,-.5) 
	to[out=0, in=270] (4,0) to (4,3) to[out=90, in=0] (3.5,3.5) to[out=180, in=90] (3,3);
	\node at (.5,1.5) {$\braid[c]$};
\end{tikzpicture}
\end{gather}
even though the indicated (boxed) braids $\braid,\braid[c]\in\braidg[2,1]$ are conjugate. 
Hence, even with a categorical representation of $\braidg[g,n]$ in hand, 
one cannot simply apply Hochschild cohomology to obtain an invariant of links 
that is sensitive to the topology of $\hbody$.

We simultaneously resolve both these problems as follows. 
We expand the point at infinity to a small segment, 
which we move close to the top of the core strands. 
As a result, we can view the closure of the ``usual strands'' 
as a link in the handlebody given by the complement 
of the graph determined by the core strands and the segment at infinity, \eg
\begin{gather}\label{eq:infinitysegment}
\begin{tikzpicture}[anchorbase,scale=.5,tinynodes]
	\draw[pole,crosspole] (2,0) to[out=90, in=270] (2,1.5);
	\draw[usual,crossline] (3,0) to[out=90, in=0] (2.75,.5) 
	to[out=180, in=0] (.75,.5) to[out=180, in=270] (.5,.75);
	\draw[pole,crosspole] (1,0) to[out=90, in=270] (1,1.5);
	\draw[usual,crossline] (.5,.75) to[out=90, in=180] (.75,1) 
	to[out=0, in=180] (2.75,1) to[out=0, in=270] (3,1.5);
	\draw[pole,crosspole] (1,1.5) to[out=90, in=270] (1,3);
	\draw[usual,crossline] (3,1.5) to[out=90, in=0] (2.75,2) 
	to[out=180, in=0] (1.75,2) to[out=180, in=270] (1.5,2.25);
	\draw[pole,crosspole] (2,1.5) to[out=90, in=270] (2,3);
	\draw[usual,crossline,directed=1] (1.5,2.25) 
	to[out=90, in=180] (1.75,2.5) to[out=0, in=180] (2.75,2.5) to[out=0, in=270] (3,3);
	\draw[cpole] (1,0) to[out=270, in=180] (1.5,-.5);
	\draw[cpole] (2,0) to[out=270, in=0] (1.5,-.5);
	\draw[cpole] (1,3) to[out=90, in=180] (1.5,3.5);
	\draw[cpole] (2,3) to[out=90, in=0] (1.5,3.5);
	\draw node[pole] at (1.5,-.7) {{\color{infinity}\LARGE$\bullet$}};
	\draw node[pole, below] at (1.5,-.5) {\color{infinity}$\infty$};
	\draw node[pole] at (1.5,3.3) {{\color{infinity}\LARGE$\bullet$}};
	\draw node[pole, above] at (1.5,3.5) {\color{infinity}$\infty$};
	\draw[cusual] (3,0) to[out=270, in=180] (3.5,-.5) 
	to[out=0, in=270] (4,0) to (4,3) to[out=90, in=0] (3.5,3.5) to[out=180, in=90] (3,3);
\end{tikzpicture}
\ouriso[\sim]
\begin{tikzpicture}[anchorbase,scale=.5,tinynodes]
	\draw[pole,crosspole] (2,0) to[out=90, in=270] (2,1.5);
	\draw[usual,crossline] (3,0) to[out=90, in=0] (2.75,.5) 
	to[out=180, in=0] (.75,.5) to[out=180, in=270] (.5,.75);
	\draw[pole,crosspole] (1,0) to[out=90, in=270] (1,1.5);
	\draw[usual,crossline] (.5,.75) to[out=90, in=180] (.75,1) 
	to[out=0, in=180] (2.75,1) to[out=0, in=270] (3,1.5);
	\draw[pole,crosspole] (1,1.5) to[out=90, in=270] (1,3);
	\draw[usual,crossline] (3,1.5) to[out=90, in=0] (2.75,2) 
	to[out=180, in=0] (1.75,2) to[out=180, in=270] (1.5,2.25);
	\draw[pole,crosspole] (2,1.5) to[out=90, in=270] (2,3);
	\draw[usual,crossline,directed=1] (1.5,2.25) 
	to[out=90, in=180] (1.75,2.5) to[out=0, in=180] (2.75,2.5) to[out=0, in=270] (3,3);
	\draw[cpole] (1,3) to[out=90, in=180] (1.5,3.5);
	\draw[cpole] (2,3) to[out=90, in=0] (1.5,3.5);
	\draw[cpole] (1.5,3.5) to (1.5,4);
	\draw[cpole] (1,4.5) to[out=270, in=180] (1.5,4);
	\draw[cpole] (2,4.5) to[out=270, in=0] (1.5,4);
	\draw[cpole] (2,4.5) to [out=90,in=180] (3.5,5.5) to [out=0,in=90] (5,4.5) to (5,0)
	to [out=270,in=0] (3.5,-1) to [out=180,in=270] (2,0);
	\draw[cpole] (1,4.5) to [out=90,in=180] (3.5,6.5) to [out=0,in=90] (6,4.5) to (6,0)
	to [out=270,in=0] (3.5,-2) to [out=180,in=270] (1,0);
	\draw[thick,mygray,fill=citron,fill opacity=.1] (0,4.5) 
	to (3.5,4.5) to (3.5,0) to (0,0) to (0,4.5);
	\draw[cusual] (3,0) to[out=270, in=180] (3.5,-.5) 
	to[out=0, in=270] (4,0) to (4,4.5) to[out=90, in=0] (3.5,5) to[out=180, in=90] (3,4.5) to (3,3);
\end{tikzpicture}
\end{gather}
In this modified presentation we are able to assign an invariant to the link $\link\subset\hbody$ 
using known structures in categorical representation theory.
Indeed, for any labeling of the core strands, 
the boxed diagram in \eqref{eq:infinitysegment} determines a complex of 
singular Soergel bimodules. 
The latter determine a ($2$-)category that contains the Hecke category \cite{Wi-sing-soergel}, 
and categorifies the Schur algebroid, a certain idempotent completion of the Hecke algebra. 
Further, the closure procedure now does not involve the point at infinity, 
and hence can be carried out algebraically as usual, using Hochschild cohomology. 
In this way, we obtain a triply-graded homology for links $\link\subset\hbody$.
We show that this indeed produces a well-defined invariant of 
handlebody links, and that it is sensitive to the topology of the handlebody, 
\eg it distinguishes the links in \eqref{eq:not-equal}.

\subsection{A digression on Artin--Tits groups}\label{subsection:intro-digression}

Our motivation for this project, from which we have now somewhat strayed, 
was to further our understanding of the connection between low-dimensional topology and Artin--Tits groups. 
Recall that a  Coxeter diagram $\coxdia=(\setstuff{V},\setstuff{E})$ consists of a 
simple, complete graph with finitely many vertices $\setstuff{V}$
whose edges $e=(i,j)\in\setstuff{E}$ carry a label $m_{ij}=m_{ji}\in\N[\geq 2]\cup\{\infty\}$. 
To any such diagram, we may associate the Artin--Tits group:
\begin{gather}\label{eq:artintits-braids}
\artintits
:=
\big\langle\atgen,\,i\in\setstuff{V}\mid
\underbrace{\dots\atgen\atgen[j]\atgen}_{m_{ij}\text{ factors}}
=
\underbrace{\dots\atgen[j]\atgen\atgen[j]}_{m_{ij}\text{ factors}}
\big\rangle.
\end{gather}
This group is an extension of the corresponding Coxeter group:
\begin{gather}\label{eq:coxeter}
\coxgroup
:=
\big\langle
\coxgen,\,i\in\setstuff{V}\mid\coxgen^{2}=1,\,
\underbrace{\dots\coxgen\coxgen[j]\coxgen}_{m_{ij}\text{ factors}}
=
\underbrace{\dots\coxgen[j]\coxgen\coxgen[j]}_{m_{ij}\text{ factors}}
\big\rangle.
\end{gather}
The jumping-off point is the classical observation that $\braidg[n]=\braidg[0,n]$ 
is isomorphic to the Artin--Tits braid group of type $\typeA$, 
while $\braidg[1,n]$ is isomorphic to the Artin--Tits group of type $\typeC=\typeB$ 
and extended affine type $\typeA$. 
More-surprising is the lesser-known fact that 
$\braidg[2,n]$ is isomorphic to the Artin--Tits group of affine type $\typeC$ 
\cite{Al-braids-abcd}. 
The following table summarizes these known connections, 
details of which can be found in 
\eg \cite[Section 4]{Al-braids-abcd} and \cite{Br-bourbaki-lectures}.
\begin{gather}\label{eq:table}
\renewcommand{\arraystretch}{1.25}
\begin{tabular}{c||c|c}
Genus & type $\typeA$ & type $\typeC$\\ 
\hline
\hline
$g=0$ &  $\braidg[n]\ouriso\artintits[{\typeA[n{-}1]}]$ & $\boldsymbol{?}$ \\ 
\hline 
$g=1$ &  $\braidg[1,n]\ouriso\Z\ltimes\artintits[{\atypeA[n{-}1]}]\ouriso\artintits[{\etypeA[n{-}1]}]$ 
&  $\braidg[1,n]\ouriso\artintits[{\typeC[n]}]$  \\
\hline 
$g=2$ & $\boldsymbol{?}$ &  $\braidg[2,n]\ouriso\artintits[{\atypeC[n]}]$ \\
\hline  
$g\geq 3$ & $\boldsymbol{?}$ & $\boldsymbol{?}$ \\
\end{tabular}
\end{gather}
Herein, $\typeA[n{-}1]$ denotes the type $\typeA$ Coxeter diagram with $n-1$ nodes, 
while $\atypeA[n{-}1]$ denotes the affine type $\typeA$ Coxeter diagram with $n$ nodes and 
and $\etypeA[n{-}1]$ is the corresponding extended affine type.
Similarly, $\typeC[n]$ and $\atypeC[n]$ denote the type $\typeC=\typeB$ and 
affine type $\typeC$ (but not affine type $\typeB$) Coxeter diagrams
with $n$ and $n+1$ nodes, respectively.

As mentioned above, the first row of the type $\typeA$ column in \eqref{eq:table} underpins 
Jones's construction of the HOMFLYPT polynomial, 
and the second row has similarly been exploited in topological considerations,
see \eg \cite{OrRa-affine-braids} and \cite{El-gaitsgory-sheaves}.
The type $\typeC$ column, however, has received less attention,
especially in the affine, $g=2$ case, 
where not much appears to be known about 
connections to link invariants.
(However, this case has been explored from a representation-theoretic point of view, 
see \eg \cite{DaRa-two-boundary-hecke}.)
A notable example is work of Geck--Lambropoulou \cite{GeLa-markov-typeB} in the $g=1$ case, 
where a HOMFLYPT polynomial for links in $\hbody[1]$ is constructed via 
the analogue of Jones's construction in type $\typeC$. 
The results in \cite{Ro-homflypt} and \cite{WeWi-markov-geometry} should pair to give a 
categorification of this invariant.
In a companion paper \cite{RoTu-homflypt-typec}, 
we plan to study this invariant, and develop its genus two analogue, 
using type $\typeC$ and affine type $\typeC$ Hecke algebras and Soergel bimodules.

By contrast, our construction in the present paper exploits the relation 
between $\braidg[g,n]$ and $\braidg[g+n]$, and the fact that the latter is an Artin--Tits group. 
Indeed, our construction can be recast as follows.
By viewing the distinguished strands as usual strands, 
we obtain an injective group homomorphism $\braidg[g,n]\hookrightarrow\braidg[g+n]$. 
Since the latter is an Artin--Tits group, we can assign a complex of Soergel bimodules to any 
braid $\braid\in\braidg[g,n]$. 
Now, before taking Hochschild cohomology 
(doing so immediately would give an invariant less-sensitive to the topology of the handlebody),
we glue on an additional Soergel bimodule that allows invariance under the Markov Theorem for 
$\braidg[g,n]$, but not for $\braidg[g+n]$.
In fact, our procedure is slightly more general in that it uses an embedding of $\braidg[g,n]$ 
into the colored braid group, and singular Soergel bimodules.

\subsection{Future outlook}\label{subsection:intro-outlook}

In addition to our planned investigation in type $\typeC$ \cite{RoTu-homflypt-typec}, 
we believe there are a number of interesting future directions.

\begin{enumerate}[label=$\bullet$]

\setlength\itemsep{.15cm}

\item \textbf{The relation between $\braidg[g,n]$ and Hecke algebras.} 
These exists a Hecke-like algebra associated to $\braidg[g,n]$ for general $g$, see \eg \cite{La-handlebodies}. 
In the $g=0,1$ cases, this algebra matches the Hecke algebras associated to the 
Artin--Tits groups in the type $\typeA$ column of \eqref{eq:table}.
These algebras have not been widely studied, \eg to our knowledge it is not known 
whether they admit Markov traces or categorifications.

In another direction, 
it is an interesting problem to extend the type $\typeC$ column of \eqref{eq:table} 
to higher genus.  
The presentation of $\braidg[g,n]$ given below in \fullref{definition:alg-braid-group} 
hints to a connection to the Artin--Tits group associated to the Coxeter diagram that is 
obtained from the type $\typeA[n]$ diagram by adjoining $g$ additional vertices. 
These vertices are attached to each other with $\infty$-labeled edges, 
and to the first type $\typeA$ vertex with $4$-labeled edges.
(Something very similar was also observed in \cite[Remark 4]{La-handlebodies}.)
For example, the $g=3$ case is as follows:
\begin{gather}
\raisebox{-.09cm}{$\text{case $g=3$}\colon$}
\begin{tikzcd}[ampersand replacement=\&,row sep=.3cm,column sep=.3cm,
arrows={shorten >=-.5ex,shorten <=-.5ex},labels={inner sep=.5ex}]
\parbox[c][.3cm][c]{.4cm}{$\elstuff{0}^{1}$}
\arrow[ultra thick,xshift=.025cm,yshift=.025cm,-]{rd}{}\arrow[ultra thick,xshift=-.025cm,yshift=-.025cm,-]{rd}{}
\arrow[ultra thick,densely dotted,out=210,in=150,-,swap]{dd}{\infty} 
\& \& \phantom{.} \phantom{.} \& \phantom{.} \& \phantom{.} \& \phantom{.}
\\
\parbox[c][.3cm][c]{.4cm}{$\elstuff{0}^{2}$}\arrow[ultra thick,yshift=.04cm,-]{r}{}\arrow[ultra thick,yshift=-.04cm,-]{r}{}
\arrow[ultra thick,densely dotted,-]{u}{\infty} \arrow[ultra thick,densely dotted,-,swap]{d}{\infty} 
\& \elstuff{1}\arrow[ultra thick,-]{r}{} \& \elstuff{2}\arrow[ultra thick,-]{r}{} 
\& \cdots\arrow[ultra thick,-]{r}{} \& \elstuff{n{-}1} 
\\
\parbox[c][.3cm][c]{.4cm}{$\elstuff{0}^{3}$}
\arrow[ultra thick,xshift=-.025cm,yshift=.025cm,-]{ru}{}
\arrow[ultra thick,xshift=.025cm,yshift=-.025cm,-]{ru}{} \& \phantom{.} \& \phantom{.} \& \phantom{.} \& \phantom{.} \& \phantom{.}
\end{tikzcd}
\end{gather}
Here, we depict $k$-labeled edges (for $k<\infty$) as $k-2$ unlabeled edges.
In fact, $\braidg[g,n]$ is a quotient of the associated Artin--Tits group, 
so one could hope to extract invariants of $\link\subset\hbody$ from 
(a suitable quotient of)
the corresponding Hecke algebra and/or Soergel bimodules.

\item \textbf{Connections to algebraic geometry.}
Work of Webster--Williamson \cite{WeWi-markov-geometry}
relates the Jones--Ocneanu trace on the type $\typeA$ Hecke algebra to the 
equivariant cohomology of sheaves on $\mathrm{SL}_{n}$, and extends this to other types. 
It would be interesting to identify geometry related to $\braidg[g,n]$ 
and, more generally, links in $\hbody$. 
One fertile avenue is the possible connection between the $g=2$ case and 
exotic Springer fibers as \eg in \cite{SaWi-exotic-springer-cups}.

In a different direction, 
work of Gorsky--Negut--Rasmussen \cite{GNR} conjectures a relation between
the category of type $\typeA$ Soergel bimodules and the flag Hilbert scheme of $\C^{2}$. 
The appearance of the latter can be interpreted as considering the closure of a braid 
$\braid\in\braidg[n]$ in the complement of an $n$-component unlink.
Since the graph giving the complement of $\hbody$ can be viewed as an unlink 
fused at the ``segment at infinity,'' this suggests a connection between the flag Hilbert scheme 
and our invariants.

Finally, another avenue of exploration is to extend various (known or conjectural) 
physical predications concerning ($g=0$) triply-graded homology to higher genus, 
see \eg \cite{GoGuSt-quadruply-graded-homfly} or \cite{GuSt-homological-bps}, 
and \cite[Section 6.3]{QuRoSa-annular-evaluation} or \cite{TuVaWe-super-howe} for related results. 
\end{enumerate}

\subsection{Conventions}\label{subsection:intro-conventions}

We now summarize various conventions used in this paper.

\begin{convention}
We work over an arbitrary field $\K$ 
of characteristic $0$. 
This requirement is only needed in \fullref{section:homology}: 
the reader interested in integral versions of our results from 
\fullref{section:genusg} needs to replace the algebraic definition of singular Soergel bimodules of type $\typeA$, 
which we use, by their diagrammatic incarnation \cite[Section 2.5]{ElLo-modular-rep-typea}. 
(The algebraic and the diagrammatic definitions differ when working 
integrally or in characteristic $p$.)
All the results from \fullref{section:genusg} then hold \ver over $\Z$.
However, we do not currently have integral versions of the singular Soergel diagrammatics in \fullref{section:homology}.
\end{convention}

\begin{convention}\label{convention:diagram-conventions}
We will find it convenient to depict morphisms in certain categories 
(and $1$-morphisms in certain $2$-categories) diagrammatically.
We will read such diagrams from bottom-to-top 
(and in the presence of a monoidal 
or $2$-categorical structure, also right-to-left).
These reading conventions are
summarized by
\begin{gather}
\begin{tikzpicture}[anchorbase,scale=.5, tinynodes]
	\draw[usual] (-1,0) to[out=90,in=270] (-1,.5);
	\draw[usual] (0,0) to[out=90,in=270] (0,.5);
	\draw[usual] (1,0) to[out=90,in=270] (1,.5);
	\draw[usual] (2,0) to[out=90,in=270] (2,.5);
	\draw[usual] (-1,1) to[out=90,in=270] (-1,2);
	\draw[usual] (0,1) to[out=90,in=270] (0,2);
	\draw[usual] (1,1) to[out=90,in=270] (1,2);
	\draw[usual] (2,1) to[out=90,in=270] (2,2);
	\draw[usual] (-1,2.5) to[out=90,in=270] (-1,3);
	\draw[usual] (0,2.5) to[out=90,in=270] (0,3);
	\draw[usual] (1,2.5) to[out=90,in=270] (1,3);
	\draw[usual] (2,2.5) to[out=90,in=270] (2,3);
	\draw[thin,black,fill=mygray,fill opacity=.35] (-1.25,.5) rectangle (.25,1);
	\draw[thin,black,fill=mygray,fill opacity=.35] (.75,.5) rectangle (2.25,1);
	\draw[thin,black,fill=mygray,fill opacity=.35] (-1.25,2) rectangle (.25,2.5);
	\draw[thin,black,fill=mygray,fill opacity=.35] (.75,2) rectangle (2.25,2.5);
	\draw[thin,black,densely dotted] (-1.25,0) node[left]{$\obstuff{A}$} to (2.25,0);
	\draw[thin,black,densely dotted] (-1.25,1.5) node[left]{$\obstuff{B}$} to (2.25,1.5);
	\draw[thin,black,densely dotted] (-1.25,3) node[left]{$\obstuff{C}$} to (2.25,3);
	\node at (-.5,.7) {$\morstuff{a}$};
	\node at (-.5,2.175) {$\morstuff{b}$};
	\node at (1.5,.7) {$\morstuff{c}$};
	\node at (1.5,2.175) {$\morstuff{d}$};
\end{tikzpicture}
\leftrightsquigarrow
(\morstuff{b}\morstuff{d})\circ(\morstuff{a}\morstuff{c})\colon\obstuff{A}\to\obstuff{B}\to\obstuff{C}.
\end{gather}
Moreover, 
all such diagrams are invariant under (distant) height exchange isotopy 
(up to isomorphism, in the $2$-categorical context). 
Finally, we will occasionally omit certain data (\eg labelings) from such diagrams 
when it may be
recovered from the given data, or is not important for the argument in question.
\end{convention}

\begin{convention}\label{convention:grading}
We will work with $\Z^{k}$-graded categories throughout, for $k=1,2,3$. 
The three gradings of importance are
the internal degree $\qpar$, the homological degree $\tpar$ 
(both appearing from \fullref{section:genusg} onward), 
and the Hochschild degree $\apar$ (making its appearance in \fullref{section:homology}).

There are competing notions of what is meant by a graded category, 
so we now detail our conventions, focusing on the $\qpar$-degree.
Let $\catstuff{C}$ be a category enriched in $\Z$-graded abelian groups, 
\ie for objects $\obstuff{X}$ and $\obstuff{Y}$, $\Hom_{\catstuff{C}}(\obstuff{X},\obstuff{Y})$ 
is a $\Z$-graded abelian group:
\begin{gather}
\Hom_{\catstuff{C}}(\obstuff{X},\obstuff{Y}) 
= 
\bigoplus_{d\in\Z} 
\Hom_{\catstuff{C}}(\obstuff{X},\obstuff{Y})_{d}
\end{gather}
Given such a category, we can introduce a formal grading-shift functor $\qpar$
and consider the category $\widetilde{\catstuff{C}^{\qpar}}$ 
in which objects are given by formal shifts 
$\qpar^{s}\obstuff{X}$ of objects in $\catstuff{C}$, 
and
\begin{gather}
\Hom_{\widetilde{\catstuff{C}^{\qpar}}}(\qpar^{s}\obstuff{X},\qpar^{t}\obstuff{Y}) 
=\bigoplus_{d\in\Z} 
\Hom_{\catstuff{C}}(\obstuff{X},\obstuff{Y})_{d+t-s}.
\end{gather}
\ie $\widetilde{\catstuff{C}^{\qpar}}$ is again enriched 
in $\Z$-graded abelian groups.
Finally, we let $\catstuff{C}^{\qpar}$ be 
the category with the same objects as $\widetilde{\catstuff{C}^{\qpar}}$, 
but where we restrict to $\qpar$-degree zero morphisms, \ie
\begin{gather}
\Hom_{\catstuff{C}^{\qpar}}(\qpar^{s}\obstuff{X},\qpar^{t}\obstuff{Y}) 
= \Hom_{\catstuff{C}}(\obstuff{X},\obstuff{Y})_{s-t}.
\end{gather}
Note that $\catstuff{C}^{\qpar}$ is not enriched in $\Z$-graded abelian groups, 
but is equipped with an autoequivalence shift functor $\qpar$. 
It is categories of this form that will be of primary interest in this work.

We note, however, that it is possible to recover the $\Z$-graded abelian group 
$\Hom_{\catstuff{C}}(\obstuff{X},\obstuff{Y})$ from the category $\catstuff{C}^{\qpar}$. 
Indeed, we can consider the $\Z$-graded abelian group
\begin{gather}\label{eq:HOMdef}
\HOM_{\catstuff{C}^\qpar}(\obstuff{X},\obstuff{Y})
:=
\bigoplus_{d}\Hom_{\catstuff{C}^{\qpar}}(\qpar^{d}\obstuff{X},\obstuff{Y})
\end{gather}
and we note that 
\begin{gather}
\HOM_{\catstuff{C}^{\qpar}}(\qpar^{s}\obstuff{X},\qpar^{t}\obstuff{Y}) 
=\qpar^{t\tm s}\HOM_{\catstuff{C}^{\qpar}}(\obstuff{X},\obstuff{Y}),
\end{gather}
where on the right-hand side the power $\qpar$ denotes 
a shift of the indicated $\Z$-graded abelian group.

Our consideration of $\Z^{2}$- and $\Z^{3}$-graded categories is analogous -- 
in these cases we have additional shift functors $\tpar$ and $\apar$, 
and we restrict to $\tpar$- and $\apar$-degree 
zero maps, unless otherwise indicated. 
However, we will 
reserve the capitalization notation $\HOM$ when considering ``graded $\Hom$s'' 
with respect to the $\qpar$-degree only.

Lastly, we note that these considerations carry over to 
$2$-categories as well, where the above 
applies to the $\Hom$-categories in our $2$-category, 
\ie to the $1$- and $2$-morphisms.
\end{convention}
\medskip

\noindent\textbf{Acknowledgments.} 
This project began during a visit of the first named author to the Universit{\"a}t Z{\"u}rich, 
and he thanks them for excellent working conditions.
It continued during visits of both authors to the Kavli Institute for Theoretical Physics 
for the program Quantum Knot Invariants and Supersymmetric Gauge Theories. 
As such, this research was supported in part by the 
National Science Foundation under Grant No. NSF PHY-1748958.
We thank them for the engaging and collaborative atmosphere.
Some of the ideas for this project came up during a discussion 
of the second author and Catharina Stroppel, and he thanks her 
for freely sharing her ideas about HOMFLYPT homology 
outside of type $\typeA$.

Moreover, D.E.V.R. thanks Andrea Appel, Ben Elias, and Matt Hogancamp for useful discussions, 
and D.T. thanks Ben Elias, Aaron Lauda, Anthony Licata, Catharina Stroppel, Emmanuel Wagner 
and Arik Wilbert for related discussions about the 
topology behind Artin--Tits groups, Mikhail Khovanov, Paul Wedrich and Oded Yacobi for 
helpful conversations and encouraging words, and also his office chair for supporting this project.

Finally, D.E.V.R. was partially supported by Simons Collaboration Grant 523992, 
and D.T. is grateful to NCCR SwissMAP for generous support that, in particular, 
partially financed the first author's visit to Z{\"u}rich.

\section{Links and braids in handlebodies}\label{section:links-handlebodies}

In this section we
collect result concerning links and braids in handlebodies.

\subsection{Topological recollections}\label{subsec:links-handlebodies-top}

Recall that a handlebody $\hbody$ of genus $g$ is the compact, 
orientable $3$-manifold with boundary obtained 
by attaching $g$ $1$-handles to the closed $3$-ball $\hbody[0]$.
An explicit model for $\hbody$, that we call the standard presentation, 
is given by the ``inside'' of a standardly embedded genus $g$ surface $\gsurface\subset\thesphere$ 
in the $3$-sphere $\thesphere$, 
\ie the $3$-manifold given by the union of $\gsurface$ with the component of 
$\thesphere\smallsetminus\gsurface$ that does not contain the point at infinity.

We will typically work with another presentation for $\hbody$, 
given by the closure in $\thesphere$ of the complement of an auxiliary handlebody $\hbody^{c}$.
We view $\hbody^{c}$ as consisting of $g$ parallel $1$-handles 
that are attached to the closure of a small neighborhood of $\infty\in\thesphere$.
See the gray
portion of the first figure in \fullref{example:closing} for the $g=3$ case.
In order to connect with the categorical representation theory used to produce our link invariant, 
we note that $\hbody^{c}$ is isotopic to a closed neighborhood of the embedded graph obtained by taking 
$g$ parallel edges, called the ``core strands'', each meeting a $(g+1)$-valent vertex, 
together with an additional ``edge at infinity'' joining the two vertices.
We will typically view the edge at infinity as being near the top of the core strands, 
hereby viewing $\hbody^{c}$ as being obtained from the core strands by first gluing on a graph with 
$g$ 1-valent vertices at its top and bottom (and two $(g+1)$-valent vertices) and then taking their closure.
See the gray portion of the second figure in \fullref{example:closing}.

We will refer to this presentation of $\hbody=\overline{\thesphere\smallsetminus\hbody^{c}}$ as the costandard presentation, 
and note that it contains a copy of the standard presentation, 
to which it is isotopic, given by intersecting with a 
closed $3$-ball that meets each core strand in a segment.

We consider oriented links $\link\subset\hbody$, which, 
in the costandard presentation, are given
by links in $\thesphere\smallsetminus\hbody^{c}$.
Equivalently, using the isotopy with the standard presentation, such links are given by 
links in the closed $3$-ball that avoid its intersection with the core strands.
Finally, two links in $\link,\link^{\prime}\subset\hbody$ are isotopic, 
denoted by $\link\ouriso[\sim]\link^{\prime}$, if and only 
if the corresponding links in $\thesphere$ (in the costandard presentation)
are isotopic through an isotopy that keeps $\hbody^{c}$ fixed pointwise.

\subsection{Alexander's theorem}\label{subsec:links-alexander}

There is a corresponding notion of $n$ strand braids in a genus $g$ handlebody. 
Strictly speaking, we define $\braidg[g,n]$ to be the braid group of the surface 
$\thedisk_{g}$ given as the complement of $g$ disjoint open disks in the closed disk $\thedisk$.
The standard presentation of $\hbody$ may be identified with the product $\thedisk_{g}\times[0,1]\subset\thesphere$, 
so braids in $\braidg[g,n]$ ``live in'' $\hbody$.
The group $\braidg[g,n]$, which we call the handlebody braid group, 
can be equivalently described as follows. 
There is a subgroup of the classical braid group $\braidg[g+n]$ on $g+n$ strands
consisting of braids that are pure on the first $g$ strands, 
and a homomorphism from this subgroup to the classical braid group $\braidg[g]$ 
given by forgetting the final $n$ strands.
The kernel of this homomorphism is precisely $\braidg[g,n]$.
(Informally, $\braidg[g,n]$ consists of braids on $g+n$ strands 
in which the first $g$ strands do not braid among themselves.)
Slightly abusing notation, we will again refer to the first $g$ strands 
of a braid in $\braidg[g,n]$ as core strands; 
they correspond to the core strands above as we now describe.

The handlebody braid group $\braidg[g,n]$ is related to links in $\hbody$ in a 
manner paralleling the relation between 
the classical braid group $\braidg[n]$ and links in $\thesphere$.
That is, given a braid $\braid\in\braidg[g,n]$, 
one obtains a link $\bclosure[\braid]\subset\hbody$
via a closure procedure as follows: 
the first $g$ strands in $\braid$ are joined at each of their ends to 
the point at infinity, and the remainder of the braid is closed as in the classical case.
In this way, we obtain a link $\bclosure[\braid]\subset\hbody$ where the closure of the 
last $n$ strands constitutes $\bclosure[\braid]$, 
and the first $g$ strands in $\braid$ become the core strands in $\hbody^c$.
As in our discussion of $\hbody^c$ above, 
we will typically work with an equivalent closure procedure, 
which again corresponds to expanding the point at infinity to an edge, 
and moving it near the top of the core strands.
Specifically, the closure procedure consists of merging the $g$ core strands to meet
the strand at infinity, then splitting the strand at infinity into $g$ strands, 
and finally taking the standard closure of all strands.

\begin{example}\label{example:lowgenus}
We have $\braidg[0,n]\ouriso\braidg[n]$, which corresponds to links in
the closed $3$-ball $\hbody[0]$; we call this the classical case. 
In genus one, $\braidg[1,n]$ consists of all braids 
in $\braidg[1+n]$ that are pure on the first strand, 
and $\hbody[1]$ is a solid torus.
\end{example}

\begin{example}\label{example:closing}
Here we illustrate the closure procedure for $\braid\in\braidg[3,4]$.
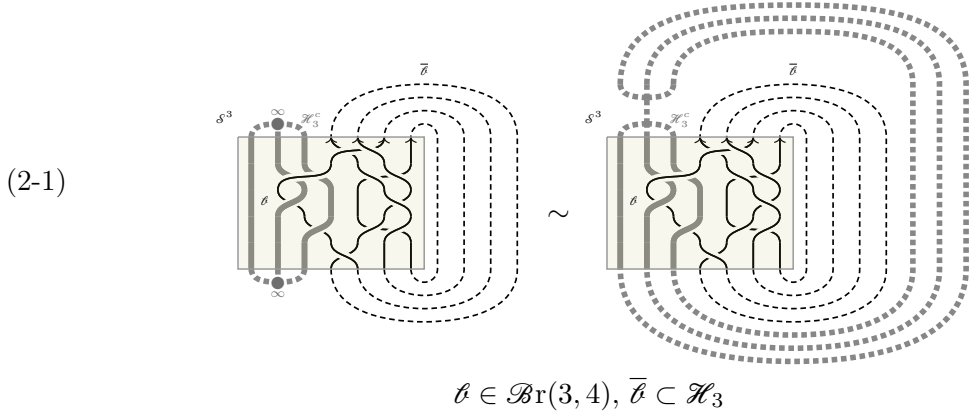
\begin{gather}
\xy
(0,0)*{
\scalebox{.7}{$
\begin{tikzpicture}[anchorbase,scale=.5,tinynodes]
	\draw[pole,crosspole] (-3,0) to[out=90,in=270] (-3,1);
	\draw[pole,crosspole] (-2,0) to[out=90,in=270] (-2,1);
	\draw[pole,crosspole] (-1,0) to[out=90,in=270] (-1,1);
	\draw[usual,crossline] (0,0) to[out=90,in=270] (1,1);
	\draw[usual,crossline] (1,0) to[out=90,in=270] (0,1);
	\draw[usual,crossline] (2,0) to[out=90,in=270] (2,1);
	\draw[usual,crossline] (3,0) to[out=90,in=270] (3,1);
	\draw[pole,crosspole] (-3,1) to[out=90,in=270] (-3,2);
	\draw[pole,crosspole] (-2,1) to[out=90,in=270] (-2,2);
	\draw[usual,crossline] (0,1) to[out=90,in=270] (-1,2);
	\draw[pole,crosspole] (-1,1) to[out=90,in=270] (0,2);
	\draw[usual,crossline] (3,1) to[out=90,in=270] (1,2);
	\draw[usual,crossline] (1,1) to[out=90,in=270] (2,2);
	\draw[usual,crossline] (2,1) to[out=90,in=270] (3,2);
	\draw[pole,crosspole] (-3,2) to[out=90,in=270] (-3,3);
	\draw[usual,crossline] (-1,2) to[out=90,in=270] (-2,3);
	\draw[pole,crosspole] (-2,2) to[out=90,in=270] (-1,3);
	\draw[pole,crosspole] (0,2) to[out=90,in=270] (0,3);
	\draw[usual,crossline] (1,2) to[out=90,in=270] (1,3);
	\draw[usual,crossline] (2,2) to[out=90,in=270] (3,3);
	\draw[usual,crossline] (3,2) to[out=90,in=270] (2,3);
	\draw[pole,crosspole] (-3,3) to[out=90,in=270] (-3,4);
	\draw[pole,crosspole] (-1,3) to[out=90,in=270] (-2,4);
	\draw[pole,crosspole] (0,3) to[out=90,in=270] (-1,4);
	\draw[usual,crossline] (-2,3) to[out=90,in=270] (0,4);
	\draw[usual,crossline] (1,3) to[out=90,in=270] (3,4);
	\draw[usual,crossline] (2,3) to[out=90,in=270] (1,4);
	\draw[usual,crossline] (3,3) to[out=90,in=270] (2,4);
	\draw[pole,crosspole] (-3,4) to[out=90,in=270] (-3,5);
	\draw[pole,crosspole] (-2,4) to[out=90,in=270] (-2,5);
	\draw[pole,crosspole] (-1,4) to[out=90,in=270] (-1,5);
	\draw[usual,crossline,directed=.99] (1,4) to[out=90,in=270] (0,5);
	\draw[usual,crossline,directed=.99] (0,4) to[out=90,in=270] (2,5);
	\draw[usual,crossline,directed=.99] (2,4) to[out=90,in=270] (1,5);
	\draw[usual,crossline,directed=.99] (3,4) to[out=90,in=270] (3,5);
	\draw[cusual] (3,0) to[out=270,in=180] (3.5,-.5) 
	to[out=0,in=270] (4,0) to[out=90,in=270] (4,5) to[out=90,in=0] (3.5,5.5) to[out=180,in=90] (3,5);
	\draw[cusual] (2,0) to[out=270,in=180] (3.5,-1.0) 
	to[out=0,in=270] (5,0) to[out=90,in=270] (5,5) to[out=90,in=0] (3.5,6.0) to[out=180,in=90] (2,5);
	\draw[cusual] (1,0) to[out=270,in=180] (3.5,-1.5) 
	to[out=0,in=270] (6,0) to[out=90,in=270] (6,5) to[out=90,in=0] (3.5,6.5) to[out=180,in=90] (1,5);
	\draw[cusual] (0,0) to[out=270,in=180] (3.5,-2.0) 
	to[out=0,in=270] (7,0) to[out=90,in=270] (7,5) to[out=90,in=0] (3.5,7.0) to[out=180,in=90] (0,5);
	\draw[cpole] (-3,0) to[out=270,in=180] (-2,-.5);
	\draw[cpole] (-2,0) to[out=270,in=90] (-2,-.5);
	\draw[cpole] (-1,0) to[out=270,in=0] (-2,-.5);
	\draw[cpole] (-3,5) to[out=90,in=180] (-2,5.5);
	\draw[cpole] (-2,5) to[out=90,in=270] (-2,5.5);
	\draw[cpole] (-1,5) to[out=90,in=0] (-2,5.5);
	\draw node[pole] at (-2,5.25) {{\color{infinity}\LARGE$\bullet$}};
	\draw node[pole, above] at (-2,5.45) {\color{infinity}$\infty$};
	\draw node[pole] at (-2,-.75) {{\color{infinity}\LARGE$\bullet$}};
	\draw node[pole, below] at (-2,-.55) {\color{infinity}$\infty$};
	\draw[thick,mygray,fill=citron,fill opacity=.1] (-3.5,5) 
	to (3.5,5) to (3.5,0) to (-3.5,0) to (-3.5,5);
	\node at (-2.5,2.5) {$\braid$};
	\node at (3.5,-4) {\phantom{a}};
	\node at (3.5,7.4) {$\bclosure[\braid]$};
	\node [pole] at (3.5,10.0) {$\phantom{\hbody[3]^{c}}$};
	\node [infinity] at (-.75,5.6) {$\hbody[3]^{c}$};
	\node at (-4.0,5.6) {$\thesphere$};
\end{tikzpicture}
$}
\quad
\raisebox{-.35cm}{$\ouriso[\sim]$}
\;
\scalebox{.7}{$
\begin{tikzpicture}[anchorbase,scale=.5,tinynodes]
	\draw[pole,crosspole] (-3,0) to[out=90,in=270] (-3,1);
	\draw[pole,crosspole] (-2,0) to[out=90,in=270] (-2,1);
	\draw[pole,crosspole] (-1,0) to[out=90,in=270] (-1,1);
	\draw[usual,crossline] (0,0) to[out=90,in=270] (1,1);
	\draw[usual,crossline] (1,0) to[out=90,in=270] (0,1);
	\draw[usual,crossline] (2,0) to[out=90,in=270] (2,1);
	\draw[usual,crossline] (3,0) to[out=90,in=270] (3,1);
	\draw[pole,crosspole] (-3,1) to[out=90,in=270] (-3,2);
	\draw[pole,crosspole] (-2,1) to[out=90,in=270] (-2,2);
	\draw[usual,crossline] (0,1) to[out=90,in=270] (-1,2);
	\draw[pole,crosspole] (-1,1) to[out=90,in=270] (0,2);
	\draw[usual,crossline] (3,1) to[out=90,in=270] (1,2);
	\draw[usual,crossline] (1,1) to[out=90,in=270] (2,2);
	\draw[usual,crossline] (2,1) to[out=90,in=270] (3,2);
	\draw[pole,crosspole] (-3,2) to[out=90,in=270] (-3,3);
	\draw[usual,crossline] (-1,2) to[out=90,in=270] (-2,3);
	\draw[pole,crosspole] (-2,2) to[out=90,in=270] (-1,3);
	\draw[pole,crosspole] (0,2) to[out=90,in=270] (0,3);
	\draw[usual,crossline] (1,2) to[out=90,in=270] (1,3);
	\draw[usual,crossline] (2,2) to[out=90,in=270] (3,3);
	\draw[usual,crossline] (3,2) to[out=90,in=270] (2,3);
	\draw[pole,crosspole] (-3,3) to[out=90,in=270] (-3,4);
	\draw[pole,crosspole] (-1,3) to[out=90,in=270] (-2,4);
	\draw[pole,crosspole] (0,3) to[out=90,in=270] (-1,4);
	\draw[usual,crossline] (-2,3) to[out=90,in=270] (0,4);
	\draw[usual,crossline] (1,3) to[out=90,in=270] (3,4);
	\draw[usual,crossline] (2,3) to[out=90,in=270] (1,4);
	\draw[usual,crossline] (3,3) to[out=90,in=270] (2,4);
	\draw[pole,crosspole] (-3,4) to[out=90,in=270] (-3,5);
	\draw[pole,crosspole] (-2,4) to[out=90,in=270] (-2,5);
	\draw[pole,crosspole] (-1,4) to[out=90,in=270] (-1,5);
	\draw[usual,crossline,directed=.99] (1,4) to[out=90,in=270] (0,5);
	\draw[usual,crossline,directed=.99] (0,4) to[out=90,in=270] (2,5);
	\draw[usual,crossline,directed=.99] (2,4) to[out=90,in=270] (1,5);
	\draw[usual,crossline,directed=.99] (3,4) to[out=90,in=270] (3,5);
	\draw[cusual] (3,0) to[out=270,in=180] (3.5,-.5) 
	to[out=0,in=270] (4,0) to[out=90,in=270] (4,5) to[out=90,in=0] (3.5,5.5) to[out=180,in=90] (3,5);
	\draw[cusual] (2,0) to[out=270,in=180] (3.5,-1.0) 
	to[out=0,in=270] (5,0) to[out=90,in=270] (5,5) to[out=90,in=0] (3.5,6.0) to[out=180,in=90] (2,5);
	\draw[cusual] (1,0) to[out=270,in=180] (3.5,-1.5) 
	to[out=0,in=270] (6,0) to[out=90,in=270] (6,5) to[out=90,in=0] (3.5,6.5) to[out=180,in=90] (1,5);
	\draw[cusual] (0,0) to[out=270,in=180] (3.5,-2.0) 
	to[out=0,in=270] (7,0) to[out=90,in=270] (7,5) to[out=90,in=0] (3.5,7.0) to[out=180,in=90] (0,5);
	\draw[cpole] (-3,5) to[out=90,in=180] (-2,5.5);
	\draw[cpole] (-2,5) to[out=90,in=270] (-2,5.5);
	\draw[cpole] (-1,5) to[out=90,in=0] (-2,5.5);
	\draw[cpole] (-2,5.5) to (-2,6.5);
	\draw[cpole] (-3,7) to[out=270,in=180] (-2,6.5);
	\draw[cpole] (-2,7) to[out=270,in=270] (-2,6.5);
	\draw[cpole] (-1,7) to[out=270,in=0] (-2,6.5);
	\draw[cpole] (-1,7) to[out=90,in=180] (3.5,9) to[out=0,in=90] (8,7) to (8,0) 
	to[out=270,in=0] (3.5,-2.5) to[out=180,in=270] (-1,0);
	\draw[cpole] (-2,7) to[out=90,in=180] (3.5,9.5) to[out=0,in=90] (9,7) to (9,0)
	to[out=270,in=0] (3.5,-3) to[out=180,in=270] (-2,0);
	\draw[cpole] (-3,7) to[out=90,in=180] (3.5,10) to[out=0,in=90] (10,7) to (10,0) 
	to[out=270,in=0] (3.5,-3.5) to[out=180,in=270] (-3,0);
	\draw[thick,mygray,fill=citron,fill opacity=.1] (-3.5,5) 
	to (3.5,5) to (3.5,0) to (-3.5,0) to (-3.5,5);
	\node at (-2.5,2.5) {$\braid$};
	\node at (3.5,-4) {\phantom{a}};
	\node at (3.5,7.4) {$\bclosure[\braid]$};
	\node [pole] at (3.5,10.0) {$\phantom{\hbody[3]^{c}}$};
	\node [infinity] at (-.75,5.6) {$\hbody[3]^{c}$};
	\node at (-4.0,5.6) {$\thesphere$};
\end{tikzpicture}
$}
};
(0,-27)*{\braid\in\braidg[3,4],\,\bclosure[\braid]\subset\hbody[3]};
\endxy
\end{gather}
The braid itself is depicted as the solid strands in the indicated rectangle, 
while the dashed edges correspond to the closure procedure described above.
The thin, black components (both solid and dashed) give the link 
$\bclosure[\braid]$, 
while the thick, gray graph (again, both solid and dashed) depicts $\hbody^c$.
\end{example} 

The next result shows that, up to isotopy, all links in 
$\hbody$ arise from the closure procedure 
for handlebody braids described above.
The proof is analogous to the classical case.

\begin{theoremqed}(Alexander's Theorem in a handlebody; 
\cite[Theorem 2]{HaOlLa-handlebodies}.)\label{theorem:alexander}
Given a link $\link\subset\hbody$ there exists a braid $\braid\in\braidg[g,n]$ such that 
$\bclosure[\braid]\ouriso[\sim]\link\subset\hbody$.
\end{theoremqed}

\subsection{Generators and relations for braids in handlebodies}\label{subsec:links-handlebodies-dia}

We now recall the algebraic presentation of $\braidg[g,n]$.

\begin{definition}\label{definition:alg-braid-group}
The group $\braidd[g,n]$
is the group generated by $\bgen[1],\dots,\bgen[n{-}1]$ 
and $\tgen[1],\dots,\tgen[g]$, called braid and twist generators, respectively,
subject to the relations
\begin{align}
\label{eq:braid-rels-typeA}
\bgen[j]\bgen\bgen[j]=\bgen\bgen[j]\bgen\;\text{ if }|i-j]=1
&,\quad
\bgen[j]\bgen=\bgen\bgen[j]\;\text{ if }|i-j]>2,
\\
\label{eq:braid-rels-typeC}
\bgen[1]\tgen[i]\bgen[1]\tgen[i]=\tgen[i]\bgen[1]\tgen[i]\bgen[1]
&,\quad
\bgen\tgen[j]=\tgen[j]\bgen\;\text{ if }i>2,
\\
\label{eq:braid-rels-special}
\big(\bgen[1]\tgen[i]\bgen[1]^{-1}\big)\tgen[j]
=&\tgen[j]\big(\bgen[1]\tgen[i]\bgen[1]^{-1}\big)\;
\text{ for }i<j.
\end{align}
By convention, $\braidd[g,0]=\{1\}$, and we omit the twist generators when $g=0$ and 
the braid generators when $n=1$.
\end{definition}

The following theorem identifies $\braidg[g,n]$ and $\braidd[g,n]$, 
and we likewise do for the duration. 
In particular, we identify $\braidd[n]=\braidd[0,n]$ 
and $\braidg[n]$.

\begin{proposition}\label{theorem:alg-braid-group}
(\cite[Theorem 1]{Ve-handlebodies} \& \cite[Section 5]{La-handlebodies}.)
There is an isomorphism of groups
\begin{equation}
\pushQED{\qed}
\braidd[g,n]\ouriso\braidg[g,n].
\qedhere
\popQED
\end{equation}
\end{proposition}

An explicit isomorphism realizing \fullref{theorem:alg-braid-group} is given on the braid 
and the twist generators as follows:
\begin{gather}\label{eq:braid-group-iso}
\bgen[i]
\mapsto
\begin{tikzpicture}[anchorbase,scale=.5,tinynodes]
	\draw[pole,crosspole] (0,0) to[out=90,in=270] (0,1.5);
	\draw[pole,crosspole] (1,0) to[out=90,in=270] (1,1.5);
	\draw[usual,crossline,directed=1] (2,0) to[out=90,in=270] (2,1.5);
	\draw[usual,crossline,directed=1] (4,0) node[below]{i\tp 1} 
	to[out=90,in=270] (3,1.5) node[above,yshift=-2pt]{i};
	\draw[usual,crossline,directed=1] (3,0) node[below]{i} 
	to[out=90,in=270] (4,1.5) node[above,yshift=-2pt]{i\tp 1};
	\draw[usual,crossline,directed=1] (5,0) to[out=90,in=270] (5,1.5);
	\node at (.5,.755) {$...$};
	\node at (2.5,.755) {$...$};
	\node at (4.5,.755) {$...$};
\end{tikzpicture}
\quad\&\quad
\tgen[i]
\mapsto
\begin{tikzpicture}[anchorbase,scale=.5,tinynodes]
	\draw[pole,crosspole] (0,0) to[out=90,in=270] (0,1.5);
	\draw[pole,crosspole] (2,0) to[out=90,in=270] (2,1.5);
	\draw[usual,crossline] (3,0) node[below]{1} 
	to[out=90,in=0] (2.75,.5) to[out=180,in=0] (.75,.5) to[out=180,in=270] (.5,.75);
	\draw[pole,crosspole] (1,0) node[below,infinity]{i} to[out=90,in=270] (1,1.5) node[above,yshift=-2pt,infinity]{i};
	\draw[usual,crossline,directed=1] (.5,.75) to[out=90,in=180] (.75,1) 
	to[out=0,in=180] (2.75,1) to[out=0,in=270] (3,1.5) node[above,yshift=-2pt]{1};
	\draw[usual,crossline,directed=1] (4,0) to[out=90,in=270] (4,1.5);
	\draw[usual,crossline,directed=1] (5,0) to[out=90,in=270] (5,1.5);
	\node at (.5,.3) {$...$};
	\node at (.5,1.3) {$...$};
	\node at (1.5,.3) {$...$};
	\node at (1.5,1.3) {$...$};
	\node at (4.5,.755) {$...$};
\end{tikzpicture}
\end{gather}
The inverse of $\bgen[i]$ is given, as usual, by the corresponding opposite crossing,
and the inverse of $\tgen[i]$ is given as above, but with the braid strand wrapping the $i^{\mathrm{th}}$ 
core strand oppositely (however, it still crosses over the other core strands).

\begin{example}\label{example:alg-braid-group}
Under this map,
the first relation in \eqref{eq:braid-rels-typeA} corresponds to the braid-like 
Reidemeister III relation, 
and the second relations in \eqref{eq:braid-rels-typeA} and \eqref{eq:braid-rels-typeC} correspond 
to planar isotopy.
The four term relation in \eqref{eq:braid-rels-typeC} and 
the relation in \eqref{eq:braid-rels-special} become \eg 
the relations \eqref{eq:braid-typeAC} and \eqref{eq:braid-typeAC-twist}, respectively:
\begin{align}
\label{eq:braid-typeAC}
\bgen[1]\tgen[i]\bgen[1]\tgen[i]=
\tgen[i]\bgen[1]\tgen[i]\bgen[1]
&\leftrightsquigarrow\hspace*{-.1cm}
\scalebox{.85}{$
\begin{tikzpicture}[anchorbase,scale=.5,tinynodes]
	\draw[pole,crosspole] (2,0) to[out=90,in=270] (2,1.5);
	\draw[usual,crossline] (3,0) node[below]{1} to[out=90,in=0] (2.75,.5) 
	to[out=180,in=0] (.75,.5) to[out=180,in=270] (.5,.75);
	\draw[pole,crosspole] (1,0) node[below,infinity]{i} to[out=90,in=270] (1,1.5);
	\draw[usual,crossline] (.5,.75) to[out=90,in=180] (.75,1) 
	to[out=0,in=180] (2.75,1) to[out=0,in=270] (3,1.5);
	\draw[usual,crossline] (4,0) to[out=90,in=270] (4,1.5);
	\draw[pole,crosspole] (1,1.5) to[out=90,in=270] (1,3);
	\draw[pole,crosspole] (2,1.5) to[out=90,in=270] (2,3);
	\draw[usual,crossline] (4,1.5) to[out=90,in=270] (3,3);
	\draw[usual,crossline] (3,1.5) to[out=90,in=270] (4,3);
	\draw[pole,crosspole] (2,3) to[out=90,in=270] (2,4.5);
	\draw[usual,crossline] (3,3) to[out=90,in=0] (2.75,3.5) 
	to[out=180,in=0] (.75,3.5) to[out=180,in=270] (.5,3.75);
	\draw[pole,crosspole] (1,3) to[out=90,in=270] (1,4.5);
	\draw[usual,crossline] (.5,3.75) to[out=90,in=180] (.75,4) 
	to[out=0,in=180] (2.75,4) to[out=0,in=270] (3,4.5);
	\draw[usual,crossline] (4,3) to[out=90,in=270] (4,4.5);
	\draw[pole,crosspole] (1,4.5) to[out=90,in=270] (1,6) node[above,yshift=-2pt,infinity]{i};
	\draw[pole,crosspole] (2,4.5) to[out=90,in=270] (2,6);
	\draw[usual,crossline,directed=1] (4,4.5) to[out=90,in=270] (3,6) node[above,yshift=-2pt]{1};
	\draw[usual,crossline,directed=1] (3,4.5) to[out=90,in=270] (4,6);
	\node at (1.5,.2) {$...$};
	\node at (1.5,5.8) {$...$};
\end{tikzpicture}
$}
=\hspace*{-.1cm}
\scalebox{.85}{$
\begin{tikzpicture}[anchorbase,scale=.5,tinynodes]
	\draw[pole,crosspole] (1,-1.5) node[below,infinity]{i} to[out=90,in=270] (1,0);
	\draw[pole,crosspole] (2,-1.5) to[out=90,in=270] (2,0);
	\draw[usual,crossline] (4,-1.5) to[out=90,in=270] (3,0);
	\draw[usual,crossline] (3,-1.5) node[below]{1} to[out=90,in=270] (4,0);
	\draw[pole,crosspole] (2,0) to[out=90,in=270] (2,1.5);
	\draw[usual,crossline] (3,0) to[out=90,in=0] (2.75,.5) 
	to[out=180,in=0] (.75,.5) to[out=180,in=270] (.5,.75);
	\draw[pole,crosspole] (1,0) to[out=90,in=270] (1,1.5);
	\draw[usual,crossline] (.5,.75) to[out=90,in=180] (.75,1) 
	to[out=0,in=180] (2.75,1) to[out=0,in=270] (3,1.5);
	\draw[usual,crossline] (4,0) to[out=90,in=270] (4,1.5);
	\draw[pole,crosspole] (1,1.5) to[out=90,in=270] (1,3);
	\draw[pole,crosspole] (2,1.5) to[out=90,in=270] (2,3);
	\draw[usual,crossline] (4,1.5) to[out=90,in=270] (3,3);
	\draw[usual,crossline] (3,1.5) to[out=90,in=270] (4,3);
	\draw[pole,crosspole] (2,3) to[out=90,in=270] (2,4.5);
	\draw[usual,crossline] (3,3) to[out=90,in=0] (2.75,3.5) 
	to[out=180,in=0] (.75,3.5) to[out=180,in=270] (.5,3.75);
	\draw[pole,crosspole] (1,3) to[out=90,in=270] (1,4.5) node[above,yshift=-2pt,infinity]{i};
	\draw[usual,crossline,directed=1] (.5,3.75) to[out=90,in=180] (.75,4) 
	to[out=0,in=180] (2.75,4) to[out=0,in=270] (3,4.5) node[above,yshift=-2pt]{1};
	\draw[usual,crossline,directed=1] (4,3) to[out=90,in=270] (4,4.5);
	\node at (1.5,-1.3) {$...$};
	\node at (1.5,4.3) {$...$};
\end{tikzpicture}
$}
\\
\label{eq:braid-typeAC-twist}
\big(\bgen[1]\tgen[i]\bgen[1]^{-1}\big)\tgen[j]
=
\tgen[j]\big(\bgen[1]\tgen[i]\bgen[1]^{-1}\big)
&\leftrightsquigarrow\hspace*{-.1cm}
\scalebox{.85}{$
\begin{tikzpicture}[anchorbase,scale=.5,tinynodes]
	\draw[pole,crosspole] (1,0) node[below,infinity]{i} to[out=90,in=270] (1,1.5);
	\draw[usual,crossline] (3,0) node[below]{1} to[out=90,in=0] (2.75,.5) 
	to[out=180,in=0] (1.75,.5) to[out=180,in=270] (1.5,.75);
	\draw[pole,crosspole] (2,0) node[below,infinity]{j} to[out=90,in=270] (2,1.5);
	\draw[usual,crossline] (1.5,.75) to[out=90,in=180] (1.75,1) 
	to[out=0,in=180] (2.75,1) to[out=0,in=270] (3,1.5);
	\draw[usual,crossline] (4,0) to[out=90,in=270] (4,1.5);
	\draw[pole,crosspole] (1,1.5) to[out=90,in=270] (1,3);
	\draw[pole,crosspole] (2,1.5) to[out=90,in=270] (2,3);
	\draw[usual,crossline] (3,1.5) to[out=90,in=270] (4,3);
	\draw[usual,crossline] (4,1.5) to[out=90,in=270] (3,3);
	\draw[pole,crosspole] (2,3) to[out=90,in=270] (2,4.5);
	\draw[usual,crossline] (3,3) to[out=90,in=0] (2.75,3.5) 
	to[out=180,in=0] (.75,3.5) to[out=180,in=270] (.5,3.75);
	\draw[pole,crosspole] (1,3) to[out=90,in=270] (1,4.5);
	\draw[usual,crossline] (.5,3.75) to[out=90,in=180] (.75,4) 
	to[out=0,in=180] (2.75,4) to[out=0,in=270] (3,4.5);
	\draw[usual,crossline] (4,3) to[out=90,in=270] (4,4.5);
	\draw[pole,crosspole] (1,4.5) to[out=90,in=270] (1,6) node[above,yshift=-2pt,infinity]{i};
	\draw[pole,crosspole] (2,4.5) to[out=90,in=270] (2,6) node[above,yshift=-2pt,infinity]{j};
	\draw[usual,crossline,directed=1] (4,4.5) to[out=90,in=270] (3,6) node[above,yshift=-2pt]{1};
	\draw[usual,crossline,directed=1] (3,4.5) to[out=90,in=270] (4,6);
	\node at (1.5,.2) {$...$};
	\node at (1.5,5.8) {$...$};
	\node at (2.5,.2) {$...$};
	\node at (2.5,5.8) {$...$};
\end{tikzpicture}
$}
=\hspace*{-.1cm}
\scalebox{.85}{$
\begin{tikzpicture}[anchorbase,scale=.5,tinynodes]
	\draw[pole,crosspole] (1,-1.5) node[below,infinity]{i} to[out=90,in=270] (1,0);
	\draw[pole,crosspole] (2,-1.5) node[below,infinity]{j} to[out=90,in=270] (2,0);
	\draw[usual,crossline] (3,-1.5) node[below]{1} to[out=90,in=270] (4,0);
	\draw[usual,crossline] (4,-1.5) to[out=90,in=270] (3,0);
	\draw[pole,crosspole] (2,0) to[out=90,in=270] (2,1.5);
	\draw[usual,crossline] (3,0) to[out=90,in=0] (2.75,.5) 
	to[out=180,in=0] (.75,.5) to[out=180,in=270] (.5,.75);
	\draw[pole,crosspole] (1,0) to[out=90,in=270] (1,1.5);
	\draw[usual,crossline] (.5,.75) to[out=90,in=180] (.75,1) 
	to[out=0,in=180] (2.75,1) to[out=0,in=270] (3,1.5);
	\draw[usual,crossline] (4,0) to[out=90,in=270] (4,1.5);
	\draw[pole,crosspole] (1,1.5) to[out=90,in=270] (1,3);
	\draw[pole,crosspole] (2,1.5) to[out=90,in=270] (2,3);
	\draw[usual,crossline] (4,1.5) to[out=90,in=270] (3,3);
	\draw[usual,crossline] (3,1.5) to[out=90,in=270] (4,3);
	\draw[pole,crosspole] (1,3) to[out=90,in=270] (1,4.5) node[above,yshift=-2pt,infinity]{i};
	\draw[usual,crossline] (3,3) to[out=90,in=0] (2.75,3.5) 
	to[out=180,in=0] (1.75,3.5) to[out=180,in=270] (1.5,3.75);
	\draw[pole,crosspole] (2,3) node[below]{j} to[out=90,in=270] (2,4.5) node[above,yshift=-2pt,infinity]{j};
	\draw[usual,crossline,directed=1] (1.5,3.75) to[out=90,in=180] (1.75,4) 
	to[out=0,in=180] (2.75,4) to[out=0,in=270] (3,4.5) node[above,yshift=-2pt]{1};
	\draw[usual,crossline,directed=1] (4,3) to[out=90,in=270] (4,4.5);
	\node at (1.5,-1.3) {$...$};
	\node at (1.5,4.3) {$...$};
	\node at (2.5,-1.3) {$...$};
	\node at (2.5,4.3) {$...$};
\end{tikzpicture}
$}
\end{align}
The non-trivial statement in \fullref{theorem:alg-braid-group}
is that these relations are sufficient.
\end{example}

\subsection{Markov's theorem}\label{subsec:links-markov}

Let $\braidg[g,\infty]:={\sqcup_{n\in\N[]}}\,\braidg[g,n]$, 
the set of all braids in $\hbody$.

\begin{definition}\label{definition:braid-quotient}
Let $\braidq[g,\infty]:=\braidg[g,\infty]/\ouriso[\sim]$ be the quotient given by 
conjugation \eqref{eq:conjugation} in (each) $\braidg[g,n]$
by elements $\topstuff{s}\in\langle\bgen[1],\dots,\bgen[n{-}1]\rangle$ 
and stabilization \eqref{eq:stabilization}, \ie
\begin{gather}\label{eq:conjugation}
\begin{gathered}
\braid\ouriso[\sim]\topstuff{s}\braid\topstuff{s}^{-1}\\
\text{for }\braid\in\braidg[g,n],\,\topstuff{s}\in\langle\bgen[1],\dots,\bgen[n{-}1]\rangle
\end{gathered}
\quad\leftrightsquigarrow\quad
\begin{tikzpicture}[anchorbase,scale=.5,tinynodes]
	\draw[pole,crosspole] (-1,-.75) to[out=90,in=270] (-1,.5);
	\node at (-.5,.1) {$...$};
	\draw[pole,crosspole] (0,-.75) to[out=90,in=270] (0,.5);
	\draw[usual,crossline] (1,-.75) to[out=90,in=270] (1,.5);
	\node at (1.5,.1) {$...$};
	\draw[usual,crossline] (2,-.75) node[below]{n} to[out=90,in=270] (2,.5);
	\draw[pole,crosspole] (-1,1) to[out=90,in=270] (-1,2.25);
	\node at (-.5,1.6) {$...$};
	\draw[pole,crosspole] (0,1) to[out=90,in=270] (0,2.25);
	\draw[usual,directed=1,crossline] (1,1) to (1,2.25);
	\node at (1.5,1.6) {$...$};
	\draw[usual,directed=1,crossline] (2,1) to (2,2.25) node[above,yshift=-2pt]{n};
	\draw[thin,black,fill=mygray,fill opacity=.35] (-1.25,.5) rectangle (2.25,1);
	\node at (.5,.675) {$\braid$};
\end{tikzpicture}
\ouriso[\sim]
\begin{tikzpicture}[anchorbase,scale=.5,tinynodes]
	\draw[pole,crosspole] (-1,-.75) to[out=90,in=270] (-1,.5);
	\node at (-.5,.1) {$...$};
	\draw[pole,crosspole] (0,-.75) to[out=90,in=270] (0,.5);
	\draw[usual,crossline] (1,-.75) to[out=90,in=270] (1,-.25);
	\draw[usual,crossline] (1,.25) to[out=90,in=270] (1,.5);
	\node at (1.5,-.4) {$...$};
	\draw[usual,crossline] (2,-.75) node[below]{n} to[out=90,in=270] (2,-.25);
	\draw[usual,crossline] (2,.25) to[out=90,in=270] (2,.5);
	\draw[pole,crosspole] (-1,1) to[out=90,in=270] (-1,2.25);
	\node at (-.5,1.6) {$...$};
	\draw[pole,crosspole] (0,1) to[out=90,in=270] (0,2.25);
	\draw[usual,crossline] (1,1) to (1,1.25);
	\draw[usual,crossline,directed=1] (1,1.75) to (1,2.25);
	\node at (1.5,2) {$...$};
	\draw[usual,crossline] (2,1) to (2,1.25);
	\draw[usual,crossline,directed=1] (2,1.75) to (2,2.25) node[above,yshift=-2pt]{n};
	\draw[thin,black,fill=mygray,fill opacity=.35] (-1.25,.5) rectangle (2.25,1);
	\node at (.5,.675) {$\braid$};
	\draw[thin,black,fill=mygray,fill opacity=.35] (.75,1.25) rectangle (2.25,1.75);
	\node at (1.5,1.5) {$\topstuff{s}$};
	\draw[thin,black,fill=mygray,fill opacity=.35] (.75,-.25) rectangle (2.25,.25);
	\node at (1.5,-.077) {$\topstuff{s}^{\scalebox{.5}{-1}}$};
\end{tikzpicture}
\end{gather}
\begin{gather}\label{eq:stabilization}
\begin{gathered}
(\braid[c]\onebraid)\bgen[n](\braid\onebraid)\ouriso[\sim]\braid[c]\braid
\ouriso[\sim](\braid[c]\onebraid)\bgen[n]^{-1}(\braid\onebraid) \\
\text{for }\braid,\braid[c]\in\braidg[g,n],
\end{gathered}
\;\leftrightsquigarrow\;
\begin{tikzpicture}[anchorbase,scale=.5,tinynodes]
	\draw[pole,crosspole] (-1,0) to[out=90,in=270] (-1,.5);
	\draw[pole,crosspole] (0,0) to[out=90,in=270] (0,.5);
	\draw[usual,crossline] (1,0) to[out=90,in=270] (1,.5);
	\draw[usual,crossline] (2,0) node[below]{n} to[out=90,in=270] (2,.5);
	\draw[usual,crossline] (3,0) to[out=90,in=270] (3,1);
	\draw[pole,crosspole] (-1,1) to[out=90,in=270] (-1,2);
	\draw[pole,crosspole] (0,1) to[out=90,in=270] (0,2);
	\draw[usual,crossline] (1,1) to[out=90,in=270] (1,2);
	\draw[usual,crossline] (3,1) to[out=90,in=270] (2,2);
	\draw[usual,crossline] (2,1) to[out=90,in=270] (3,2);
	\draw[pole,crosspole] (-1,2.5) to[out=90,in=270] (-1,3);
	\draw[pole,crosspole] (0,2.5) to[out=90,in=270] (0,3);
	\draw[usual,directed=1,crossline] (1,2.5) to[out=90,in=270] (1,3);
	\draw[usual,directed=1,crossline] (2,2.5) to[out=90,in=270] (2,3) node[above,yshift=-2pt]{n};
	\draw[usual,directed=1,crossline] (3,2) to (3,3);
	\draw[thin,black,fill=mygray,fill opacity=.35] (-1.25,.5) rectangle (2.25,1);
	\draw[thin,black,fill=mygray,fill opacity=.35] (-1.25,2) rectangle (2.25,2.5);
	\node at (.5,.675) {$\braid$};
	\node at (.5,2.25) {$\braid[c]$};
	\node at (-.5,.2) {$...$};
	\node at (-.5,2.8) {$...$};
	\node at (1.5,.2) {$...$};
	\node at (1.5,2.8) {$...$};
\end{tikzpicture}
\ouriso[\sim]
\begin{tikzpicture}[anchorbase,scale=.5,tinynodes]
	\draw[pole,crosspole] (-1,0) to[out=90,in=270] (-1,.5);
	\draw[pole,crosspole] (0,0) to[out=90,in=270] (0,.5);
	\draw[usual,crossline] (1,0) to[out=90,in=270] (1,.5);
	\draw[usual,crossline] (2,0) node[below]{n} to[out=90,in=270] (2,.5);
	\draw[pole,crosspole] (-1,1) to[out=90,in=270] (-1,2);
	\draw[pole,crosspole] (0,1) to[out=90,in=270] (0,2);
	\draw[usual,crossline] (1,1) to[out=90,in=270] (1,2);
	\draw[usual,crossline] (2,1) to[out=90,in=270] (2,2);
	\draw[pole,crosspole] (-1,2.5) to[out=90,in=270] (-1,3);
	\draw[pole,crosspole] (0,2.5) to[out=90,in=270] (0,3);
	\draw[usual,directed=1,crossline] (1,2.5) to[out=90,in=270] (1,3);
	\draw[usual,directed=1,crossline] (2,2.5) to[out=90,in=270] (2,3) node[above,yshift=-2pt]{n};
	\draw[thin,black,fill=mygray,fill opacity=.35] (-1.25,.5) rectangle (2.25,1);
	\draw[thin,black,fill=mygray,fill opacity=.35] (-1.25,2) rectangle (2.25,2.5);
	\node at (.5,.675) {$\braid$};
	\node at (.5,2.25) {$\braid[c]$};
	\node at (-.5,.2) {$...$};
	\node at (-.5,2.8) {$...$};
	\node at (1.5,.2) {$...$};
	\node at (1.5,2.8) {$...$};
\end{tikzpicture}
\ouriso[\sim]
\begin{tikzpicture}[anchorbase,scale=.5,tinynodes]
	\draw[pole,crosspole] (-1,0) to[out=90,in=270] (-1,.5);
	\draw[pole,crosspole] (0,0) to[out=90,in=270] (0,.5);
	\draw[usual,crossline] (1,0) to[out=90,in=270] (1,.5);
	\draw[usual,crossline] (2,0) node[below]{n} to[out=90,in=270] (2,.5);
	\draw[usual,crossline] (3,0) to[out=90,in=270] (3,1);
	\draw[pole,crosspole] (-1,1) to[out=90,in=270] (-1,2);
	\draw[pole,crosspole] (0,1) to[out=90,in=270] (0,2);
	\draw[usual,crossline] (1,1) to[out=90,in=270] (1,2);
	\draw[usual,crossline] (2,1) to[out=90,in=270] (3,2);
	\draw[usual,crossline] (3,1) to[out=90,in=270] (2,2);
	\draw[pole,crosspole] (-1,2.5) to[out=90,in=270] (-1,3);
	\draw[pole,crosspole] (0,2.5) to[out=90,in=270] (0,3);
	\draw[usual,directed=1,crossline] (1,2.5) to[out=90,in=270] (1,3);
	\draw[usual,directed=1,crossline] (2,2.5) to[out=90,in=270] (2,3) node[above,yshift=-2pt]{n};
	\draw[usual,directed=1,crossline] (3,2) to (3,3);
	\draw[thin,black,fill=mygray,fill opacity=.35] (-1.25,.5) rectangle (2.25,1);
	\draw[thin,black,fill=mygray,fill opacity=.35] (-1.25,2) rectangle (2.25,2.5);
	\node at (.5,.675) {$\braid$};
	\node at (.5,2.25) {$\braid[c]$};
	\node at (-.5,.2) {$...$};
	\node at (-.5,2.8) {$...$};
	\node at (1.5,.2) {$...$};
	\node at (1.5,2.8) {$...$};
\end{tikzpicture}
\end{gather}
where $\braid\onebraid\in\braidg[g,n+1]$ is the braid obtained from $\braid\in\braidg[g,n]$
by adding a strand to the right.
\end{definition}

\begin{remark}\label{remark:stuck}
The conjugation \eqref{eq:conjugation} is weaker than in the classical case --
there one can conjugate by any element, instead of just by certain elements. 
This will play an important role in our construction of the HOMFLYPT invariant, see \fullref{proposition:links-get-stuck}. 
On the other hand, 
the stabilization \eqref{eq:stabilization} 
is stronger than the classical case when considered on its own, 
but together with the classical conjugation 
relation is equivalent to the classical stabilization.
\end{remark}

\begin{theoremqed}(Markov's Theorem in a handlebody; \cite[Theorem 5]{HaOlLa-handlebodies}.)\label{theorem:markov}
Let $\braid,\braid[c]\in\braidg[g,\infty]$,
then $\bclosure\ouriso[\sim]\bclosure[c]\subset\hbody$ if and only if 
$\braid=\braid[c]\in\braidq[g,\infty]$.
\end{theoremqed}

\begin{remark}
Although it may appear that \fullref{definition:braid-quotient}
omits conjugation by certain elements that clearly give isotopic closures
\eg by the ``maximal loop'' $\omega=\tgen[g]\dots\tgen[1]$ or its inverse, 
\cite[Section 5]{HaOlLa-handlebodies} shows how 
conjugation by such elements
can be described in terms of the above Markov moves.
\end{remark}

\subsection{From handlebody braids to classical braids}\label{subsec:handlebody-to-usual}

Recall from \fullref{subsection:intro-digression} that one of our main ingredients 
in constructing homological invariants of links in $\hbody$ for all $g\geq 0$ 
is the relation between $\braidg[g,n]$ and (a colored variant of) the type $\typeA$ braid group $\braidg[g+n]$: 
We have a group homomorphism $\braidg[g,n]\to\braidg[0,g+n]$ given by 
viewing the core strands as ``usual'' strands, \eg
\begin{gather}\label{eq:into-typeA}
\begin{tikzpicture}[anchorbase,scale=.5,tinynodes]
	\draw[pole,crosspole] (2,0) node[below,infinity]{g} to[out=90,in=270] (2,1.5) node[above,yshift=-2pt,infinity]{g};
	\draw[usual,crossline] (3,0) node[below]{1} 
	to[out=90,in=0] (2.75,.5) to[out=180,in=0] (.75,.5) to[out=180,in=270] (.5,.75);
	\draw[pole,crosspole] (1,0) node[below,infinity]{i} to[out=90,in=270] (1,1.5) node[above,yshift=-2pt,infinity]{i};
	\draw[usual,crossline,directed=1] (.5,.75) to[out=90,in=180] (.75,1) 
	to[out=0,in=180] (2.75,1) to[out=0,in=270] (3,1.5) node[above,yshift=-2pt]{1};
	\node at (1.5,.1) {$...$};
	\node at (1.5,1.5) {$...$};
\end{tikzpicture}
\mapsto
\begin{tikzpicture}[anchorbase,scale=.5,tinynodes]
	\draw[usual,crossline,directed=1] (2,0) node[below]{g} 
	to[out=90,in=270] (2,1.5) node[above,yshift=-2pt]{g};
	\draw[usual,crossline] (3,0) node[below]{g\tp 1} to[out=90,in=0] (2.75,.5) 
	to[out=180,in=0] (.75,.5) to[out=180,in=270] (.5,.75);
	\draw[usual,crossline,directed=1] (1,0) node[below]{i} 
	to[out=90,in=270] (1,1.5) node[above,yshift=-2pt]{i};
	\draw[usual,crossline,directed=1] (.5,.75) to[out=90,in=180] (.75,1) 
	to[out=0,in=180] (2.75,1) to[out=0,in=270] (3,1.5) node[above,yshift=-2pt]{g\tp 1};
	\node at (1.5,.1) {$...$};
	\node at (1.5,1.5) {$...$};
\end{tikzpicture}
\end{gather}
As discussed above, this map is clearly injective, 
hence we have:

\begin{proposition}\label{proposition:artintits-a-higher-g}
The map induced by \eqref{eq:into-typeA} gives rise to an embedding of groups
\begin{equation}
\pushQED{\qed}
\braidg[g,n]\hookrightarrow\braidg[g+n].
\qedhere
\popQED
\end{equation}
\end{proposition}

However,
\fullref{proposition:artintits-a-higher-g} 
is only one ingredient in our construction, since
invariance under the procedure
\begin{gather}\label{eq:classical-conjugation}
\begin{tikzpicture}[anchorbase,scale=.5,tinynodes]
	\draw[pole,crosspole] (-1,-.75) to[out=90,in=270] (-1,.5);
	\node at (-.5,.1) {$...$};
	\draw[pole,crosspole] (0,-.75) to[out=90,in=270] (0,.5);
	\draw[usual,crossline] (1,-.75) to[out=90,in=270] (1,.5);
	\node at (1.5,.1) {$...$};
	\draw[usual,crossline] (2,-.75) node[below]{n} to[out=90,in=270] (2,.5);
	\draw[pole,crosspole] (-1,1) to[out=90,in=270] (-1,2.25);
	\node at (-.5,1.6) {$...$};
	\draw[pole,crosspole] (0,1) to[out=90,in=270] (0,2.25);
	\draw[usual,directed=1,crossline] (1,1) to (1,2.25);
	\node at (1.5,1.6) {$...$};
	\draw[usual,directed=1,crossline] (2,1) to (2,2.25) node[above,yshift=-2pt]{n};
	\draw[thin,black,fill=mygray,fill opacity=.35] (-1.25,.5) rectangle (2.25,1);
	\node at (.5,.675) {$\braid$};
\end{tikzpicture}
\mapsto
\begin{tikzpicture}[anchorbase,scale=.5,tinynodes]
	\draw[usual,crossline] (-1,-.75) to[out=90,in=270] (-1,.5);
	\node at (-.5,.1) {$...$};
	\draw[usual,crossline] (0,-.75) to[out=90,in=270] (0,.5);
	\draw[usual,crossline] (1,-.75) to[out=90,in=270] (1,.5);
	\node at (1.5,.1) {$...$};
	\draw[usual,crossline] (2,-.75) node[below]{g\tp n} to[out=90,in=270] (2,.5);
	\draw[usual,crossline,directed=1] (-1,1) to[out=90,in=270] (-1,2.25);
	\node at (-.5,1.6) {$...$};
	\draw[usual,crossline,directed=1] (0,1) to[out=90,in=270] (0,2.25);
	\draw[usual,crossline,directed=1] (1,1) to (1,2.25);
	\node at (1.5,1.6) {$...$};
	\draw[usual,crossline,directed=1] (2,1) to (2,2.25) node[above,yshift=-2pt]{g\tp n};
	\draw[thin,black,fill=mygray,fill opacity=.35] (-1.25,.5) rectangle (2.25,1);
	\node at (.5,.675) {$\braid$};
\end{tikzpicture}
\hspace*{-.2cm}
\ouriso[\sim]
\begin{tikzpicture}[anchorbase,scale=.5,tinynodes]
	\draw[usual,crossline] (-1,-.75) to[out=90,in=270] (-1,-.25);
	\draw[usual,crossline] (-1,.25) to[out=90,in=270] (-1,.5);
	\node at (-.5,-.4) {$...$};
	\draw[usual,crossline] (0,-.75) to[out=90,in=270] (0,-.25);
	\draw[usual,crossline] (0,.25) to[out=90,in=270] (0,.5);
	\draw[usual,crossline] (1,-.75) to[out=90,in=270] (1,-.25);
	\draw[usual,crossline] (1,.25) to[out=90,in=270] (1,.5);
	\node at (1.5,-.4) {$...$};
	\draw[usual,crossline] (2,-.75) node[below]{g\tp n} to[out=90,in=270] (2,-.25);
	\draw[usual,crossline] (2,.25) to[out=90,in=270] (2,.5);
	\draw[usual,crossline] (-1,1) to[out=90,in=270] (-1,1.25);
	\draw[usual,crossline,directed=1] (-1,1.75) to[out=90,in=270] (-1,2.25);
	\node at (-.5,2) {$...$};
	\draw[usual,crossline] (0,1) to[out=90,in=270] (0,1.25);
	\draw[usual,crossline,directed=1] (0,1.75) to[out=90,in=270] (0,2.25);
	\draw[usual,crossline] (1,1) to (1,1.25);
	\draw[usual,crossline,directed=1] (1,1.75) to (1,2.25);
	\node at (1.5,2) {$...$};
	\draw[usual,crossline] (2,1) to (2,1.25);
	\draw[usual,crossline,directed=1] (2,1.75) to (2,2.25) node[above,yshift=-2pt]{g\tp n};
	\draw[thin,black,fill=mygray,fill opacity=.35] (-1.25,.5) rectangle (2.25,1);
	\node at (.5,.675) {$\braid$};
	\draw[thin,black,fill=mygray,fill opacity=.35] (-1.25,1.25) rectangle (2.25,1.75);
	\node at (.5,1.5) {$\topstuff{s}$};
	\draw[thin,black,fill=mygray,fill opacity=.35] (-1.25,-.25) rectangle (2.25,.25);
	\node at (.5,-.078) {$\topstuff{s}^{\scalebox{.5}{-1}}$};
\end{tikzpicture}
\hspace*{-.2cm}
\mapsfrom
\begin{tikzpicture}[anchorbase,scale=.5,tinynodes]
	\draw[pole,crosspole] (-1,-.75) to[out=90,in=270] (-1,-.25);
	\draw[pole,crosspole] (-1,.25) to[out=90,in=270] (-1,.5);
	\node at (-.5,-.4) {$...$};
	\draw[pole,crosspole] (0,-.75) to[out=90,in=270] (0,-.25);
	\draw[pole,crosspole] (0,.25) to[out=90,in=270] (0,.5);
	\draw[usual,crossline] (1,-.75) to[out=90,in=270] (1,-.25);
	\draw[usual,crossline] (1,.25) to[out=90,in=270] (1,.5);
	\node at (1.5,-.4) {$...$};
	\draw[usual,crossline] (2,-.75) node[below]{n} to[out=90,in=270] (2,-.25);
	\draw[usual,crossline] (2,.25) to[out=90,in=270] (2,.5);
	\draw[pole,crosspole] (-1,1) to[out=90,in=270] (-1,1.25);
	\draw[pole,crosspole] (-1,1.75) to[out=90,in=270] (-1,2.25);
	\node at (-.5,2) {$...$};
	\draw[pole,crosspole] (0,1) to[out=90,in=270] (0,1.25);
	\draw[pole,crosspole] (0,1.75) to[out=90,in=270] (0,2.25);
	\draw[usual,crossline] (1,1) to (1,1.25);
	\draw[usual,crossline,directed=1] (1,1.75) to (1,2.25);
	\node at (1.5,2) {$...$};
	\draw[usual,crossline] (2,1) to (2,1.25);
	\draw[usual,crossline,directed=1] (2,1.75) to (2,2.25) node[above,yshift=-2pt]{n};
	\draw[thin,black,fill=mygray,fill opacity=.35] (-1.25,.5) rectangle (2.25,1);
	\node at (.5,.675) {$\braid$};
	\draw[thin,black,fill=mygray,fill opacity=.35] (-1.25,1.25) rectangle (2.25,1.75);
	\node at (.5,1.5) {$\topstuff{s}$};
	\draw[thin,black,fill=mygray,fill opacity=.35] (-1.25,-.25) rectangle (2.25,.25);
	\node at (.5,-.078) {$\topstuff{s}^{\scalebox{.5}{-1}}$};
\end{tikzpicture}
\end{gather}
(i.e. under ``conjugation in $\braidg[g+n]$'') is not desirable 
for an invariant of $\link\ouriso[\sim]\bclosure[\braid]\subset\hbody$, \cf \fullref{remark:stuck}.
As such, we will use the theory of singular Soergel bimodules 
to mimic the merging and splitting of the core strands in 
the closure procedure for $\braidg[g,n]$, which will lead to 
invariants of $\bclosure[\braid]$ that are not invariant under \eqref{eq:classical-conjugation}.

\section{Braids in handlebodies and singular type \texorpdfstring{$\typeA$}{A} Soergel bimodules}\label{section:genusg}

In the present section, 
we construct a map from
$\braidg[g,n]$ 
to the $2$-category of
singular Soergel bimodules.

\subsection{Parabolic subgroups and Frobenius extensions}\label{subsec:singular-pre-primer}

Fix $N\in\N[\geq 1]$
and let $\rring:=\rring_{N}:=\K[\elstuff{x}_{1},\dots,\elstuff{x}_{N}]$ 
be the $\qpar$-graded polynomial ring with $\qdeg(x_i)=2$ for all $i$
(by convention, $\rring_{0}:=\K$).
The symmetric group
$\symgroup[N]=\coxgroup[{\typeA[N{-}1]}]$
acts on $\rring$ via 
\begin{gather}\label{eq:typeA-rep}
\coxgen\acts
\elstuff{x}_{j}=
\begin{cases}
\elstuff{x}_{i{+}1} &\text{if }i=j,
\\
\elstuff{x}_{i} &\text{if }i=j+1,
\\
\elstuff{x}_{j} &\text{else}.
\end{cases}
\end{gather}

\begin{remark}\label{remark:typeA-rep}
Recall that Tits defined a faithful representation of 
any Coxeter group $\coxgroup$ on a real vector space of dimension $|\setstuff{V}|$, 
commonly called the reflection representation of $\coxgroup$. 
(Recall our notation from \fullref{subsection:intro-digression}.)
This representation is a crucial ingredient in the original definition 
of the associated category of Soergel bimodules, see \cite[Section 1.4]{So-hcbim}. 
In our case, this is the standard (irreducible) representation 
of $\symgroup[N]$ of dimension $N-1$.
By contrast, the representation given by \eqref{eq:typeA-rep} is built from
the $N$-dimensional permutation representation, 
which decomposes as a direct sum of the standard representation and the trivial representation.
By \eg \cite[Section 4.6]{ElKh-diagrams-soergel} 
and \cite[Theorem 2.7 and Proposition 2.10]{ElLo-modular-rep-typea}, 
the difference (akin to the difference 
between considering $\mathfrak{gl}_{N}$ rather than $\mathfrak{sl}_{N}$) 
will not play a role in the present work, in the sense that all results
from the cited literature hold in this case as well.
\end{remark}

Fix any tuple $\stuple=(k_{1},\dots,k_{r})\in\N[\geq 1]^{r}$ with $k_{1}+\dots+k_{r}=N$. 
(Note further that choosing $\stuple$ also determines $N$ since $N=k_{1}+\dots+k_{r}$. 
We will tacitly use this throughout.)
By definition, the corresponding parabolic subgroup is
\begin{gather}\label{eq:parabolic}
\psymgroup{N}{\stuple}:=
\symgroup[k_{1}]\times\dots\times\symgroup[k_{r}]
\subset\symgroup[N].
\end{gather}
Since there is a bijection between tuples and parabolic subgroups, 
we will implicitly identify them,
\eg $\stuple[I]\subset\stuple[J]$ denotes an inclusion of parabolic subgroups.

Given a parabolic subgroup $\stuple$, 
we let $\rring[\stuple]:=\rring[\psymgroup{N}{\stuple}]$ be the ring of invariants.
This ring is $\qpar$-graded, since the action in \eqref{eq:typeA-rep} 
is $\qpar$-homogeneous.

\begin{example}\label{example:parabolic2}
The parabolic subgroups in \eqref{eq:parabolic} can alternatively be defined
by choosing corresponding subsets of the 
vertices $\setstuff{V}=\{\elstuff{1},\dots,\elstuff{N{-}1}\}$ of the type $\typeA[N{-}1]$ 
Coxeter diagram (with the left-right order of the vertices). For type $\typeA[3]$ one gets 
\begin{gather}
\begin{aligned}
(1,1,1,1)&\leftrightsquigarrow\emptyset,
\\
(2,1,1)&\leftrightsquigarrow\{\elstuff{1}\},
\end{aligned}
\quad\quad
\begin{aligned}
(1,2,1)&\leftrightsquigarrow\{\elstuff{2}\},
\\
(1,1,2)&\leftrightsquigarrow\{\elstuff{3}\},
\end{aligned}
\quad\quad
\begin{aligned}
(3,1)&\leftrightsquigarrow\{\elstuff{1},\elstuff{2}\},
\\
(2,2)&\leftrightsquigarrow\{\elstuff{1},\elstuff{3}\},
\end{aligned}
\quad\quad
\begin{aligned}
(1,3)&\leftrightsquigarrow\{\elstuff{2},\elstuff{3}\},
\\
(4)&\leftrightsquigarrow\{\elstuff{1},\elstuff{2},\elstuff{3}\}.
\end{aligned}
\end{gather}
Above we have listed all choices of tuples and the associated parabolic subgroups. 
Thus, $\rring[(1,1,1,1)]=\rring[\emptyset]$ is $\rring$ itself, while 
$\rring[(4)]=\rring[\{\elstuff{1},\elstuff{2},\elstuff{3}\}]$ is the $\K$-algebra 
of symmetric polynomials in four variables.
\end{example}

For the duration, we will use the following ordering convention for parabolic subgroups
$\stuple,\stuple[J],\stuple[K],\stuple[L]$ and their rings of invariants:
\begin{gather}\label{eq:ParabolicConv}
\begin{tikzcd}[column sep=.1em,row sep=.01em,ampersand replacement=\&]
\phantom{.} \& \fcolorbox{mygray}{white}{$\stuple[J]\phantom{\rring[{\stuple[J]}]}\hspace*{-.45cm}$} \arrow[rd,specialsymbol=\subset] \& \phantom{.} \\
\fcolorbox{mygray}{white}{$\stuple\phantom{\rring[{\stuple[J]}]}\hspace*{-.45cm}$} \arrow[ru,specialsymbol=\subset]\arrow[rd,specialsymbol=\subset] \& \phantom{.} \& 
\fcolorbox{mygray}{white}{$\stuple[L]\phantom{\rring[{\stuple[J]}]}\hspace*{-.45cm}$} \\
\phantom{.} \& \fcolorbox{mygray}{white}{$\stuple[K]\phantom{\rring[{\stuple[J]}]}\hspace*{-.45cm}$} \arrow[ru,specialsymbol=\subset] \& \phantom{.} \\
\end{tikzcd}
\quad
\raisebox{-.1cm}{$\leftrightsquigarrow$}
\quad
\begin{tikzcd}[column sep=.1em,row sep=.01em,ampersand replacement=\&]
\phantom{.} \& \fcolorbox{mygray}{white}{$\rring[{\stuple[J]}]$} \arrow[rd,specialsymbol=\supset] \& \phantom{.} \\
\fcolorbox{mygray}{white}{$\rring[{\stuple}]$} \arrow[ru,specialsymbol=\supset]\arrow[rd,specialsymbol=\supset] \& \phantom{.} \& \fcolorbox{mygray}{white}{$\rring[{\stuple[L]}]$} \\
\phantom{.} \& \fcolorbox{mygray}{white}{$\rring[{\stuple[K]}]$} \arrow[ru,specialsymbol=\supset] \& \phantom{.} \\
\end{tikzcd}\raisebox{-.1cm}{.}
\end{gather}
The $\qpar$-degree $0$ inclusion
\begin{gather}\label{eq:singular-iota}
\iota_{\stuple}^{\stuple[J]}\colon\rring[{\stuple[J]}]\hookrightarrow\rring[{\stuple}]
\end{gather}
of $\K$-algebras is a $\qpar$-graded Frobenius extension (see \cite{ElSnWi-frob-extensions}), 
meaning that $\rring[{\stuple}]$ is a $\qpar$-graded, free $\rring[{\stuple[J]}]$-module of finite rank, 
possessing a non-degenerate, $\rring[{\stuple[J]}]$-linear trace map $\rring[{\stuple}]\to\rring[{\stuple[J]}]$.
In the present case, the latter is built using
the Demazure operators 
$\partial_{\coxgen}\colon\rring\to\rring[\{\elstuff{i}\}]\subset\rring$, given 
by $\partial_{\coxgen}(f):=(f-\coxgen\acts f)/\rooty$ 
for the roots $\rooty=\elstuff{x}_{i}-\elstuff{x}_{i{+}1}$.
The collection $\{\partial_{\coxgen}\}$
satisfies the classical braid relations, and 
thus gives a well-defined map $\partial_{w}$ 
associated to any $w\in\symgroup[N]$ 
using a reduced expression for $w$.
Using these, the aforementioned trace map is
given by
\begin{gather}\label{eq:singular-demazure}
\partial_{\stuple}^{\stuple[J]}\colon\rring[{\stuple}]\to\rring[{\stuple[J]}],\;
f\mapsto\partial_{w_{\stuple}w_{\stuple[J]}^{-1}}(f)
\end{gather}
and is of $\qpar$-degree $2\ell(\stuple)-2\ell(\stuple[J])$.
Here $w_{\stuple}$ is the longest element in $\psymgroup{N}{\stuple}$, and
$\ell(\stuple)$ denotes its length.

The Frobenius extension data allows for the definition of maps 
between certain $\rring[{\stuple}]$-bimodules, 
that will serve as important morphisms between singular Soergel bimodules 
(we recall the definition of the latter below).
To wit, given a basis $\{a_{i}\}$ for $\rring[{\stuple}]$ over $\rring[{\stuple[J]}]$, 
we can find a dual basis $\{a_{i}^{\star}\}$ satisfying 
$\partial_{\stuple}^{\stuple[J]}(a_{i}a_{j}^{\star})=\delta_{ij}$.
Given this, we obtain the Frobenius element 
$\frobel:={\textstyle\sum_{i}}\;a_{i}\otimes a_{i}^{\star}$, 
which is of $\qpar$-degree $2\ell(\stuple[J])-2\ell(\stuple)$ 
and independent of the choice of $\{a_{i}\}$. 
This gives multiplication and comultiplication maps
\begin{gather}\label{eq:singular-dot-maps}
\begin{aligned}
\dotdown{\stuple}{\stuple[J]}&\colon
\rring[{\stuple}]\otimes_{\rring[{\stuple[J]}]}\rring[{\stuple}]\to\rring[{\stuple}],
\;f\otimes g\mapsto fg,
\quad\text{$\qpar$-degree $0$},
\\
\dotup{\stuple}{\stuple[J]}&\colon
\rring[{\stuple}]\to
\rring[{\stuple}]\otimes_{\rring[{\stuple[J]}]}\rring[{\stuple}],
\; f\mapsto f\frobel,
\quad\text{$\qpar$-degree $2(\ell(\stuple[J])-\ell(\stuple))$}.
\end{aligned}
\end{gather}
These morphisms of bimodules are unital and counital with respect to 
$\iota_{\stuple}^{\stuple[J]}$ and $\partial_{\stuple}^{\stuple[J]}$, respectively.

\begin{example}\label{example:graded-frobenius}
For $\stuple=\emptyset$ and $\stuple[J]=\{\elstuff{i}\}$, 
we have $w_{\stuple}=1$ and $w_{\stuple[J]}=\coxgen$. 
It follows that
$\{1,\tfrac{1}{2}\rooty\}$ and $\{1^{\star}=\tfrac{1}{2}\rooty,(\tfrac{1}{2}\rooty)^{\star}=1\}$ 
are dual bases for $\rring$ as an $\rring[\{\elstuff{i}\}]$-module,  
and $\frobel=\tfrac{1}{2}(1\otimes\rooty+\rooty\otimes 1)$.
\end{example}

Finally, for later use, let us explicitly identify the rings $\rring[{\stuple}]$ 
for all $\stuple=(k_{1},\dots,k_{r})$. 
To this end, we consider $r$ alphabets $\ualph[i]$ (we tend to omit the alphabets 
if no confusion can arise) 
with $k_{i}$ variables, and write $\otimes_{\K}=\otimes$. 
A classical result about symmetric functions gives that
\begin{gather}\label{eq:el-sym}
\rring[{\stuple}]
\ouriso
\K\big[\efunc_{1}(\ualph[1]),\dots,\efunc_{k_{1}}(\ualph[1])\big]
\otimes
\dots
\otimes
\K\big[\efunc_{1}(\ualph[r]),\dots,\efunc_{k_{r}}(\ualph[r])\big],
\end{gather}
where $\efunc_{j}(\ualph[i])$ denotes the $j^{\mathrm{th}}$ elementary 
symmetric function in the variables $\ualph[i]$. Note that 
$\qdeg(\efunc_{j})=2j$. In particular,
\begin{gather}
\qdim(\rring[{\stuple}])
={
\prod_{j=1}^{r}} \;
{
\prod_{i=1}^{k_{j}}}\,
\tfrac{1}{1-\qpar^{2i}}.
\end{gather}

\subsection{A reminder on type \texorpdfstring{$\typeA$}{A} singular Soergel bimodules}\label{subsec:singular-primer}

We now briefly recall the category of singular
Soergel bimodules $\sSbim[N]=\sSbim[{\typeA[N{-}1]}]$ 
of type $\typeA[N{-}1]$, which categorifies the 
Hecke/Schur algebroid of type $\typeA$ 
\cite[Theorem 1.2]{Wi-sing-soergel} in characteristic $0$.
Details (in more generality) can be found 
\eg in \cite{Wi-sing-soergel}, or \cite{ElLo-modular-rep-typea} and 
\cite{ElSnWi-frob-extensions} for the underlying diagrammatic calculus.

Define the merge (``restriction'') and split (``induction'') bimodules as follows:
\begin{gather}\label{eq:indres}
\pmerge{\stuple}{\stuple[J]}
:=\qpar^{\ell(\stuple){-}\ell(\stuple[J])}\rring[{\stuple[J]}]\otimes_{\rring[{\stuple[J]}]}\rring[{\stuple}],
\quad\quad
\psplit{\stuple[J]}{\stuple}
:=\rring[{\stuple}]\otimes_{\rring[{\stuple[J]}]}\rring[{\stuple[J]}],
\end{gather}
where we follow the conventions from \eqref{eq:ParabolicConv}.
Here, we have indicated the left/right actions using left/right subscripts,
a convention that we will use throughout. 
There is a (horizontal) composition of such bimodules given by
tensoring over the common (``middle'') ring, 
which we denote \eg by $\pmerge{}{\stuple[L]}\pmerge{\stuple}{\stuple[J]}
=\pmerge{\stuple[J]}{\stuple[L]}\otimes_{\rring[{\stuple[J]}]}\pmerge{\stuple}{\stuple[J]}$.
In particular, we have the following 
$\qpar$-degree $0$ bimodule isomorphisms that we implicitly use below:
\begin{gather}\label{eq:asso-maps}
\pmerge{}{\stuple[L]}\pmerge{\stuple}{\stuple[J]}
\ouriso
\pmerge{\stuple}{\stuple[L]}
\ouriso
\pmerge{}{\stuple[L]}\pmerge{\stuple}{\stuple[K]},
\quad\quad
\psplit{}{\stuple}\psplit{\stuple[L]}{\stuple[J]}
\ouriso
\psplit{\stuple[L]}{\stuple}
\ouriso
\psplit{}{\stuple}\psplit{\stuple[L]}{\stuple[K]}.
\end{gather}
All of the isomorphism in \eqref{eq:asso-maps} are essentially identities, 
as the careful reader is invited to check.
(Note \eg that 
$f\otimes g \otimes h= 1\otimes 1\otimes fgh
\in\pmerge{}{\stuple[L]}\pmerge{\stuple}{\stuple[J]}$.)

\begin{definition}\label{definition:singsoergel-and-bs}
Let $\sSbim[N]$ be 
the $\K$-linear, $\qpar$-graded $2$-category 
given as the additive Karoubi $2$-closure 
(meaning taking direct sums and summands) of the 
$2$-category where objects are parabolic subgroups $\stuple\subset\symgroup[N]$,
$1$-morphisms are generated 
by $\qpar$-shifts of
\begin{gather}\label{eq:SBimGen}
\rring[{\stuple}]\colon\stuple\to\stuple,\quad
\pmerge{\stuple}{\stuple[J]}\colon\stuple\to\stuple[J],
\quad\text{and}\quad
\psplit{\stuple[J]}{\stuple}\colon\stuple[J]\to\stuple
\end{gather} 
for $\stuple[I]\subset\stuple[J]$, 
and $2$-morphisms are (all) bimodule maps of $\qpar$-degree $0$.
\end{definition}

\begin{example}\label{example:usual-soergel}
We have 
$\qpar^{\tm 1}\rring\otimes_{\rring[\{\elstuff{i}\}]}\rring
\ouriso\psplit{2}{1,1}\pmerge{1,1}{}$, 
which the reader familiar with 
(usual) Soergel bimodules of type $\typeA$ (see \eg \cite{ElWi-hodge-sbim}) might recognize 
as being so-called Bott--Samelson bimodules. 
In particular, Soergel bimodules of type $\typeA[{N{-}1}]$ can be identified with the 
$2$-category $\twoEnd_{\sSbim[N]}(\emptyset)$, which has just one object
(hence is a monoidal category).
\end{example}

\subsection{Web diagrammatics}\label{subsec:singular-diagrams}

Following \eg ideas in \cite[Section 3]{MaStVa-colored-homfly},
the generating $1$-morphisms from \eqref{eq:SBimGen}, and compositions thereof, 
admit a description in terms of an MOY-type calculus, 
which we now sketch.

The basic building blocks are the identity, merge, and split 
bimodules, 
which are depicted using the following (local) graphical notation:
\begin{gather}
\begin{tikzpicture}[anchorbase,scale=.5,tinynodes]
	\draw[usualbs] (1,0) node[below]{$k_{i}$}
	 to[out=90,in=270] (1,1) node[above,yshift=-2pt]{$k_{i}$};
\end{tikzpicture}
\leftrightsquigarrow
\rring[k_{i}],
\quad\quad
\begin{tikzpicture}[anchorbase,scale=.5,tinynodes]
	\draw[usualbs] (0,0) node[below]{$k$} to[out=90,in=180] (.5,.5);
	\draw[usualbs] (1,0) node[below]{$l$} to[out=90,in=0] (.5,.5);
	\draw[usualbs] (.5,.5) to[out=90,in=270] (.5,1) node[above,yshift=-2pt]{$k\tp l$};
\end{tikzpicture}
\leftrightsquigarrow
\pmerge{k,l}{k{+}l},
\quad\quad
\begin{tikzpicture}[anchorbase,scale=.5,tinynodes]
	\draw[usualbs] (.5,.5) to[out=180,in=270] (0,1) node[above,yshift=-2pt]{$k$};
	\draw[usualbs] (.5,.5) to[out=0,in=270] (1,1) node[above,yshift=-2pt]{$l$};
	\draw[usualbs] (.5,0) node[below]{$k\tp l$} to[out=90,in=270] (.5,.5);
\end{tikzpicture}
\leftrightsquigarrow
\psplit{k{+}l}{k,l}.
\end{gather}
Here, and for the duration, 
we use the abbreviation $\rring[k]$ for the ring associated to $\stuple=(k)$.

Recall that, by \fullref{convention:diagram-conventions}, 
vertical concatenation of such pictures corresponds to composition of $1$-morphisms, 
and \eg composing on the left corresponds to stacking on the top.
Moreover, we can place such diagrams side-by-side, 
which corresponds to taking the tensor product over $\K$.
Hence, we can associate a singular Soergel bimodule to
each trivalent graph that we can build from these diagrams via 
these operations.
(Note that webs corresponding to singular Soergel bimodules 
never have edges of negative label, 
but we will allow them in formulas
for convenience of notation, 
with the understanding that the corresponding bimodules are zero.)

\begin{example}\label{example:crossing-singular}
For $N=2$,
the standard way to depict the Soergel bimodule
from \fullref{example:usual-soergel}
(see \eg \cite[Figure 2]{Kh-homfly-soergel})
is built into our conventions:
\begin{gather}
\begin{tikzpicture}[anchorbase,scale=.5,tinynodes]
	\draw[usualbs] (0,0) node[below]{$1$} to[out=90,in=180] (.5,.5);
	\draw[usualbs] (1,0) node[below]{$1$} to[out=90,in=0] (.5,.5);
	\draw[usualbs] (.5,.5) to[out=90,in=270] (.5,1);
	\draw[usualbs] (.5,1) to[out=180,in=270] (0,1.5) node[above,yshift=-2pt]{$1$};
	\draw[usualbs] (.5,1) to[out=0,in=270] (1,1.5) node[above,yshift=-2pt]{$1$};
\end{tikzpicture}
:=
\begin{tikzpicture}[anchorbase,scale=.5,tinynodes]
	\draw[usualbs] (.5,.5) to[out=180,in=270] (0,1) node[above,yshift=-2pt]{$1$};
	\draw[usualbs] (.5,.5) to[out=0,in=270] (1,1) node[above,yshift=-2pt]{$1$};
	\draw[usualbs] (.5,0) node[below]{$2$} to[out=90,in=270] (.5,.5);
\end{tikzpicture}
\otimes_{\rring[2]}
\begin{tikzpicture}[anchorbase,scale=.5,tinynodes]
	\draw[usualbs] (0,0) node[below]{$1$} to[out=90,in=180] (.5,.5);
	\draw[usualbs] (1,0) node[below]{$1$} to[out=90,in=0] (.5,.5);
	\draw[usualbs] (.5,.5) to[out=90,in=270] (.5,1) node[above,yshift=-2pt]{$2$};
\end{tikzpicture}
\leftrightsquigarrow
\qpar^{\tm 1}\rring\otimes_{\rring[\{\elstuff{i}\}]}\rring\ouriso
\psplit{2}{1,1}\pmerge{1,1}{}.
\end{gather}
Also of importance will be the ladder-rung bimodules:
\begin{gather}
\begin{tikzpicture}[anchorbase,scale=.5,tinynodes]
	\draw[usualbs] (0,0) node[below]{$k$} to (0,1) node[above,yshift=-2pt]{\phantom{k}};
	\draw[usualbs] (1,0) node[below]{$l$} to (1,1);
	\draw[usualbs] (1,.35) to (.5,.5) node[above,yshift=-2pt]{$a$} to (0,.65);
\end{tikzpicture}
:=
\begin{tikzpicture}[anchorbase,scale=.5,tinynodes]
	\draw[usualbs] (-.5,.5) to[out=0,in=270] (0,1) to (0,2);
	\draw[usualbs] (-.5,.5) to[out=180,in=270] (-1,1);
	\draw[usualbs] (-.5,0) node[below]{$l$} to[out=90,in=270] (-.5,.5);
	\draw[usualbs] (-1,1) to[out=90,in=0] (-1.5,1.5);
	\draw[usualbs] (-2,0) node[below]{$k$} to (-2,1) to[out=90,in=180] (-1.5,1.5);
	\draw[usualbs] (-1.5,1.5) to[out=90,in=270] (-1.5,2) node[above,yshift=-2pt]{\phantom{k}};
	\node at (-.75,.9) {$a$};
\end{tikzpicture}
\quad\&\quad
\begin{tikzpicture}[anchorbase,scale=.5,tinynodes]
	\draw[usualbs] (0,0) node[below]{$k$} to (0,1) node[above,yshift=-2pt]{\phantom{k}};
	\draw[usualbs] (1,0) node[below]{$l$} to (1,1);
	\draw[usualbs] (0,.35) to (.5,.5) node[above,yshift=-2pt]{$a$} to (1,.65);
\end{tikzpicture}
:=
\begin{tikzpicture}[anchorbase,scale=.5,tinynodes]
	\draw[usualbs] (.5,.5) to[out=180,in=270] (0,1) to (0,2) node[above,yshift=-2pt]{\phantom{k}};
	\draw[usualbs] (.5,.5) to[out=0,in=270] (1,1);
	\draw[usualbs] (.5,0) node[below]{$k$} to[out=90,in=270] (.5,.5);
	\draw[usualbs] (1,1) to[out=90,in=180] (1.5,1.5);
	\draw[usualbs] (2,0) node[below]{$l$} to (2,1) to[out=90,in=0] (1.5,1.5);
	\draw[usualbs] (1.5,1.5) to[out=90,in=270] (1.5,2);
	\node at (.75,1.1) {$a$};
\end{tikzpicture}
\end{gather}
that will be used to build the square bimodules appearing in the complexes in \eqref{eq:colored-rouquier} below.
\end{example}

\begin{example}\label{example:assomaps-singular}
There exist $\qpar$-degree $0$ bimodule isomorphisms
\begin{gather}\label{eq:asso-bimodules}
\begin{tikzpicture}[anchorbase,scale=.5,tinynodes]
	\draw[usualbs] (.5,-.5) to[out=180,in=90] (0,-1);
	\draw[usualbs] (.5,-.5) to[out=0,in=90] (1,-1);
	\draw[usualbs] (.5,0) node[above,yshift=-2pt]{$k\tp l\tp m$} to[out=270,in=90] (.5,-.5);
	\draw[usualbs] (0,-1) to[out=270,in=90] (0,-2) node[below]{$k$};
	\draw[usualbs] (1,-1.5) to[out=180,in=90] (.5,-2) node[below]{$l$};
	\draw[usualbs] (1,-1.5) to[out=0,in=90] (1.5,-2) node[below]{$m$};
	\draw[usualbs] (1,-1) to[out=270,in=90] (1,-1.5);
\end{tikzpicture}
\ouriso
\begin{tikzpicture}[anchorbase,scale=.5,tinynodes]
	\draw[usualbs] (-.5,-.5) to[out=0,in=90] (0,-1);
	\draw[usualbs] (-.5,-.5) to[out=180,in=90] (-1,-1);
	\draw[usualbs] (-.5,0) node[above,yshift=-2pt]{$k\tp l\tp m$} to[out=270,in=90] (-.5,-.5);
	\draw[usualbs] (0,-1) to[out=270,in=90] (0,-2) node[below]{$m$};
	\draw[usualbs] (-1,-1.5) to[out=0,in=90] (-.5,-2) node[below]{$l$};
	\draw[usualbs] (-1,-1.5) to[out=180,in=90] (-1.5,-2) node[below]{$k$};
	\draw[usualbs] (-1,-1) to[out=270,in=90] (-1,-1.5);
\end{tikzpicture}
\quad\&\quad
\begin{tikzpicture}[anchorbase,scale=.5,tinynodes]
	\draw[usualbs] (.5,.5) to[out=180,in=270] (0,1);
	\draw[usualbs] (.5,.5) to[out=0,in=270] (1,1);
	\draw[usualbs] (.5,0) node[below]{$k\tp l\tp m$} to[out=90,in=270] (.5,.5);
	\draw[usualbs] (0,1) to[out=90,in=270] (0,2) node[above,yshift=-2pt]{$k$};
	\draw[usualbs] (1,1.5) to[out=180,in=270] (.5,2) node[above,yshift=-2pt]{$l$};
	\draw[usualbs] (1,1.5) to[out=0,in=270] (1.5,2) node[above,yshift=-2pt]{$m$};
	\draw[usualbs] (1,1) to[out=90,in=270] (1,1.5);
\end{tikzpicture}
\ouriso
\begin{tikzpicture}[anchorbase,scale=.5,tinynodes]
	\draw[usualbs] (-.5,.5) to[out=0,in=270] (0,1);
	\draw[usualbs] (-.5,.5) to[out=180,in=270] (-1,1);
	\draw[usualbs] (-.5,0) node[below]{$k\tp l\tp m$} to[out=90,in=270] (-.5,.5);
	\draw[usualbs] (0,1) to[out=90,in=270] (0,2) node[above,yshift=-2pt]{$m$};
	\draw[usualbs] (-1,1.5) to[out=0,in=270] (-.5,2) node[above,yshift=-2pt]{$l$};
	\draw[usualbs] (-1,1.5) to[out=180,in=270] (-1.5,2) node[above,yshift=-2pt]{$k$};
	\draw[usualbs] (-1,1) to[out=90,in=270] (-1,1.5);
\end{tikzpicture}
\end{gather}
that follow from the isomorphisms in \eqref{eq:asso-maps}.
Hence, we can unambiguously write
\begin{gather}
\begin{tikzpicture}[anchorbase,scale=.5,tinynodes]
	\draw[usualbs] (0,0) node[below]{$k_{1}$} to[out=90,in=180] (.5,.5);
	\draw[usualbs] (1,0) node[below]{$k_{r}$} to[out=90,in=0] (.5,.5);
	\draw[usualbs] (.5,.5) to[out=90,in=270] (.5,1) node[above,yshift=-2pt]{$k_{1}\tp\dots\tp k_{r}$};
	\node at (.5,.1) {...};
\end{tikzpicture}
\quad\&\quad
\begin{tikzpicture}[anchorbase,scale=.5,tinynodes]
	\draw[usualbs] (.5,.5) to[out=180,in=270] (0,1) node[above,yshift=-2pt]{$k_{1}$};
	\draw[usualbs] (.5,.5) to[out=0,in=270] (1,1) node[above,yshift=-2pt]{$k_{r}$};
	\draw[usualbs] (.5,0) node[below]{$k_{1}\tp\dots\tp k_{r}$} to[out=90,in=270] (.5,.5);
	\node at (.5,1) {...};
\end{tikzpicture}
\end{gather}
\end{example}

\subsection{Rickard--Rouquier complexes}\label{subsec:genusg-complex}

\begin{definition}\label{definition:homcat}
Given an additive category $\catstuff{C}$, 
we denote its bounded homotopy category by
$\hcat{\catstuff{C}}$.
This is the category 
whose objects are bounded chain complexes, and whose morphisms are 
homotopy classes of chain maps. 
We will use $\ouriso[\simeq]$ to denote isomorphisms in $\hcat{\catstuff{C}}$, 
\ie homotopy equivalence.
\end{definition}

Recalling \fullref{subsection:intro-conventions}, we can view the 
objects in $\hcat{\catstuff{C}}$ as finite direct sums 
$\bigoplus_{i}\tpar^{k_{i}}\obstuff{X}_{i}$, 
equipped with a differential $d$ with $\qdeg[\tpar](d)=-1$.
There is a $\tpar$-degree zero inclusion of categories
$\catstuff{C}\hookrightarrow\hcat{\catstuff{C}}$ given
by considering objects of $\catstuff{C}$ as one-term complexes 
concentrated in $\tpar$-degree $0$.
We also remark that we can consider $\hcat{\twocatstuff{C}}$ for a 
$2$-category $\twocatstuff{C}$, 
by passing to the homotopy category in each $\Hom$-category. 
In particular, if $\catstuff{C}$ is monoidal, then so is $\hcat{\catstuff{C}}$.

We now recall Rickard--Rouquier complexes,
\ie complexes of singular Soergel bimodules that determine maps from the 
(colored) braid group(oid) into certain $\Hom$-categories in $\hcat{\sSbim[N]}$.
Our terminology here arises as these complexes correspond to the Rickard complexes 
(originally defined for symmetric groups)
in categorified quantum groups, 
but also agree with the type $\typeA$ Rouquier complexes in the ``uncolored'' $k=l=1$ case.

They are given as follows:
\begin{gather}\label{eq:colored-rouquier}
\begin{aligned}
\begin{tikzpicture}[anchorbase,scale=.5,tinynodes]
	\draw[usualbs,crossline] (1,0) node[below]{$l$} to[out=90,in=270] (0,1.5) node[above,yshift=-2pt]{$l$};
	\draw[usualbs,crossline] (0,0) node[below]{$k$} to[out=90,in=270] (1,1.5) node[above,yshift=-2pt]{$k$};
\end{tikzpicture}
&:=
\begin{tikzpicture}[anchorbase,scale=.5,tinynodes]
	\draw[usualbs] (0,0) node[below]{$k$} to (0,2) node[above,yshift=-2pt]{$l$};
	\draw[usualbs] (1,0) node[below]{$l$} to (1,.35) to (.5,.5) to (0,.65);
	\draw[usualbs] (0,1.35) to (.5,1.5) to (1,1.65) to (1,2) node[above,yshift=-2pt]{$k$};
\end{tikzpicture}
\xrightarrow{d_{0}^{+}}
\tpar\qpar^{\tm 1}
\begin{tikzpicture}[anchorbase,scale=.5,tinynodes]
	\draw[usualbs] (0,0) node[below]{$k$} to (0,2) node[above,yshift=-2pt]{$l$};
	\draw[usualbs] (1,0) node[below]{$l$} to node[right,xshift=-2pt]{$1$} (1,2) node[above,yshift=-2pt]{$k$};
	\draw[usualbs] (1,.35) to (.5,.5)  to (0,.65);
	\draw[usualbs] (0,1.35) to (.5,1.5) to (1,1.65);
\end{tikzpicture}
\xrightarrow{d_{1}^{+}}
\dots
\xrightarrow{d_{m{-}1}^{+}}
\tpar^{m}\qpar^{\tm m}
\begin{tikzpicture}[anchorbase,scale=.5,tinynodes]
	\draw[usualbs] (0,0) node[below]{$k$} to (0,2) node[above,yshift=-2pt]{$l$};
	\draw[usualbs] (1,0) node[below]{$l$} to node[right,xshift=-2pt]{$m$} (1,2) node[above,yshift=-2pt]{$k$};
	\draw[usualbs] (1,.35) to (.5,.5)  to (0,.65);
	\draw[usualbs] (0,1.35) to (.5,1.5) to (1,1.65);
\end{tikzpicture}
\\
\begin{tikzpicture}[anchorbase,scale=.5,tinynodes]
	\draw[usualbs,crossline] (0,0) node[below]{$k$} to[out=90,in=270] (1,1.5) node[above,yshift=-2pt]{$k$};
	\draw[usualbs,crossline] (1,0) node[below]{$l$} to[out=90,in=270] (0,1.5) node[above,yshift=-2pt]{$l$};
\end{tikzpicture}
&:=
\tpar^{\tm m}\qpar^{m}
\begin{tikzpicture}[anchorbase,scale=.5,tinynodes]
	\draw[usualbs] (0,0) node[below]{$k$} to (0,2) node[above,yshift=-2pt]{$l$};
	\draw[usualbs] (1,0) node[below]{$l$} to node[right,xshift=-2pt]{$m$} (1,2) node[above,yshift=-2pt]{$k$};
	\draw[usualbs] (1,.35) to (.5,.5)  to (0,.65);
	\draw[usualbs] (0,1.35) to (.5,1.5) to (1,1.65);
\end{tikzpicture}
\xrightarrow{d_{m{-}1}^{-}}
\dots
\xrightarrow{d_{1}^{-}}
\tpar^{\tm 1}\qpar
\begin{tikzpicture}[anchorbase,scale=.5,tinynodes]
	\draw[usualbs] (0,0) node[below]{$k$} to (0,2) node[above,yshift=-2pt]{$l$};
	\draw[usualbs] (1,0) node[below]{$l$} to node[right,xshift=-2pt]{$1$} (1,2) node[above,yshift=-2pt]{$k$};
	\draw[usualbs] (1,.35) to (.5,.5)  to (0,.65);
	\draw[usualbs] (0,1.35) to (.5,1.5) to (1,1.65);
\end{tikzpicture}
\xrightarrow{d_{0}^{-}}
\begin{tikzpicture}[anchorbase,scale=.5,tinynodes]
	\draw[usualbs] (0,0) node[below]{$k$} to (0,2) node[above,yshift=-2pt]{$l$};
	\draw[usualbs] (1,0) node[below]{$l$} to (1,.35) to (.5,.5) to (0,.65);
	\draw[usualbs] (0,1.35) to (.5,1.5) to (1,1.65) to (1,2) node[above,yshift=-2pt]{$k$};
\end{tikzpicture}
\end{aligned}
\end{gather}
where $m=\min(k,l)$. 
Our notation denotes \eg that, as a $\tpar\qpar$-graded 
bimodule, $\khbracket{\atgen}{k,l}$ is the direct sum of the indicated terms, 
and the arrows depict the non-zero components of the differentials. 
Recalling the bimodule maps from \eqref{eq:asso-maps}, 
\eqref{eq:singular-iota}, \eqref{eq:singular-demazure}, \eqref{eq:singular-dot-maps}, and omitting 
the $\tpar\qpar$-shifts, these are given by
\begin{gather}\label{eq:diff-rouquier}
d_{i}^{+}\colon
\hspace*{-.25cm}
\begin{tikzpicture}[anchorbase]
	\matrix (m) [matrix of math nodes, ampersand replacement=\&, row sep=2em, column
	sep=3.0em, text height=1.8ex, text depth=0.25ex]{
	\begin{tikzpicture}[anchorbase,scale=.5,tinynodes]
	\draw[usualbs] (0,0) node[below]{$k$} to (0,2.5) node[above,yshift=-2pt]{$l$};
	\draw[usualbs] (1,0) node[below]{$l$} to (1,2.5) node[above,yshift=-2pt]{$k$};
	\draw[usualbs] (1,.35) to (.5,.5) to (0,.65);
	\draw[usualbs] (0,1.85) to (.5,2.0) to (1,2.15);
	\node at (1.5,1.2) {$i\phantom{\tp 1}$};
	\end{tikzpicture}  
	\& 
	\begin{tikzpicture}[anchorbase,scale=.5,tinynodes]
	\draw[usualbs] (0,0) node[below]{$k$} to (0,2.5) node[above,yshift=-2pt]{$l$};
	\draw[usualbs] (1,0) node[below]{$l$} to (1,2.5) node[above,yshift=-2pt]{$k$};
	\draw[usualbs] (1,.35) to (.5,.5) to (0,.65);
	\draw[usualbs] (.75,.4) to[out=90,in=0] (.5,.75) node[above,yshift=-3pt]{$1$} to[out=180,in=90] (.25,.6);
	\draw[usualbs] (0,1.85) to (.5,2.0) to (1,2.15);
	\draw[usualbs] (.75,2.1) to[out=270,in=0] (.5,1.75) node[below,yshift=1pt]{$1$} to[out=180,in=270] (.25,1.9);
	\node at (1.5,1.2) {$i\phantom{\tp 1}$};
	\end{tikzpicture}
	\&
	\begin{tikzpicture}[anchorbase,scale=.5,tinynodes]
	\draw[usualbs] (0,0) node[below]{$k$} to (0,2.5) node[above,yshift=-2pt]{$l$};
	\draw[usualbs] (1,0) node[below]{$l$} to (1,2.5) node[above,yshift=-2pt]{$k$};
	\draw[usualbs] (1,.35) to (.5,.5) to (0,.65);
	\draw[usualbs] (1,.55) to (.5,.7) node[above,yshift=-2pt]{$1$} to (0,.85);
	\draw[usualbs] (0,1.85) to (.5,2.0) to (1,2.15);
	\draw[usualbs] (0,1.65) to (.5,1.8) node[below]{$1$} to (1,1.95);
	\node at (1.5,1.2) {$i\phantom{\tp 1}$};
	\end{tikzpicture}
	\&
	\begin{tikzpicture}[anchorbase,scale=.5,tinynodes]
	\draw[usualbs] (0,0) node[below]{$k$} to (0,2.5) node[above,yshift=-2pt]{$l$};
	\draw[usualbs] (1,0) node[below]{$l$} to (1,2.5) node[above,yshift=-2pt]{$k$};
	\draw[usualbs] (1,.35) to (.5,.5) to (0,.65);
	\draw[usualbs] (0,1.85) to (.5,2.0) to (1,2.15);
	\draw[usualbs] (1,.55) to[out=135,in=270] (.75,1.25) node[left,xshift=2pt]{$1$} to[out=90,in=225] (1,1.95);
	\node at (1.5,1.2) {$i\phantom{\tp 1}$};
	\end{tikzpicture}
	\&
	\begin{tikzpicture}[anchorbase,scale=.5,tinynodes]
	\draw[usualbs] (0,0) node[below]{$k$} to (0,2.5) node[above,yshift=-2pt]{$l$};
	\draw[usualbs] (1,0) node[below]{$l$} to (1,2.5) node[above,yshift=-2pt]{$k$};
	\draw[usualbs] (1,.35) to (.5,.5) to (0,.65);
	\draw[usualbs] (0,1.85) to (.5,2.0) to (1,2.15);
	\node at (1.5,1.2) {$i\tp 1$};
	\end{tikzpicture}
	\\};
	\path[->] ($(m-1-1) + (.3,.1)$) edge node[above]{$\partial$} ($(m-1-2) + (-.55,.1)$);
	\path[->] ($(m-1-2) + (-.55,-.1)$) edge node[below]{$\iota$} ($(m-1-1) + (.3,-.1)$);
	\path[->] ($(m-1-2) + (.3,.1)$) edge node[above]{\eqref{eq:asso-maps}} ($(m-1-3) + (-.55,.1)$);
	\path[->] ($(m-1-3) + (-.55,-.1)$) edge node[below]{\eqref{eq:asso-maps}} ($(m-1-2) + (.3,-.1)$);
	\path[->] ($(m-1-3) + (.3,.1)$) edge node[above]{$\dotdown{}{}$} ($(m-1-4) + (-.55,.1)$);
	\path[->] ($(m-1-4) + (-.55,-.1)$) edge node[below]{$\dotup{}{}$} ($(m-1-3) + (.3,-.1)$);
	\path[->] ($(m-1-4) + (.3,.1)$) edge node[above]{$\iota$} ($(m-1-5) + (-.55,.1)$);
	\path[->] ($(m-1-5) + (-.55,-.1)$) edge node[below]{$\partial$} ($(m-1-4) + (.3,-.1)$);
	\node at (1,1) {\phantom{.}};
	\node at (1,-1) {\phantom{.}};
\end{tikzpicture}
\hspace*{-.25cm}
\colon d_{i}^{-}
\end{gather}
Here the corresponding parabolic subsets, which determine the 
bimodule maps, can be read from the 
indicated sequence of webs, and we use \eg $\iota_{\stuple}^{\stuple[J]}$ as
\begin{gather}
\iota_{\stuple}^{\stuple[J]}\colon
\rring[{\stuple[J]}]
\ouriso
\rring[{\stuple[J]}]\otimes_{\rring[{\stuple[J]}]}\rring[{\stuple[J]}]\otimes_{\rring[{\stuple[J]}]}\rring[{\stuple[J]}]
\hookrightarrow
\rring[{\stuple[J]}]\otimes_{\rring[{\stuple[J]}]}\rring[{\stuple}]\otimes_{\rring[{\stuple[J]}]}\rring[{\stuple[J]}]
=\pmerge{\stuple}{\stuple[J]}\psplit{\stuple[J]}{}.
\end{gather}

\begin{remark}\label{remark:foams}
We note that the differential in the Rickard--Rouquier complexes can be described diagrammatically 
using type $\typeA$ singular Soergel calculus, see \eg \cite[Section 2]{ElLo-modular-rep-typea}. 
Alternatively, we could work with the $n\to\infty$ limit of the $2$-category of
$\mathfrak{gl}_{n}$ foams to describe these $2$-morphisms in $\sSbim[N]$ (here, $n$ is a parameter independent of $N$).
In fact, these two descriptions are equivalent, 
as the type $\typeA$ singular Soergel calculus corresponds to the ``calculus of seams'' in the foam framework.
(See \eg \cite[Section 5.2]{QuRoSa-annular-evaluation} for a precise statement.)
\end{remark}

Finally, the fact that these indeed are complexes follows \eg by comparing 
\eqref{eq:colored-rouquier} to the Rickard complex in the categorified quantum group, 
as in \fullref{remark:foams}.

\begin{example}\label{example:uncolored}
In the uncolored case $k=l=1$ the complexes are
\begin{gather}
\begin{tikzpicture}[anchorbase,scale=.5,tinynodes]
	\draw[usualbs,crossline] (1,0) node[below]{$1$} to[out=90,in=270] (0,1.5) node[above,yshift=-2pt]{$1$};
	\draw[usualbs,crossline] (0,0) node[below]{$1$} to[out=90,in=270] (1,1.5) node[above,yshift=-2pt]{$1$};
\end{tikzpicture}
=
\begin{tikzpicture}[anchorbase,scale=.5,tinynodes]
	\draw[usualbs] (1,0) node[below]{$1$} to (1,.35) to (0,.65);
	\draw[usualbs] (0,1.35) to (0,2) node[above,yshift=-2pt]{$1$};
	\draw[usualbs] (0,0) node[below]{$1$} to (0,1.35) to (1,1.65) to (1,2) node[above,yshift=-2pt]{$1$};
\end{tikzpicture}
\xrightarrow{
\dotdown{1,1}{2}
}
\tpar\qpar^{\tm 1}
\begin{tikzpicture}[anchorbase,scale=.5,tinynodes]
	\draw[usualbs] (1,0) node[below]{$1$} to (1,2) node[above,yshift=-2pt]{$1$};
	\draw[usualbs] (0,0) node[below]{$1$} to (0,2) node[above,yshift=-2pt]{$1$};
\end{tikzpicture}
\quad\&\quad
\begin{tikzpicture}[anchorbase,scale=.5,tinynodes]
	\draw[usualbs,crossline] (0,0) node[below]{$1$} to[out=90,in=270] (1,1.5) node[above,yshift=-2pt]{$1$};
	\draw[usualbs,crossline] (1,0) node[below]{$1$} to[out=90,in=270] (0,1.5) node[above,yshift=-2pt]{$1$};
\end{tikzpicture}
=
\tpar^{\tm 1}\qpar
\begin{tikzpicture}[anchorbase,scale=.5,tinynodes]
	\draw[usualbs] (1,0) node[below]{$1$} to (1,2) node[above,yshift=-2pt]{$1$};
	\draw[usualbs] (0,0) node[below]{$1$} to (0,2) node[above,yshift=-2pt]{$1$};
\end{tikzpicture}
\xrightarrow{
\dotup{1,1}{2}
}
\begin{tikzpicture}[anchorbase,scale=.5,tinynodes]
	\draw[usualbs] (1,0) node[below]{$1$} to (1,.35) to (0,.65);
	\draw[usualbs] (0,1.35) to (0,2) node[above,yshift=-2pt]{$1$};
	\draw[usualbs] (0,0) node[below]{$1$} to (0,1.35) to (1,1.65) to (1,2) node[above,yshift=-2pt]{$1$};
\end{tikzpicture}
\end{gather}
\end{example}

\begin{remark}\label{remark:crossing-conventions}
The conventions in 
\fullref{example:uncolored} are the same as in \cite{Ro-cat-braid-group}, 
except that in that work, there is no shift on
the bimodule $\psplit{2}{1,1}\pmerge{1,1}{}$.
\end{remark}

\begin{example}\label{example:pitchfork}
There exist $\qpar$-degree $0$ isomorphisms in $\hcat{\sSbim[N]}$
\begin{gather}\label{eq:pitchfork}
\begin{tikzpicture}[anchorbase,scale=.5,tinynodes]
	\draw[usualbs] (0,0) node[below]{$l$} to[out=90,in=180] (.5,.5);
	\draw[usualbs] (1,0) node[below]{$m$} to[out=90,in=0] (.5,.5);
	\draw[usualbs] (.5,.5) to[out=90,in=270] (.5,1) to (.5,2) node[above,yshift=-2pt]{$l\tp m$};
	\draw[usualbs,crossline] (-1,0) node[below]{$k$} to[out=90,in=270] (2,2) node[above,yshift=-2pt]{$k$};
\end{tikzpicture}
\ouriso[\simeq]
\begin{tikzpicture}[anchorbase,scale=.5,tinynodes]
	\draw[usualbs] (0,0) node[below]{$l$} to (0,1) to[out=90,in=180] (.5,1.5);
	\draw[usualbs] (1,0) node[below]{$m$} to (1,1) to[out=90,in=0] (.5,1.5);
	\draw[usualbs] (.5,1.5) to[out=90,in=270] (.5,2) node[above,yshift=-2pt]{$l\tp m$};
	\draw[usualbs,crossline] (-1,0) node[below]{$k$} to[out=90,in=270] (2,2) node[above,yshift=-2pt]{$k$};
\end{tikzpicture}
\quad\&\quad
\begin{tikzpicture}[anchorbase,scale=.5,tinynodes]
	\draw[usualbs] (0,0) to[out=90,in=180] (.5,.5);
	\draw[usualbs] (1,0) to[out=90,in=0] (.5,.5);
	\draw[usualbs] (.5,.5) to[out=90,in=270] (.5,1) node[above,yshift=-2pt]{$k\tp l$};
	\draw[usualbs,crossline] (1,-1) node[below]{$l$} to[out=90,in=270] (0,0);
	\draw[usualbs,crossline] (0,-1) node[below]{$k$} to[out=90,in=270] (1,0);
\end{tikzpicture}
\ouriso[\simeq]
\qpar^{kl}
\begin{tikzpicture}[anchorbase,scale=.5,tinynodes]
	\draw[usualbs] (0,-1) node[below]{$k$} to (0,0) to[out=90,in=180] (.5,.5);
	\draw[usualbs] (1,-1) node[below]{$l$} to (1,0) to[out=90,in=0] (.5,.5);
	\draw[usualbs] (.5,.5) to[out=90,in=270] (.5,1) node[above,yshift=-2pt]{$k\tp l$};
\end{tikzpicture}
\quad\&\quad
\begin{tikzpicture}[anchorbase,scale=.5,tinynodes]
	\draw[usualbs] (0,0) to[out=90,in=180] (.5,.5);
	\draw[usualbs] (1,0) to[out=90,in=0] (.5,.5);
	\draw[usualbs] (.5,.5) to[out=90,in=270] (.5,1) node[above,yshift=-2pt]{$k\tp l$};
	\draw[usualbs,crossline] (0,-1) node[below]{$k$} to[out=90,in=270] (1,0);
	\draw[usualbs,crossline] (1,-1) node[below]{$l$} to[out=90,in=270] (0,0);
\end{tikzpicture}
\ouriso[\simeq]
\qpar^{\tm kl}
\begin{tikzpicture}[anchorbase,scale=.5,tinynodes]
	\draw[usualbs] (0,-1) node[below]{$k$} to (0,0) to[out=90,in=180] (.5,.5);
	\draw[usualbs] (1,-1) node[below]{$l$} to (1,0) to[out=90,in=0] (.5,.5);
	\draw[usualbs] (.5,.5) to[out=90,in=270] (.5,1) node[above,yshift=-2pt]{$k\tp l$};
\end{tikzpicture}
\end{gather}
as well as variants with analogous $\qpar$-shifts involving split bimodules.
\end{example}

Let $\stuple$ be a parabolic subgroup, 
and let $\num{\stuple}$ denote the number of 
entries in the corresponding tuple 
(\ie for an $r$-tuple $\stuple$, $\num{\stuple} =r$). 
Given a braid generator $\bgen\in\braidg[\num{\stuple}]$, 
we let $\khbracket{\bgen^{\pm1}}{\stuple}$ denote the complex 
given by placing appropriately labeled 
vertical strands next to the corresponding complex in \eqref{eq:colored-rouquier}, 
\ie by taking tensor product over $\K$ with 
the rings $\rring[(k_1,\ldots,k_{i-1})]$ and 
$\rring[(k_{i+2},\ldots,k_{\num{\stuple}})]$.

\begin{definition}\label{definition:colored-rouquier}
For $\stuple$ and $\braid\in\braidg[{\num{\stuple}}]$, 
fix an expression $\braid=\bgen[i_{1}]^{\pm1}\dots\bgen[i_{r}]^{\pm1}$.
Define 
\[
\khbracket{\braid}{\stuple} 
:= 
\khbracket{\bgen[i_{1}]^{\pm1}}{\stuple^{\prime}} 
\dots 
\khbracket{\bgen[i_{r}]^{\pm1}}{\stuple}
\]
where, on the right-hand side, we use composition in $\hcat{\sSbim[N]}$, 
\ie tensor product of the complexes of singular Soergel bimodules.
\end{definition}

By \eg the results in \cite[Section 5.2]{QuRoSa-annular-evaluation}, 
the complex $\khbracket{\braid}{\stuple}$ does not depend,
up to isomorphism,
on the choice of expression for $\braid$.
Thus, the assignment $\braid\mapsto\khbracket{\braid}{\stuple}$
gives an action
of $\braidg[\num{\stuple}]$ on  $\hcat{\sSbim[N]}$. We get:

\begin{proposition}\label{proposition:handlebody-rouquier}
There is an action of $\braidg[g,n]$ 
on $\hcat{\sSbim[N]}$ determined by 
the assignment $\braid\mapsto\khbracket{\braid}{\stuple}$.
\end{proposition}

\begin{proof}
By the discussion above, we have an action of the classical braid group. 
Composing this action with the map from 
\fullref{proposition:artintits-a-higher-g} gives the desired action of the handlebody braid group.
\end{proof}

\section{Colored HOMFLYPT homology for links in handlebodies}\label{section:homology}

In this section, we proceed to construct our
triply-graded invariant of links in $\hbody$, 
with \fullref{theorem:homfly} and \fullref{corollary:colored-homfly} being the main statements.
We keep the notation from the previous sections and begin with some preliminaries. 

\subsection{A reminder on Hochschild cohomology}\label{subsec:reminder-hochschild}

Let $\setstuff{A}$ be a $\qpar$-graded $\K$-algebra, 
and recall that we may regard any $\qpar$-graded 
$\setstuff{A}$-bimodule $\morstuff{B}$ 
as a $\qpar$-graded left module over the enveloping algebra $\ering{\setstuff{A}}$.
The Hochschild cohomology
of $\setstuff{A}$ with coefficients in $\morstuff{B}$
is the $\apar\qpar$-graded $\K$-vector space
\begin{gather}
\cHH(\setstuff{A},\morstuff{B})
:=
{
\bigoplus_{a\in\Z}}\,\cHH[a](\setstuff{A},\morstuff{B})
\end{gather}
with $a$-degree component defined by
\begin{gather}\label{eq:defhochschild}
\cHH[a](\setstuff{A},\morstuff{B}) 
:=
\EXT_{\ering{\setstuff{A}}}^{a}(\setstuff{A},\morstuff{B})
={
\bigoplus_{s\in\Z}}\,
\Ext_{\ering{\setstuff{A}}}^{a}(\qpar^{s}\setstuff{A},\morstuff{B}).
\end{gather}
(Compare our notation here to \fullref{convention:grading}.)

The relevant case for our considerations is when 
$\setstuff{A}=\rring[\stuple]=(\rring[\stuple])^{\mathrm{op}}$.
Here, for $\stuple=\emptyset$, 
Khovanov showed that the triply-graded link homology from \cite{KhRo-link-homologies-2} 
can be constructed using the Hochschild homology 
(defined using $\setstuff{Tor}$ instead of $\setstuff{Ext}$)
of Soergel bimodules; see \cite[Section 1.1]{Kh-homfly-soergel}.
Recall from \eqref{eq:el-sym} that $\rring[\stuple]$ is a polynomial ring, so
Hochschild homology and cohomology are isomorphic (up to a shift).
We work with the latter since 
\eg in this framework the invariant of the (colored) unknot inherits a natural algebra structure \cite{Ho-young-torus}, 
which is important for various considerations.

\begin{example}\label{example:hochschild}
Let $\stuple=(k_{1},\dots,k_{r})$.
Recall that $\rring[\stuple]$ is a 
polynomial ring (and, in particular, is Koszul). 
Hence, we can compute Hochschild cohomology using the Koszul resolution of $\rring[\stuple]$, 
which is the free resolution of $\rring[\stuple]$ as an 
$\rring[\stuple]$-bimodule given by
\begin{gather}\label{eq:koszul}
{
\bigotimes_{j=1}^{r}}\,
\left(
{
\bigotimes_{i=1}^{k_{j}}}\,
\big(
\hpar\qpar^{2i}
\eering{\rring[\stuple]}
\xrightarrow{\efunc_{i}\otimes 1-1\otimes\efunc_{i}}
\eering{\rring[\stuple]}
\big)
\right).
\end{gather}
Here $\hpar$ denotes a shift up 
in an auxiliary homological degree, 
and the outer tensor products are taken over $\eering{\rring[\stuple]}$.
Given a $\rring[\stuple]$-bimodule $\morstuff{B}$, 
taking the ``internal'' $\qpar$-graded $\Hom$ of complexes $\iHOM_{\eering{\rring[\stuple]}}(\placeholder,\morstuff{B})$
(\ie applying $\HOM_{\eering{\rring[\stuple]}}(\placeholder,\morstuff{B})$ 
to the terms and differentials of a chain complex to 
obtain a cochain complex)
gives a complex concentrated in non-negative cohomological degree $\apar$, 
which is the negative of the $\hpar$-degree.
The $a^{\mathrm{th}}$ cohomology of this 
complex is $\cHH[a](\rring[\stuple],\morstuff{B})$. 

Computing for $\morstuff{B}=\rring[\stuple]$ gives the following.
For each $j$, fix a set of variables $\{\theta_{i}\mid 1\leq i\leq k_{j}\}$ 
with $\qdeg[\apar\qpar](\theta_{i})=(1,-2i)$, and recall that $\qdeg[\apar\qpar](\efunc_{i})=(0,2i)$.
We then have an isomorphism of $\apar\qpar$-graded $\K$-vector spaces
\begin{gather}
\cHH(\rring[\stuple],\rring[\stuple])
\ouriso
{
\bigotimes_{j=1}^{r}}\,
\left(
\K[\efunc_{1},\dots,\efunc_{k_{j}}]\otimes\extalg\{\theta_{i}\mid 1\leq i\leq k_{j}\}
\right),
\end{gather}
where $\extalg\{\theta_{i}\mid 1\leq i\leq k_{j}\}$ denotes the exterior algebra.
\end{example}

Since Hochschild cohomology is functorial with respect to bimodule morphisms, 
we can apply $\cHH$ to a complex of $\rring[\stuple]$-bimodules term-wise to obtain a complex of 
$\apar\qpar$-graded $\K$-vector spaces. 
(In fact, since our ring is commutative, these $\K$-vectors spaces inherit an action of $\rring[\stuple]$, 
so can be thought of as trivial $\rring[\stuple]$-bimodules.)

In particular, 
let $\bimodq{\rring[\stuple]}$ denote the category of $\qpar$-graded, 
finitely-generated $\rring[\stuple]$-bimodules,
and let $\hcat{\bimodq{\rring[\stuple]}}$ be its homotopy category. 
We get a functor
\begin{gather}
\cHH_{\stuple}(\placeholder)
:=
{
\bigoplus_{a\in\Z}}\,\cHH[a]_{\stuple}(\placeholder)
\colon\hcat{\bimodq{\rring[\stuple]}}\to\hcat{\vecaq}
\end{gather}
whose $a$-degree component is the functor
\begin{gather}\label{eq:hochschild-functor}
\cHH[a]_{\stuple}(\placeholder):=\cHH[a](\rring[\stuple],\placeholder)
\colon\hcat{\bimodq{\rring[\stuple]}}\to\hcat{\vecq}.
\end{gather}

\subsection{Towards handlebody HOMFLYPT homology}\label{subsec:towards-homfly}

Fix integers 
$M,l_{1},\dots,l_{n}\in\N[\geq 1]$,
called the core and link colors, respectively. 
For any $g \geq 0$, these choices determine
a parabolic subset 
$\mtuple:=(M,\dots,M,l_{1},\dots,l_{n})$
with $\num{\mtuple}=g+n$.
We view $\mtuple$ as providing a coloring for
braids $\braid\in\braidg[g,n]$ as in \fullref{subsec:genusg-complex}, 
where strands are colored at the bottom by the entries of $\mtuple$.
We will call a colored braid $(\braid,\mtuple)$ balanced if the colors at the 
top and bottom of the $i^{\mathrm{th}}$ position agree for all $i$.
For the duration, we only consider balanced colorings and any 
braid or link will be colored by default.

\begin{example}\label{example:coloring-verma}
The prototypical example of a balanced coloring is the case where the link is uncolored, 
\ie where 
$l_{1}=\dots=l_{n}=1$ 
and $M$ is arbitrary. 
In general, $M$ should be viewed as being ``very large,'' \ie 
$M\gg l_{i}$ 
for all $i$;
compare \eg to \cite{IoLeZh-verma-schur-weyl}, 
where the core of the solid torus is colored by a Verma module.
\end{example}

\begin{remark}\label{remark:all-colors}
It is possible to work with 
any balanced coloring of $\braid\in\braidg[g,n]$. 
However, the core strands are not topologically distinguishable, 
hence should be colored uniformly.
\end{remark}

Consider $\ssbim[\stuple]:=\twoEnd_{\sSbim[N]}(\stuple)$
which is a $\qpar$-graded, full, monoidal subcategory of $\bimodq{\rring[\stuple]}$.
The monoidal structure is inherited from the horizontal composition in $\sSbim[N]$, 
\ie it is given by tensor product over $\rring[\stuple]$. 
We will occasionally denote this by $\otimes_{\rring[\stuple]}$, 
in addition to our previous notation for this operation, 
which was simply concatenation.

Recalling \fullref{example:assomaps-singular} and \fullref{proposition:handlebody-rouquier}, 
and motivated by \fullref{remark:stuck},
we define:

\begin{definition}\label{definition:coloring-merge-split}
For $\braid\in\braidg[g,n]$ and $(\braid,\mtuple)$ a balanced coloring, 
we let
\begin{gather}
\khbracket{\braid}{\hbody}
:=
\left(
\begin{tikzpicture}[anchorbase,scale=.5,tinynodes]
	\draw[polebs] (.5,1) to[out=180,in=270] (0,1.5) node[above,yshift=-2pt,infinity]{$M$};
	\draw[polebs] (.5,1) to[out=0,in=270] (1,1.5) node[above,yshift=-2pt,infinity]{$M$};
	\draw[polebs] (0,0) node[below,infinity]{$M$} to[out=90,in=180] (.5,.5);
	\draw[polebs] (1,0) node[below,infinity]{$M$} to[out=90,in=0] (.5,.5);
	\draw[polebs] (.5,.5) to[out=90,in=270] (.5,1);
	\draw[usualbs] (2,0) node[below]{$l_{1}$} to (2,1.5) node[above,yshift=-2pt]{$l_{1}$};
	\draw[usualbs] (3,0) node[below]{$l_{n}$} to (3,1.5) node[above,yshift=-2pt]{$l_{n}$};
	\node at (.5,.2) {...};
	\node at (.5,1.4) {...};
	\node at (2.5,.5) {...};
\end{tikzpicture}\right)
\otimes_{\rring[\mtuple]}
\khbracket{\braid}{\mtuple}
\in\hcat{\ssbim[\mtuple]}.
\end{gather}
\end{definition}

\begin{example}\label{example:coloring}
In the cases $g=0,1$ we have $\khbracket{\braid}{\hbody}=\khbracket{\braid}{\mtuple}$.
For $g=2$, we have
\begin{gather}
\braid=
\begin{tikzpicture}[anchorbase,scale=.425,smallnodes]
	\draw[pole,crosspole] (2,0) to[out=90,in=270] (2,2);
	\draw[usual,crossline] (3,0) to[out=90,in=0] (2.75,.625) to[out=180,in=0] (.75,.625) to[out=180,in=270] (.5,1);
	\draw[pole,crosspole] (1,0) to[out=90,in=270] (1,2);
	\draw[usual,crossline,directed=1] (.5,1) to[out=90,in=180] (.75,1.375) to[out=0,in=180] (2.75,1.375) to[out=0,in=270] (3,2);
\end{tikzpicture}
\quad\&\quad
\khbracket{\braid}{\mtuple}
=
\begin{tikzpicture}[anchorbase,scale=.425,smallnodes]
	\draw[polebs] (2,0) node[below,infinity]{$M$} to[out=90,in=270] (2,2) node[above,yshift=-2pt,infinity]{$M$};
	\draw[usualbs,crossline] (3,0) node[below]{$l$} to[out=90,in=0] (2.75,.625) to[out=180,in=0] (.75,.625) to[out=180,in=270] (.5,1);
	\draw[polebs,crossline] (1,0) node[below,infinity]{$M$} to[out=90,in=270] (1,2) node[above,yshift=-2pt,infinity]{$M$};
	\draw[usualbs,crossline] (.5,1) to[out=90,in=180] (.75,1.375) to[out=0,in=180] (2.75,1.375) to[out=0,in=270] (3,2) node[above,yshift=-2pt]{$l$};
\end{tikzpicture}
\quad\&\quad
\khbracket{\braid}{\hbody[2]}=
\begin{tikzpicture}[anchorbase,scale=.425,smallnodes]
	\draw[polebs] (2,0) node[below,infinity]{$M$} to[out=90,in=270] (2,2);
	\draw[usualbs,crossline] (3,0) node[below]{$l$} to[out=90,in=0] (2.75,.625) to[out=180,in=0] (.75,.625) to[out=180,in=270] (.5,1);
	\draw[polebs,crossline] (1,0) node[below,infinity]{$M$} to[out=90,in=270] (1,2);
	\draw[usualbs,crossline] (.5,1) to[out=90,in=180] (.75,1.375) to[out=0,in=180] (2.75,1.375) to[out=0,in=270] (3,2);
	\draw[polebs] (1.5,3) to[out=180,in=270] (1,3.5) node[above,yshift=-2pt,infinity]{$M$};
	\draw[polebs] (1.5,3) to[out=0,in=270] (2,3.5) node[above,yshift=-2pt,infinity]{$M$};
	\draw[polebs] (1,2) to[out=90,in=180] (1.5,2.5);
	\draw[polebs] (2,2) to[out=90,in=0] (1.5,2.5);
	\draw[polebs] (1.5,2.5) to[out=90,in=270] (1.5,3);
	\draw[usualbs,crossline] (3,2) to (3,3.5) node[above,yshift=-2pt]{$l$};
\end{tikzpicture}
\end{gather}
For $g\geq 2$, we generally have
$\khbracket{\braid}{\hbody}\not\ouriso[\simeq]\khbracket{\braid}{\mtuple}$, \cf \fullref{proposition:links-get-stuck} below.
\end{example}

\begin{definition}\label{definition:towards-homfly}
For $\braid\in\braidg[g,n]$ and $(\braid,\mtuple)$ a balanced coloring we let
\begin{gather}
\cHH_{\hbody}(\braid)
:={
\bigoplus_{a\in\Z}}\,\cHH[a]\big(\khbracket{\braid}{\hbody}\big). 
\end{gather}
\end{definition}

By \fullref{proposition:handlebody-rouquier}, 
$\cHH_{\hbody}(\braid)$ is an invariant of 
the colored braid $\braid\in\braidg[g,n]$ taking values in $\hcat{\vecaq}$.

\subsection{Colored handlebody HOMFLYPT homology}\label{subsec:homfly}

In \eqref{eq:colored-homfly} below 
we use $\cHH_{\hbody}(\braid)$ to define $\cHHH{\braid}{\hbody}$, 
an invariant of the colored link $\bclosure\subset\hbody$, valued in $\vecatq$. 
To do so, we establish the following.

\begin{theorem}\label{theorem:homfly}
The assignment 
\begin{gather}
\braidg[g,n]\to\hcat{\vecaq},
\;
\braid\mapsto\cHH_{\hbody}(\braid)
\end{gather}
is invariant under the conjugation \eqref{eq:conjugation} 
and stabilization \eqref{eq:stabilization} 
relations for $\braidg[g,n]$, up to homotopy equivalence 
and grading normalization.
Moreover, it is not generally invariant under the classical 
conjugation relation \eqref{eq:classical-conjugation}
for $\braidg[g+n]$.
\end{theorem}

The remainder of this section constitutes a proof of this theorem.
Namely, invariance under conjugation holds as a 
special case of the corresponding result for 
colored, triply-graded link homology in $\thesphere$, 
and \fullref{lemma:skein-II} establishes invariance 
(up to a grading shift) under stabilization.
\fullref{proposition:links-get-stuck} shows the failure 
of invariance under the classical conjugation relation.
We stress the importance of this latter fact: any 
invariant of the classical braid group $\braidg[g+n]$, that 
additionally is invariant under classical conjugation and stabilization, 
gives rise to invariants of links in $\hbody$ using 
the inclusion $\braidg[g,n]\hookrightarrow\braidg[g+n]$.
However, such invariants are less-sensitive to 
the topology of $\hbody$, as our results show.

\begin{proposition}\label{proposition:links-get-stuck}
For $g>1$ and $n>0$, there exists handlebody 
braids $\braid,\braid^{\prime}\in\braidg[g,n]$ that are conjugate 
in $\braidg[g+n]$, but satisfy 
$\cHH_{\hbody}(\braid)\not\ouriso[\simeq]\cHH_{\hbody}(\braid^{\prime})$.
\end{proposition}

This shows that our handlebody homology, which is defined 
below to be the cohomology of a renormalization of $\cHH_{\hbody}(\placeholder)$, 
distinguishes these handlebody links, while the invariant obtained 
by including into $\braidg[g+n]$ and using the classical (colored) 
triply-graded homology does not.

\begin{proof}
It suffices to give an example, and we provide one in the $g=2$ and $n=1$ 
case that immediately generalizes to any $g\geq 2$ and $n\geq 1$.
Let $\braid=\tgen[2]\tgen[1]$ and $\braid^{\prime}=\tgen[1]\tgen[2]$, 
which are conjugate braids in $\braidg[g+n]$. 
We claim that $\cHH_{\hbody}(\braid)\not\ouriso[\simeq]\cHH_{\hbody}(\braid^{\prime})$, 
and exhibit this explicitly in the case that $M=1$.

Indeed, if they were homotopy equivalent, then the Euler characteristics (\ie alternating 
sums of $\apar\qpar$-graded dimensions) of these complexes would agree. 
However, since the category of (usual) type $\typeA$ Soergel bimodules categorifies the 
type $\typeA$ Hecke algebra, and Hochschild cohomology categorifies the 
Jones--Ocneanu trace, 
this would imply that the Jones--Ocneanu traces of the following braided, 
trivalent graphs agree:
\begin{gather}
\begin{tikzpicture}[anchorbase,scale=.5,tinynodes]
	\draw[usual,crossline] (1,0) to[out=90,in=270] (1,1.5);
	\draw[usual,crossline] (3,0) to[out=90,in=0] (2.75,.5) 
	to[out=180,in=0] (1.75,.5) to[out=180,in=270] (1.5,.75);
	\draw[usual,crossline] (2,0) to[out=90,in=270] (2,1.5);
	\draw[usual,crossline] (1.5,.75) to[out=90,in=180] (1.75,1) to[out=0,in=180] (2.75,1) to[out=0,in=270] (3,1.5);
	\draw[usual,crossline,directed=1] (2,1.5) to[out=90,in=270] (2,3);
	\draw[usual,crossline] (3,1.5) to[out=90,in=0] (2.75,2) 
	to[out=180,in=0] (.75,2) to[out=180,in=270] (.5,2.25);
	\draw[usual,crossline,directed=1] (1,1.5) to[out=90,in=270] (1,3);
	\draw[usual,crossline,directed=1] (.5,2.25) to[out=90,in=180] (.75,2.5) 
	to[out=0,in=180] (2.75,2.5) to[out=0,in=270] (3,3);
	\draw[usual] (1,4.5) to [out=270,in=180] (1.5,4);
	\draw[usual] (2,4.5) to [out=270,in=0] (1.5,4);
	\draw[usual] (1.5,4) to (1.5,3.5);	
	\draw[usual] (1,3) to [out=90,in=180] (1.5,3.5);
	\draw[usual] (2,3) to [out=90,in=0] (1.5,3.5);
	\draw[usual,crossline] (3,3) to (3,4.5);
	\draw[thick,mygray,fill=citron,fill opacity=.1] (0,3) 
	to (3.5,3) to (3.5,0) to (0,0) to (0,3);
	\node at (.5,1.5) {$\braid$};
\end{tikzpicture}
\quad\&\quad
\begin{tikzpicture}[anchorbase,scale=.5,tinynodes]
	\draw[usual,crossline] (2,0) to[out=90,in=270] (2,1.5);
	\draw[usual,crossline] (3,0) to[out=90,in=0] (2.75,.5) to[out=180,in=0] (.75,.5) to[out=180,in=270] (.5,.75);
	\draw[usual,crossline] (1,0) to[out=90,in=270] (1,1.5);
	\draw[usual,crossline] (.5,.75) to[out=90,in=180] (.75,1) 
	to[out=0,in=180] (2.75,1) to[out=0,in=270] (3,1.5);
	\draw[usual,crossline,directed=1] (1,1.5) to[out=90,in=270] (1,3);
	\draw[usual,crossline] (3,1.5) to[out=90,in=0] (2.75,2) to[out=180,in=0] (1.75,2) to[out=180,in=270] (1.5,2.25);
	\draw[usual,crossline,directed=1] (2,1.5) to[out=90,in=270] (2,3);
	\draw[usual,crossline,directed=1] (1.5,2.25) 
	to[out=90,in=180] (1.75,2.5) to[out=0,in=180] (2.75,2.5) to[out=0,in=270] (3,3);
	\draw[usual] (1,4.5) to [out=270,in=180] (1.5,4);
	\draw[usual] (2,4.5) to [out=270,in=0] (1.5,4);
	\draw[usual] (1.5,4) to (1.5,3.5);	
	\draw[usual] (1,3) to [out=90,in=180] (1.5,3.5);
	\draw[usual] (2,3) to [out=90,in=0] (1.5,3.5);
	\draw[usual,crossline] (3,3) to (3,4.5);
	\draw[thick,mygray,fill=citron,fill opacity=.1] (0,3) 
	to (3.5,3) to (3.5,0) to (0,0) to (0,3);
	\node at (.5,1.5) {$\braid^{\prime}$};
\end{tikzpicture}
\end{gather}
Using the decategorification of the first equation in \fullref{example:uncolored}, 
this in turn would imply that the HOMFLYPT polynomials of the links given as the closures of
\begin{gather}
\begin{tikzpicture}[anchorbase,scale=.5,tinynodes]
	\draw[usual,crossline] (1,0) to[out=90,in=270] (1,1.5);
	\draw[usual,crossline] (3,0) to[out=90,in=0] (2.75,.5) 
	to[out=180,in=0] (1.75,.5) to[out=180,in=270] (1.5,.75);
	\draw[usual,crossline] (2,0) to[out=90,in=270] (2,1.5);
	\draw[usual,crossline] (1.5,.75) to[out=90,in=180] (1.75,1) 
	to[out=0,in=180] (2.75,1) to[out=0,in=270] (3,1.5);
	\draw[usual,crossline,directed=1] (2,1.5) to[out=90,in=270] (2,3);
	\draw[usual,crossline] (3,1.5) to[out=90,in=0] (2.75,2) to[out=180,in=0] (.75,2) to[out=180,in=270] (.5,2.25);
	\draw[usual,crossline,directed=1] (1,1.5) to[out=90,in=270] (1,3);
	\draw[usual,crossline,directed=1] (.5,2.25) to[out=90,in=180] (.75,2.5) 
	to[out=0,in=180] (2.75,2.5) to[out=0,in=270] (3,3);
	\draw[usual,crossline] (2,3) to [out=90,in=270] (1,4.5);
	\draw[usual,crossline] (1,3) to [out=90,in=270] (2,4.5);
	\draw[usual,crossline] (3,3) to (3,4.5);
	\draw[thick,mygray,fill=citron,fill opacity=.1] (0,3) 
	to (3.5,3) to (3.5,0) to (0,0) to (0,3);
	\node at (.5,1.5) {$\braid$};
\end{tikzpicture}
\quad\&\quad
\begin{tikzpicture}[anchorbase,scale=.5,tinynodes]
	\draw[usual,crossline] (2,0) to[out=90,in=270] (2,1.5);
	\draw[usual,crossline] (3,0) to[out=90,in=0] (2.75,.5) 
	to[out=180,in=0] (.75,.5) to[out=180,in=270] (.5,.75);
	\draw[usual,crossline] (1,0) to[out=90,in=270] (1,1.5);
	\draw[usual,crossline] (.5,.75) to[out=90,in=180] (.75,1) 
	to[out=0,in=180] (2.75,1) to[out=0,in=270] (3,1.5);
	\draw[usual,crossline,directed=1] (1,1.5) to[out=90,in=270] (1,3);
	\draw[usual,crossline] (3,1.5) to[out=90,in=0] (2.75,2) 
	to[out=180,in=0] (1.75,2) to[out=180,in=270] (1.5,2.25);
	\draw[usual,crossline,directed=1] (2,1.5) to[out=90,in=270] (2,3);
	\draw[usual,crossline,directed=1] (1.5,2.25) to[out=90,in=180] (1.75,2.5) 
	to[out=0,in=180] (2.75,2.5) to[out=0,in=270] (3,3);
	\draw[usual,crossline] (2,3) to [out=90,in=270] (1,4.5);
	\draw[usual,crossline] (1,3) to [out=90,in=270] (2,4.5);
	\draw[usual,crossline] (3,3) to (3,4.5);
	\draw[thick,mygray,fill=citron,fill opacity=.1] (0,3) 
	to (3.5,3) to (3.5,0) to (0,0) to (0,3);
	\node at (.5,1.5) {$\braid^{\prime}$};
\end{tikzpicture}
\end{gather}
agree. 
However, a computation shows that the difference 
between their (reduced) HOMFLYPT polynomials is 
$(a-a^{\tm 1})^{2}-(q-q^{\tm 1})^{2}$, where $a$, $q$ 
are variables (at the decategorified level) corresponding to $\apar$, $\qpar$.
\end{proof}

We now turn out attention to the behavior of $\cHH_{\hbody}(\braid)$ 
under stabilization \eqref{eq:stabilization}.
Our main technical tool will be the partial Hochschild 
trace from \cite[Section 3]{Ho-young-torus}, 
which we now adapt to the colored setting.
The construction of this functor is motivated as follows.
Since Hochschild cohomology satisfies the classical conjugation relation 
(\ie the relation \eqref{eq:classical-conjugation}), 
we informally view this operation as a mean 
to take the closure of (the singular Soergel bimodule associated to) a web diagram. 
In order to study the stabilization relation, we would like to be 
able to take this closure ``one strand at a time''
in a manner that is compatible with taking Hochschild cohomology.

Recall that the Hochschild cohomology of an $\rring[\stuple]$-bimodule $\morstuff{M}$ is defined as
\begin{gather}
\cHH[a](\rring[\stuple],\morstuff{M}) 
=
\EXT_{\eering{\rring[\stuple]}}^{a}(\rring[\stuple],\morstuff{M}) 
\ouriso
\HOM_{\ddcat{\bimodq{\rring[\stuple]}}}(\rring[\stuple],\hpar^{a}\morstuff{M}),
\end{gather}
where here we follow \fullref{convention:grading} for the $\qpar$-graded $\Hom$.
Here $\dcat{\bimodq{\rring[\stuple]}}$ is the bounded derived category, 
and 
we emphasize that the homological degree therein
is not the $\tpar$-degree from \fullref{subsec:genusg-complex}, 
but rather the $\hpar$-degree from \fullref{example:hochschild}
(which should be viewed as ``perpendicular'' to the homological degree of the Rickard--Rouquier complexes).
Our discussion above suggests that we should consider functors between 
the categories $\dcat{\bimodq{\rring[\stuple]}}$ 
for various $\stuple$ that are compatible with the functors 
$\HOM_{\ddcat{\bimodq{\rring[\stuple]}}}(\rring[\stuple],\placeholder)$.

To this end, 
given $\stuple=(k_{1},\dots,k_{r})$ we let $\stuple^{-}:=(k_{1},\dots,k_{r-1})$, 
\ie $\stuple^{-}$ is obtained from $\stuple$ by removing the last entry. 
Let $\morstuff{Q}_{\stuple}^{k_{r}}:=\eering{\rring[\stuple]} 
\big/(\efunc_{i}(\ualph[r])\otimes 1-1\otimes\efunc_{i}\big(\ualph[r])\big)_{i=1}^{k_{r}}$.
Using the notation in \eqref{eq:el-sym}, we have
\begin{gather}
\rring[\stuple]
\ouriso
\morstuff{Q}_{\stuple}^{k_{r}}
\otimes_{\eering{\rring[\stuple^{-}]}}\rring[\stuple^{-}],
\end{gather}
which suggests that we consider the functor 
$\ifunctor_{\stuple}\colon 
\dcat{\bimodq{\rring[\stuple^{-}]}} 
\to\dcat{\bimodq{\rring[\stuple]}}$ given by derived tensor product 
$\dtimes$ with 
$\morstuff{Q}_{\stuple}^{k_{r}}$ 
over $\eering{\rring[\stuple^{-}]}$. 
We then obtain $\pfunctor_{\stuple}\colon\dcat{\bimodq{\rring[\stuple]}}\to\dcat{\bimodq{\rring[\stuple^{-}]}}$, 
which we define to be the right adjoint to $\ifunctor_{\stuple}$, using derived tensor-hom adjunction.

The functors $\pfunctor_{\stuple}$ and $\ifunctor_{\stuple}$ admit the following explicit descriptions. 
We have an isomorphism
\begin{gather}
\morstuff{Q}_{\stuple}^{k_{r}}
\ouriso
{
\bigotimes_{i=1}^{k_{r}}}\,
\left(
\hpar\qpar^{2i}
\eering{\rring[\stuple]}
\xrightarrow{\efunc_{i}\otimes 1-1\otimes\efunc_{i}}
\eering{\rring[\stuple]}
\right)=:\morstuff{K}_{\stuple}^{k_{r}}
\end{gather}
in $\dcat{\bimodq{\rring[\stuple]}}$, where the (outer) tensor product is taken over $\eering{\rring[\stuple]}$. 
Since $\morstuff{K}_{\stuple}^{k_{r}}$ is a complex of free $\eering{\rring[\stuple]}$-modules, 
given any complex $\morstuff{M}\in\dcat{\bimodq{\rring[\stuple]}}$, 
$\pfunctor_{\stuple}(\morstuff{M})$ is the complex
\begin{gather}
\iHOM_{\eering{\rring[\stuple]}}(\morstuff{K}_{\stuple}^{k_{r}},\morstuff{M})
\ouriso
{
\bigotimes_{i=1}^{k_{r}}}\,
\left(
\morstuff{M}
\xrightarrow{\efunc_{i}\otimes 1-1\otimes\efunc_{i}}
\apar\qpar^{\tm 2i}
\morstuff{M}
\right).
\end{gather}
(In the case that $\morstuff{M}$ is a complex, 
we interpret the latter as a shift of the cone of the indicated chain map.)
Similarly, for $\morstuff{N}\in\dcat{\bimodq{\rring[\stuple^{-}]}}$, 
we have that 
\begin{gather}
\ifunctor_{\stuple}(\morstuff{N}) 
=\morstuff{K}_{\stuple}^{k_{r}}\otimes_{\eering{\rring[\stuple^{-}]}}\morstuff{N},
\end{gather}
where we again interpret the latter as the total complex of this double complex. 
Here, since $\morstuff{K}_{\stuple}^{k_{r}}\ouriso\morstuff{Q}_{\stuple}^{k_{r}}$ 
and the latter is a free $\rring[\stuple^{-}]$-bimodule, 
we also have
\begin{gather}
\ifunctor_{\stuple}(\morstuff{N})= 
\morstuff{Q}_{\stuple}^{k_{r}}\otimes_{\eering{\rring[\stuple^{-}]}}\morstuff{N}.
\end{gather}

Our next result collects the salient features of 
$\ifunctor_{\stuple}$ and $\pfunctor_{\stuple}$ needed 
for our considerations, 
both of which follow immediately from the definition of 
$\ifunctor_{\stuple}$ and $\pfunctor_{\stuple}$.

\begin{lemma}\label{lemma:partial-trace}
For all $\morstuff{M}\in\dcat{\bimodq{\rring[\stuple]}}$ and all $a\in\Z$, there is a functorial isomorphism 
\begin{gather}
\cHH[a](\rring[\stuple],\morstuff{M})\cong\cHH[a]\big(\rring[\stuple^{-}],\pfunctor_{\stuple}(\morstuff{M})\big).
\end{gather}
Additionally, given $\morstuff{N},\morstuff{P}\in\dcat{\bimodq{\rring[\stuple^{-}]}}$,
we have 
\begin{gather}
\pfunctor_{\stuple}\big(\ifunctor_{\stuple}(N) 
\dtimes_{\rring[\stuple]}\morstuff{M} \dtimes_{\rring[\stuple]}
\ifunctor_{\stuple}(\morstuff{P})\big)
\ouriso
\morstuff{N}\dtimes_{\rring[\stuple]}\pfunctor_{\stuple}(\morstuff{M})\dtimes_{\rring[\stuple]}\morstuff{P}
\end{gather}
Finally, 
setting $\stuple^{1}:=\{k_{1},\dots,k_{s}\}$ and $\stuple^{2}:=\{k_{s{+}1},\dots,k_{r}\}$,
we have
\begin{gather}
\pfunctor_{\stuple}\big(
\morstuff{M}_{1}\otimes_{\K}\morstuff{M}_{2}
\big)
\ouriso
\morstuff{M}_{1}\otimes_{\K}\pfunctor_{\stuple}(\morstuff{M}_{2}).
\end{gather}
for $\morstuff{M}_{1}\in\dcat{\bimodq{\rring[\stuple^{1}]}}$ 
and $\morstuff{M}_{2}\in\dcat{\bimodq{\rring[\stuple^{2}]}}$.\qedhere
\end{lemma}

We will use the functors $\ifunctor_{\stuple}$ and $\pfunctor_{\stuple}$ to 
give a ``local'' proof of invariance under stabilization. 
Note that there is a $\qpar$-degree $0$, fully faithful inclusion functor 
$\ssbim[\stuple]\hookrightarrow\dcat{\bimodq{\rring[\stuple]}}$
given by viewing a singular Soergel bimodule as a complex concentrated in 
$\hpar$-degree zero. 
Further, this functor is monoidal (with respect to $\otimes_{\rring[\stuple]}$ 
on the former and $\dtimes_{\rring[\stuple]}$ on the latter)
since singular Soergel bimodules are free as either left or right 
$\rring[\stuple]$-modules.
(The latter can be deduced from the fact that $\rring[\stuple]$ is free 
over $\rring[{\stuple[J]}]$ for $\stuple\subseteq\stuple[J]$, \cf \fullref{subsec:singular-pre-primer}.)

Given this, we now develop a graphical interpretation for the action of the functors $\ifunctor_{\stuple}$ and $\pfunctor_{\stuple}$ 
on singular Bott--Samelson bimodules, again adapting \cite[Section 3.3]{Ho-young-torus} to the singular setting. 
Since our eventual aim is to apply these results to $\cHH_{\hbody}(\placeholder)$, 
we will focus on the (bimodules appearing in the) complex $\khbracket{\braid}{\hbody}$.
Let $\mtuple:=(M,\dots,M,l_{1},\dots,l_{n})$, then, 
for $\morstuff{B}\in\ssbim[\mtuple^{-}]$ and $\morstuff{C}\in\ssbim[\mtuple]$,
we depict $\ifunctor_{\mtuple}$ and $\pfunctor_{\mtuple}$ as follows:
\begin{gather}\label{eq:partial-diagrams}
\ifunctor_{\mtuple}\left(
\begin{tikzpicture}[anchorbase,scale=.5, tinynodes]
	\draw[polebs] (-1,-.25) to[out=90,in=270] (-1,.25);
	\draw[polebs] (0,-.25) to[out=90,in=270] (0,.25);
	\draw[usualbs] (1,-.25) to[out=90,in=270] (1,.25);
	\draw[usualbs] (2,-.25) to[out=90,in=270] (2,.25);
	\draw[usualbs] (1,.75) to[out=90,in=270] (1,1);
	\draw[usualbs] (2,.75) to[out=90,in=270] (2,1);
	\draw[usualbs] (1,1) to (1,2);
	\draw[usualbs] (2,1) to (2,2);
	\draw[polebs] (-1,.75) to[out=90,in=180] (-.5,1);
	\draw[polebs] (0,.75) to[out=90,in=0] (-.5,1);
	\draw[polebs] (-.5,1) to[out=90,in=270] (-.5,1.5);
	\draw[polebs] (-.5,1.5) to[out=180,in=270] (-1,1.75) to (-1,2);
	\draw[polebs] (-.5,1.5) to[out=0,in=270] (0,1.75) to (0,2);
	\draw[black,fill=mygray,fill opacity=.35] (-1.25,.25) rectangle (2.25,.75);
	\node at (.5,.4) {$\morstuff{B}$};
	\node at (-.5,0) {$\text{...}$};
	\node at (-.5,1.875) {$\text{...}$};
	\node at (1.5,0) {$\text{...}$};
	\node at (1.5,1.25) {$\text{...}$};
\end{tikzpicture}
\right)
=
\begin{tikzpicture}[anchorbase,scale=.5, tinynodes]
	\draw[polebs] (-1,-.25) to[out=90,in=270] (-1,.25);
	\draw[polebs] (0,-.25) to[out=90,in=270] (0,.25);
	\draw[usualbs] (1,-.25) to[out=90,in=270] (1,.25);
	\draw[usualbs] (2,-.25) to[out=90,in=270] (2,.25);
	\draw[usualbs] (1,.75) to[out=90,in=270] (1,1);
	\draw[usualbs] (2,.75) to[out=90,in=270] (2,1);
	\draw[usualbs] (1,1) to (1,2);
	\draw[usualbs] (2,1) to (2,2);
	\draw[polebs] (-1,.75) to[out=90,in=180] (-.5,1);
	\draw[polebs] (0,.75) to[out=90,in=0] (-.5,1);
	\draw[polebs] (-.5,1) to[out=90,in=270] (-.5,1.5);
	\draw[polebs] (-.5,1.5) to[out=180,in=270] (-1,1.75) to (-1,2);
	\draw[polebs] (-.5,1.5) to[out=0,in=270] (0,1.75) to (0,2);
	\draw[black,fill=mygray,fill opacity=.35] (-1.25,.25) rectangle (2.25,.75);
	\node at (.5,.4) {$\morstuff{B}$};
	\node at (-.5,0) {$\text{...}$};
	\node at (-.5,1.875) {$\text{...}$};
	\node at (1.5,0) {$\text{...}$};
	\node at (1.5,1.25) {$\text{...}$};
	\draw[usualbs] (3,-.25) node[below]{$l_{n}$} to (3,2) node[above,yshift=-2pt]{$l_{n}$};
\end{tikzpicture}
\quad\&\quad
\pfunctor_{\mtuple} 
\left(
\begin{tikzpicture}[anchorbase,scale=.5, tinynodes]
	\draw[polebs] (-1,-.25) to[out=90,in=270] (-1,.25);
	\draw[polebs] (0,-.25) to[out=90,in=270] (0,.25);
	\draw[usualbs] (1,-.25) to[out=90,in=270] (1,.25);
	\draw[usualbs] (2,-.25) to[out=90,in=270] (2,.25);
	\draw[usualbs] (1,.75) to[out=90,in=270] (1,1);
	\draw[usualbs] (2,.75) to[out=90,in=270] (2,1);
	\draw[usualbs] (1,1) to (1,2);
	\draw[usualbs] (2,1) to (2,2);
	\draw[polebs] (-1,.75) to[out=90,in=180] (-.5,1);
	\draw[polebs] (0,.75) to[out=90,in=0] (-.5,1);
	\draw[polebs] (-.5,1) to[out=90,in=270] (-.5,1.5);
	\draw[polebs] (-.5,1.5) to[out=180,in=270] (-1,1.75) to (-1,2);
	\draw[polebs] (-.5,1.5) to[out=0,in=270] (0,1.75) to (0,2);
	\draw[black,fill=mygray,fill opacity=.35] (-1.25,.25) rectangle (2.25,.75);
	\node at (.5,.4) {$\morstuff{C}$};
	\node at (-.5,0) {$\text{...}$};
	\node at (-.5,1.875) {$\text{...}$};
	\node at (1.5,0) {$\text{...}$};
	\node at (1.5,1.25) {$\text{...}$};
\end{tikzpicture}
\right)
=
\begin{tikzpicture}[anchorbase,scale=.5, tinynodes]
	\draw[polebs] (-1,-.25) to[out=90,in=270] (-1,.25);
	\draw[polebs] (0,-.25) to[out=90,in=270] (0,.25);
	\draw[usualbs] (1,-.25) to[out=90,in=270] (1,.25);
	\draw[usualbs] (2,-.25) to[out=90,in=270] (2,.25);
	\draw[usualbs] (1,.75) to[out=90,in=270] (1,1);
	\draw[usualbs] (2,.75) to[out=90,in=270] (2,1);
	\draw[usualbs] (1,1) to (1,2);
	\draw[usualbs] (2,1) to (2,2);
	\draw[polebs] (-1,.75) to[out=90,in=180] (-.5,1);
	\draw[polebs] (0,.75) to[out=90,in=0] (-.5,1);
	\draw[polebs] (-.5,1) to[out=90,in=270] (-.5,1.5);
	\draw[polebs] (-.5,1.5) to[out=180,in=270] (-1,1.75) to (-1,2);
	\draw[polebs] (-.5,1.5) to[out=0,in=270] (0,1.75) to (0,2);
	\draw[black,fill=mygray,fill opacity=.35] (-1.25,.25) rectangle (2.25,.75);
	\node at (.5,.4) {$\morstuff{C}$};
	\node at (-.5,0) {$\text{...}$};
	\node at (-.5,1.875) {$\text{...}$};
	\node at (1.5,0) {$\text{...}$};
	\node at (1.5,1.25) {$\text{...}$};
	\draw[cusual] (2,-.25) to[out=270,in=180] (2.5,-.75) to[out=0,in=270] (3,-.25) 
	to (3,.5) node[right]{$l_{n}$} to (3,2) 
	to[out=90,in=0] (2.5,2.5) to[out=180,in=90] (2,2);
\end{tikzpicture}
\end{gather} 
Similarly,
taking Hochschild cohomology 
will be depicted
by closing all (non-core and core) strands.
In this language, the first statement in \fullref{lemma:partial-trace} says 
that we obtain the same result whether we close all strands at once 
or one at a time, while the second and third are
\begin{gather}\label{eq:tracelocality}
\begin{gathered}
\begin{tikzpicture}[anchorbase,scale=.5, tinynodes]
	\draw[polebs] (-1,-1) to[out=90,in=270] (-1,-.5);
	\draw[polebs] (-2,-1) to[out=90,in=270] (-2,-.5);
	\draw[usualbs] (0,-1) to[out=90,in=270] (0,-.5);
	\draw[usualbs] (1,-1) to[out=90,in=270] (1,-.5);
	\draw[usualbs] (2,-1) to[out=90,in=270] (2,.5);
	\draw[polebs] (-1,0) to[out=90,in=270] (-1,.5);
	\draw[polebs] (-2,0) to[out=90,in=270] (-2,.5);
	\draw[usualbs] (0,0) to[out=90,in=270] (0,.5);
	\draw[usualbs] (1,0) to[out=90,in=270] (1,.5);
	\draw[polebs] (-1,1) to[out=90,in=270] (-1,1.5);
	\draw[polebs] (-2,1) to[out=90,in=270] (-2,1.5);
	\draw[usualbs] (0,1) to[out=90,in=270] (0,1.5);
	\draw[usualbs] (1,1) to[out=90,in=270] (1,1.5);
	\draw[usualbs] (0,2) to[out=90,in=270] (0,2.5);
	\draw[usualbs] (1,2) to[out=90,in=270] (1,2.5);
	\draw[usualbs] (2,1) to[out=90,in=270] (2,3.35);
	\draw[usualbs] (0,2.5) to[out=90,in=270] (0,3.25);
	\draw[usualbs] (1,2.5) to[out=90,in=270] (1,3.25);
	\draw[polebs] (-1,2) to[out=90,in=0] (-1.5,2.25);
	\draw[polebs] (-2,2) to[out=90,in=180] (-1.5,2.25);
	\draw[polebs] (-1.5,2.25) to[out=90,in=270] (-1.5,2.75);
	\draw[polebs] (-1.5,2.75) to[out=180,in=270] (-2,3) to (-2,3.25);
	\draw[polebs] (-1.5,2.75) to[out=0,in=270] (-1,3) to (-1,3.25);
	\draw[black,fill=mygray,fill opacity=.35] (-2.25,-.5) rectangle (1.75,0);
	\draw[black,fill=mygray,fill opacity=.35] (-2.25,.5) rectangle (2.25,1);
	\draw[black,fill=mygray,fill opacity=.35] (-2.25,1.5) rectangle (1.75,2);
	\draw[cusual] (2,-1) to[out=270,in=180] (2.5,-1.5) to[out=0,in=270] (3,-1) 
	to (3,.75) node[right]{$l_{n}$} to (3,3.25) to[out=90,in=0] (2.5,3.75) to[out=180,in=90] (2,3.25);
	\node at (-.25,1.675) {\scalebox{.95}{$\morstuff{N}$}};
	\node at (0,.675) {\scalebox{.95}{$\morstuff{M}$}};
	\node at (-.25,-.325) {\scalebox{.95}{$\morstuff{P}$}};
	\node at (2,4) {\phantom{a}};
	\node at (2,-1.5) {\phantom{a}};
	\node at (-1.5,-.7) {...};
	\node at (-1.5,3.125) {...};
	\node at (.5,2.5) {...};
	\node at (.5,-.7) {...};
\end{tikzpicture}
\hspace{-.25cm}\ouriso
\begin{tikzpicture}[anchorbase,scale=.5, tinynodes]
	\draw[polebs] (-1,-1) to[out=90,in=270] (-1,-.5);
	\draw[polebs] (-2,-1) to[out=90,in=270] (-2,-.5);
	\draw[usualbs] (0,-1) to[out=90,in=270] (0,-.5);
	\draw[usualbs] (1,-1) to[out=90,in=270] (1,-.5);
	\draw[polebs] (-1,0) to[out=90,in=270] (-1,.5);
	\draw[polebs] (-2,0) to[out=90,in=270] (-2,.5);
	\draw[usualbs] (0,0) to[out=90,in=270] (0,.5);
	\draw[usualbs] (1,0) to[out=90,in=270] (1,.5);
	\draw[polebs] (-1,1) to[out=90,in=270] (-1,1.5);
	\draw[polebs] (-2,1) to[out=90,in=270] (-2,1.5);
	\draw[usualbs] (0,1) to[out=90,in=270] (0,1.5);
	\draw[usualbs] (1,1) to[out=90,in=270] (1,1.5);
	\draw[usualbs] (0,2) to[out=90,in=270] (0,2.5);
	\draw[usualbs] (1,2) to[out=90,in=270] (1,2.5);
	\draw[usualbs] (0,2.5) to[out=90,in=270] (0,3.25);
	\draw[usualbs] (1,2.5) to[out=90,in=270] (1,3.25);
	\draw[polebs] (-1,2) to[out=90,in=0] (-1.5,2.25);
	\draw[polebs] (-2,2) to[out=90,in=180] (-1.5,2.25);
	\draw[polebs] (-1.5,2.25) to[out=90,in=270] (-1.5,2.75);
	\draw[polebs] (-1.5,2.75) to[out=180,in=270] (-2,3) to (-2,3.25);
	\draw[polebs] (-1.5,2.75) to[out=0,in=270] (-1,3) to (-1,3.25);
	\draw[black,fill=mygray,fill opacity=.35] (-2.25,-.5) rectangle (1.75,0);
	\draw[black,fill=mygray,fill opacity=.35] (-2.25,.5) rectangle (2.25,1);
	\draw[black,fill=mygray,fill opacity=.35] (-2.25,1.5) rectangle (1.75,2);
	\draw[cusual] (2,.5) to[out=270,in=180] (2.5,0) to[out=0,in=270] (3,.5) 
	to (3,.75) node[right]{$l_{n}$} to (3,1) to[out=90,in=0] (2.5,1.5) to[out=180,in=90] (2,1);
	\node at (-.25, 1.675) {\scalebox{.95}{$\morstuff{N}$}};
	\node at (0,.675) {\scalebox{.95}{$\morstuff{M}$}};
	\node at (-.25,-.325) {\scalebox{.95}{$\morstuff{P}$}};
	\node at (2,4) {\phantom{a}};
	\node at (2,-1.5) {\phantom{a}};
	\node at (-1.5,-.65) {...};
	\node at (-1.5,3.125) {...};
	\node at (.5,2.5) {...};
	\node at (.5,-.65) {...};
\end{tikzpicture}
\\
\begin{tikzpicture}[anchorbase,scale=.5, tinynodes]
	\draw[polebs] (-1,1) to[out=90,in=270] (-1,1.5);
	\draw[polebs] (-2,1) to[out=90,in=270] (-2,1.5);
	\draw[usualbs] (-.625,1) to[out=90,in=270] (-.625,1.5);
	\draw[usualbs] (0,1) to[out=90,in=270] (0,1.5);
	\draw[usualbs] (1,1) to[out=90,in=270] (1,1.5);
	\draw[usualbs] (2,1) to[out=90,in=270] (2,1.5);
	\draw[usualbs] (-.625,2) to[out=90,in=270] (-.625,3.25);
	\draw[usualbs] (0,2) to[out=90,in=270] (0,3.25);
	\draw[usualbs] (1,2) to[out=90,in=270] (1,3.25);
	\draw[usualbs] (2,2) to[out=90,in=270] (2,3.25);
	\draw[polebs] (-1,2) to[out=90,in=0] (-1.5,2.25);
	\draw[polebs] (-2,2) to[out=90,in=180] (-1.5,2.25);
	\draw[polebs] (-1.5,2.25) to[out=90,in=270] (-1.5,2.75);
	\draw[polebs] (-1.5,2.75) to[out=180,in=270] (-2,3) to (-2,3.25);
	\draw[polebs] (-1.5,2.75) to[out=0,in=270] (-1,3) to (-1,3.25);
	\draw[black,fill=mygray,fill opacity=.35] (-2.25,1.5) rectangle (.25,2);
	\draw[black,fill=mygray,fill opacity=.35] (.75,1.5) rectangle (2.25,2);
	\node at (-1,1.675) {\scalebox{.95}{$\morstuff{M}_{1}$}};
	\node at (1.5,1.675) {\scalebox{.95}{$\morstuff{M}_{2}$}};
	\node at (-1.5,1.3) {...};
	\node at (-.3,1.3) {...};
	\node at (-.3,2.5) {...};
	\node at (-1.5,3) {...};
	\node at (1.5,2.5) {...};
	\node at (1.5,1.3) {...};
	\draw[cusual] (2,1) to[out=270,in=180] (2.5,.5) to[out=0,in=270] (3,1) 
	to (3,1.5) node[right,yshift=3pt]{$l_{n}$} to (3,3.25) 
	to[out=90,in=0] (2.5,3.75) to[out=180,in=90] (2,3.25);
\end{tikzpicture}
\ouriso
\begin{tikzpicture}[anchorbase,scale=.5, tinynodes]
	\draw[polebs] (-1,1) to[out=90,in=270] (-1,1.5);
	\draw[polebs] (-2,1) to[out=90,in=270] (-2,1.5);
	\draw[usualbs] (-.625,1) to[out=90,in=270] (-.625,1.5);
	\draw[usualbs] (0,1) to[out=90,in=270] (0,1.5);
	\draw[usualbs] (-.625,2) to[out=90,in=270] (-.625,3.25);
	\draw[usualbs] (0,2) to[out=90,in=270] (0,3.25);
	\draw[polebs] (-1,2) to[out=90,in=0] (-1.5,2.25);
	\draw[polebs] (-2,2) to[out=90,in=180] (-1.5,2.25);
	\draw[polebs] (-1.5,2.25) to[out=90,in=270] (-1.5,2.75);
	\draw[polebs] (-1.5,2.75) to[out=180,in=270] (-2,3) to (-2,3.25);
	\draw[polebs] (-1.5,2.75) to[out=0,in=270] (-1,3) to (-1,3.25);
	\draw[black,fill=mygray,fill opacity=.35] (-2.25,1.5) rectangle (.25,2);
	\node at (-1,1.675) {\scalebox{.95}{$\morstuff{M}_{1}$}};
	\node at (-1.5,1.3) {...};
	\node at (-.3,1.3) {...};
	\node at (-.3,2.5) {...};
	\node at (-1.5,3) {...};
\end{tikzpicture}
\otimes_{\K}
\left(
\begin{tikzpicture}[anchorbase,scale=.5, tinynodes]
	\draw[usualbs] (1,1) to[out=90,in=270] (1,1.5);
	\draw[usualbs] (2,1) to[out=90,in=270] (2,1.5);
	\draw[usualbs] (1,2) to[out=90,in=270] (1,3.25);
	\draw[usualbs] (2,2) to[out=90,in=270] (2,3.25);
	\draw[black,fill=mygray,fill opacity=.35] (.75,1.5) rectangle (2.25,2);
	\node at (1.5,1.675) {\scalebox{.95}{$\morstuff{M}_{2}$}};
	\node at (1.5,2.5) {...};
	\node at (1.5,1.3) {...};
	\draw[cusual] (2,1) to[out=270,in=180] (2.5,.5) to[out=0,in=270] (3,1) 
	to (3,1.5) node[right,yshift=3pt]{$l_{n}$} to (3,3.25) 
	to[out=90,in=0] (2.5,3.75) to[out=180,in=90] (2,3.25);
\end{tikzpicture}
\right)
\end{gathered}
\end{gather}

Next, we compute the value of the colored partial trace on the ``merge-split'' bimodule.
(Strictly speaking, we will only use the $k=l=1$ case of \fullref{lemma:skein-I}, 
which is given \eg in \cite[Equation (3.1b)]{Ho-young-torus}.
However, as we are developing the skein calculus for the colored partial trace, 
and since we anticipate applications of this formula to explicit computations of our invariant, 
we take the opportunity to extend \loccit to the colored setting.)

\begin{lemma}\label{lemma:skein-I}
For $k,l\geq 0$, 
there is an $\apar\tpar\qpar$-degree $0$ isomorphism
\begin{gather}\label{eq:preskein-1}
\begin{tikzpicture}[anchorbase,scale=.5,tinynodes]
	\draw[usualbs] (0,0) node[below]{$k$} to[out=90,in=180] (.5,.5);
	\draw[usualbs] (1,0) to[out=90,in=0] (.5,.5);
	\draw[usualbs] (.5,.5) to[out=90,in=270] (.5,1);
	\draw[usualbs] (.5,1) to[out=180,in=270] (0,1.5) node[above,yshift=-2pt]{$k$};
	\draw[usualbs] (.5,1) to[out=0,in=270] (1,1.5);
	\draw[cusual] (1,0) to[out=270,in=180] (1.5,-.5) to[out=0,in=270] (2,0) to (2,.75) node[right,xshift=-2pt]{$l$}
	to[out=90,in=270] (2,1.5) to[out=90,in=0] (1.5,2) to[out=180,in=90] (1,1.5);
\end{tikzpicture}
\ouriso
{
\prod_{i=1}^{l}}\,
\tfrac{\qpar^{k}+\apar\qpar^{\tm k\tm 2i}}{1-\qpar^{2i}} 
\begin{tikzpicture}[anchorbase,scale=.5,tinynodes]
	\draw[usualbs] (0,0) node[below]{$k$} to (0,1.5) node[above,yshift=-2pt]{$k$};
\end{tikzpicture}
\end{gather}
\end{lemma}

\begin{proof}
In the $k=0$ case, the result simply 
claims that the $l$-colored circle 
is a $\K$-vector space of $\apar\qpar$-graded dimension 
$\prod_{i=1}^{l}\frac{1+\apar\qpar^{\tm 2i}}{1-\qpar^{2i}}$.
This follows directly from \fullref{example:hochschild}.

We thus assume that $k\geq 1$, and proceed as in the proof of \cite[Proposition 3.10]{Ho-young-torus}.
Namely, we explicitly write down the value of $\pfunctor_{\stuple}$ on the bimodule in the 
left-hand side of \eqref{eq:preskein-1}, apply a change of variables, 
and use this to explicitly identify the result in the derived category.

To this end, we assign alphabets of $\qpar$-degree $2$ variables 
to the boundary points of the corresponding web as follows:
\begin{gather}
\begin{tikzpicture}[anchorbase,scale=.5,tinynodes]
	\draw[usualbs] (0,0) node[below]{$\ualph[1]^{\prime}$} to[out=90,in=180] (.5,.5);
	\draw[usualbs] (1,0) node[below]{$\ualph[2]^{\prime}$} to[out=90,in=0] (.5,.5);
	\draw[usualbs] (.5,.5) to[out=90,in=270] (.5,1);
	\draw[usualbs] (.5,1) to[out=180,in=270] (0,1.5) node[above,yshift=-2pt]{$\ualph[1]$};
	\draw[usualbs] (.5,1) to[out=0,in=270] (1,1.5) node[above,yshift=-2pt]{$\ualph[2]$};
\end{tikzpicture}
\end{gather}
where $\#\ualph[1]=k=\#\ualph[1]^{\prime}$ and $\#\ualph[2]=l=\#\ualph[2]^{\prime}$.
Precisely, by this assignment we identify the 
(singular) Bott--Samelson bimodule $\psplit{}{k,l}\pmerge{k,l}{k+l}$ with the 
following quotient of the (shifted) polynomial 
ring generated by the elementary symmetric functions in these alphabets:
\begin{gather}
\qpar^{\tm kl} 
\K\big[
\efunc_{r}(\ualph[1]),\efunc_{r}(\ualph[1]^{\prime}),
\efunc_{s}(\ualph[2]),\efunc_{s}(\ualph[2]^{\prime})
\big]
\Big/
\big(\efunc_{t}(\ualph[1]\cup\ualph[2])-\efunc_{t}(\ualph[1]^{\prime}\cup\ualph[2]^{\prime}) 
\big).
\end{gather}
Here $r,s,t$ are indices ranging $1\leq r\leq k$, $1\leq s\leq l$ and $1\leq t\leq k+l$
(\ie we slightly abuse notation and let $\efunc_{r}(\ualph[1])$ denote 
$\efunc_{1}(\ualph[1]),\ldots,\efunc_{k}(\ualph[1])$, \etc).
The latter is quasi-isomorphic to the object in $\dcat{\bimodq{\rring[k,l]}}$ given by the 
dg algebra
\begin{gather}
\morstuff{K}:=
\qpar^{\tm kl} 
\K\big[
\efunc_{r}(\ualph[1]),\efunc_{r}(\ualph[1]^{\prime}),
\efunc_{s}(\ualph[2]),\efunc_{s}(\ualph[2]^{\prime})
\big]\otimes_{\K}
\extalg\{\theta_{t}\},
\end{gather}
where $\qdeg[\apar\qpar](\theta_{t})=(-1,2t)$ 
and $d(\theta_{t})=\efunc_{t}(\ualph[1]\cup\ualph[2])
-\efunc_{t}(\ualph[1]^{\prime}\cup\ualph[2]^{\prime})$.
Computing partial trace then gives that
\begin{gather}
\pfunctor_{k,l}(\morstuff{K})
\ouriso[\cong]
\apar^{l}
\qpar^{\tm kl}
\qpar^{\tm l(l\tp 1)}
\K\big[
\efunc_{r}(\ualph[1]),\efunc_{r}(\ualph[1]^{\prime}),
\efunc_{s}(\ualph[2]),\efunc_{s}(\ualph[2]^{\prime})
\big]\otimes_{\K}
\extalg\{\theta_{t},\xi_{s}\}
\end{gather}
where $\qdeg[\apar\qpar](\xi_{s})=(-1,2s)$ and $d(\xi_{s})=\efunc_{s}(\ualph[2])-\efunc_{s}(\ualph[2]^{\prime})$.

Since the right-hand side of \eqref{eq:preskein-1} is quasi-isomorphic to a direct sum of copies of the Koszul complex 
associated to the elements $\efunc_{r}(\ualph[1])-\efunc_{t}(\ualph[1])^{\prime}$, 
we now aim to change variables in $\pfunctor_{k,l}(\morstuff{K})$, with the hope of identifying it as such.
Note that
\begin{gather}\label{eq:motivation}
\begin{aligned}
d(\theta_{t}) 
=&\efunc_{t}(\ualph[1]\cup\ualph[2])-\efunc_{t}(\ualph[1]^{\prime}\cup\ualph[2]^{\prime})
=
{
\sum_{j=0}^{t}}\,\efunc_{t-j}(\ualph[1])\efunc_{j}(\ualph[2])-
{
\sum_{j=0}^{t}}\,\efunc_{t-j}(\ualph[1]^{\prime})\efunc_{j}(\ualph[2]^{\prime}) 
\\
=& 
\efunc_{t}(\ualph[1])-\efunc_{t}(\ualph[1]^{\prime})+ 
{
\sum_{j=1}^{t}}\,\efunc_{t-j}(\ualph[1]) 
\big(\efunc_{j}(\ualph[2])-\efunc_{j}(\ualph[2]^{\prime})\big)
+
{
\sum_{j=1}^{t}}\,
\big(\efunc_{t-j}(\ualph[1])-\efunc_{t-j}(\ualph[1]^{\prime})\big)
\efunc_{j}(\ualph[2]^{\prime}) 
\\
=& 
\efunc_{t}(\ualph[1])-\efunc_{t}(\ualph[1]^{\prime})+
{
\sum_{j=1}^{t}}\,\efunc_{t-j}(\ualph[1])d(\xi_{j})
+
{
\sum_{j=1}^{t}}\,
\big(\efunc_{t-j}(\ualph[1])-\efunc_{t-j}(\ualph[1]^{\prime})\big) 
\efunc_{j}(\ualph[2]^{\prime}).
\end{aligned}
\end{gather}
This suggests that we recursively define
\begin{gather}
\Theta_{t}:=
\theta_{t}-{
\sum_{j=0}^{t}}\,\efunc_{t-j}(\ualph[1])\xi_{j} 
-{
\sum_{j=0}^{t}}\,\Theta_{t-j}\efunc_{j}(\ualph[2]^{\prime}).
\end{gather}
By \eqref{eq:motivation}, this gives
\begin{gather}
d(\Theta_{t})=\efunc_{t}(\ualph[1])-\efunc_{t}(\ualph[1]^{\prime}).
\end{gather}
and, in particular, $d(\Theta_{t})=0$ for $t>k$.

It then follows that we have quasi-isomorphisms
\begin{gather}
\begin{aligned}
\pfunctor_{k,l}(\morstuff{K})
&\ouriso[\cong]
\apar^{l}\qpar^{\tm l(k\tp l\tp 1)} 
\K\big[
\efunc_{r}(\ualph[1]),\efunc_{r}(\ualph[1]^{\prime}),
\efunc_{s}(\ualph[2]),\efunc_{s}(\ualph[2]^{\prime})
\big]\otimes_{\K}
\extalg\{\theta_{t},\xi_{s}\} 
\\
&\ouriso[\cong]
\apar^{l}\qpar^{\tm l(k\tp l\tp 1)} 
\K\big[
\efunc_{r}(\ualph[1]),\efunc_{r}(\ualph[1]^{\prime}),
\efunc_{s}(\ualph[2])
\big]\otimes_{\K}
\extalg\{\theta_{t}\} 
\\
&\ouriso[\cong] 
\apar^{l}\qpar^{\tm l(k\tp l\tp 1)} 
\rring[k]\otimes_{\K}
\K\big[\efunc_{s}(\ualph[2])\big]\otimes_{\K}
\extalg\{\Theta_{b}\},
\end{aligned}
\end{gather}
where, in this last equation, the index $b$ ranges from
$k+1,\ldots,k+l$.

This implies that $\pfunctor_{k,l}(\morstuff{K})$ is quasi-isomorphic to a direct sum of
\begin{gather}
\apar^{l}\qpar^{\tm l(k\tp l\tp 1)} 
{
\prod_{i=1}^{l}}\,
\tfrac{1+\apar^{\tm 1}\qpar^{2(k\tp i)}}{1-\qpar^{2i}} 
=
{
\prod_{i=1}^{l}}\,
\tfrac{\qpar^{k}+\apar\qpar^{\tm k\tm 2i}}{1-\qpar^{2i}} 
\end{gather}
copies of $\rring[k]$, as desired.
\end{proof}

\begin{lemma}\label{lemma:vertextrace}
Let $\stuple=(k_{1},\dots,k_{r})$, $\stuple[J]=(k_{1},\dots,k_{r-1}+k_{r})$, 
$\morstuff{B}\in\dcat{\bimodq{\rring[{\stuple[J]}]\text{-}\rring[\stuple]}}$
and $\morstuff{C}\in\dcat{\bimodq{\rring[\stuple]\text{-}\rring[{\stuple[J]}]}}$,
then we have
\begin{gather}\label{eq:slide}
\begin{gathered}
\raisebox{.25cm}{$\begin{tikzpicture}[anchorbase,scale=.5,tinynodes]
	\draw[usualbs] (1,-.25) to[out=90,in=270] (1,.25);
	\draw[usualbs] (0,-.25) to[out=90,in=270] (0,.25);
	\draw[usualbs] (-1,-.25) to[out=90,in=270] (-1,.25);
	\draw[usualbs] (-2,-.25) to[out=90,in=270] (-2,.25);
	\draw[usualbs] (-1,1) to[out=90,in=270] (-1,1);
	\draw[usualbs] (-2,1) to[out=90,in=270] (-2,1);
	\draw[usualbs] (-1,1) to (-1,2);
	\draw[usualbs] (-2,1) to (-2,2);
	\draw[usualbs] (.5,1) to[out=90,in=270] (.5,1.5);
	\draw[usualbs] (.5,1.5) to[out=180,in=270] (1,1.75) to (1,2);
	\draw[usualbs] (.5,1.5) to[out=0,in=270] (0,1.75) to (0,2);
	\draw[black,fill=mygray,fill opacity=.35] (1.25,.25) rectangle (-2.25,1);
	\node at (-.5,.55) {$\morstuff{B}$};
	\node at (-1.5,0) {$\text{...}$};
	\node at (-1.5,1.75) {$\text{...}$};
	\draw[cusual] (1,-.25) to[out=270,in=180] (1.5,-.75) to[out=0,in=270] (2,-.25) 
	to (2,.5) node[right,xshift=-3pt]{$k_{r}$} to (2,2) 
	to[out=90,in=0] (1.5,2.5) to[out=180,in=90] (1,2);
	\draw[cusual] (0,-.25) to[out=270,in=180] (1.5,-1.25) to[out=0,in=270] (3,-.25) 
	to (3,.5) node[right,xshift=-3pt]{$k_{r{-}1}$} to (3,2) 
	to[out=90,in=0] (1.5,3) to[out=180,in=90] (0,2);
	\node at (0,-.75) {$\phantom{a}$};
	\node at (0,3) {$\phantom{a}$};
\end{tikzpicture}$}
\ouriso
\qpar^{2k_{r{-}1}k_r}
\raisebox{.1cm}{$\begin{tikzpicture}[anchorbase,scale=.5, tinynodes]
	\draw[usualbs] (-1,-.625) to[out=90,in=270] (-1,.25);
	\draw[usualbs] (-2,-.625) to[out=90,in=270] (-2,.25);
	\draw[usualbs] (-1,1) to[out=90,in=270] (-1,1);
	\draw[usualbs] (-2,1) to[out=90,in=270] (-2,1);
	\draw[usualbs] (-1,1) to (-1,1.625);
	\draw[usualbs] (-2,1) to (-2,1.625);
	\draw[usualbs] (.5,1) to[out=90,in=270] (.5,1.5);
	\draw[usualbs] (.5,-.5) to (.5,-.125);
	\draw[usualbs] (.5,-.125) to[out=0,in=270] (1,.25);
	\draw[usualbs] (.5,-.125) to[out=180,in=270] (0,.25);
	\draw[black,fill=mygray,fill opacity=.35] (1.25,.25) rectangle (-2.25,1);
	\node at (-.5,.55) {$\morstuff{B}$};
	\node at (-1.5,-.375) {$\text{...}$};
	\node at (-1.5,1.375) {$\text{...}$};
	\draw[cusual] (.5,-.5) to [out=270,in=180] 
	node[right,xshift=-20pt,yshift=-7.5pt]{$k_{r{-}1}{+}k_{r}$}
	(1,-1) to[out=0,in=270] (1.5,-.5) to (1.5,.5) to (1.5,1.5) 
	to[out=90,in=0] (1,2) to[out=180,in=90] (.5,1.5);
	\node at (0,-.75) {$\phantom{a}$};
	\node at (0,3) {$\phantom{a}$};
\end{tikzpicture}$}
\\
\raisebox{-.2cm}{$\begin{tikzpicture}[anchorbase,scale=.5,yscale=-1, tinynodes]
	\draw[usualbs] (1,-.25) to[out=90,in=270] (1,.25);
	\draw[usualbs] (0,-.25) to[out=90,in=270] (0,.25);
	\draw[usualbs] (-1,-.25) to[out=90,in=270] (-1,.25);
	\draw[usualbs] (-2,-.25) to[out=90,in=270] (-2,.25);
	\draw[usualbs] (-1,1) to[out=90,in=270] (-1,1);
	\draw[usualbs] (-2,1) to[out=90,in=270] (-2,1);
	\draw[usualbs] (-1,1) to (-1,2);
	\draw[usualbs] (-2,1) to (-2,2);
	\draw[usualbs] (.5,1) to[out=90,in=270] (.5,1.5);
	\draw[usualbs] (.5,1.5) to[out=180,in=270] (1,1.75) to (1,2);
	\draw[usualbs] (.5,1.5) to[out=0,in=270] (0,1.75) to (0,2);
	\draw[black,fill=mygray,fill opacity=.35] (1.25,.25) rectangle (-2.25,1);
	\node at (-.5,.6) {$\morstuff{C}$};
	\node at (-1.5,0) {$\text{...}$};
	\node at (-1.5,1.75) {$\text{...}$};
	\draw[cusual] (1,-.25) to[out=270,in=180] (1.5,-.75) to[out=0,in=270] (2,-.25) 
	to (2,.5) node[right,xshift=-3pt,yshift=-1pt]{$k_{r}$} to (2,2) 
	to[out=90,in=0] (1.5,2.5) to[out=180,in=90] (1,2);
	\draw[cusual] (0,-.25) to[out=270,in=180] (1.5,-1.25) to[out=0,in=270] (3,-.25) 
	to (3,.5) node[right,xshift=-3pt,yshift=-1pt]{$k_{r{-}1}$} to (3,2) 
	to[out=90,in=0] (1.5,3) to[out=180,in=90] (0,2);
	\node at (0,-.75) {$\phantom{a}$};
	\node at (0,3) {$\phantom{a}$};
\end{tikzpicture}$}
\ouriso
\qpar^{2k_{r{-}1}k_{r}}
\raisebox{-.25cm}{$\begin{tikzpicture}[anchorbase,scale=.5,yscale=-1, tinynodes]
	\draw[usualbs] (-1,-.625) to[out=90,in=270] (-1,.25);
	\draw[usualbs] (-2,-.625) to[out=90,in=270] (-2,.25);
	\draw[usualbs] (-1,1) to[out=90,in=270] (-1,1);
	\draw[usualbs] (-2,1) to[out=90,in=270] (-2,1);
	\draw[usualbs] (-1,1) to (-1,1.625);
	\draw[usualbs] (-2,1) to (-2,1.625);
	\draw[usualbs] (.5,1) to[out=90,in=270] (.5,1.5);
	\draw[usualbs] (.5,-.5) to (.5,-.125);
	\draw[usualbs] (.5,-.125) to[out=0,in=270] (1,.25);
	\draw[usualbs] (.5,-.125) to[out=180,in=270] (0,.25);
	\draw[black,fill=mygray,fill opacity=.35] (1.25,.25) rectangle (-2.25,1);
	\node at (-.5,.60) {$\morstuff{C}$};
	\node at (-1.5,-.375) {$\text{...}$};
	\node at (-1.5,1.375) {$\text{...}$};
	\draw[cusual] (.5,-.5) to [out=270,in=180] (1,-1) to[out=0,in=270] (1.5,-.5) 
	to (1.5,.5) to (1.5,1.5) to[out=90,in=0] 
	node[right,xshift=-30pt, yshift=-7.5pt]{$k_{r{-}1}{+}k_{r}$} (1,2) 
	to[out=180,in=90] (.5,1.5);
	\node at (0,-.75) {$\phantom{a}$};
	\node at (0,3) {$\phantom{a}$};
\end{tikzpicture}$}
\end{gathered}
\end{gather}
\end{lemma}

\begin{proof}
We show the first quasi-isomorphism in \eqref{eq:slide}, as the proof of the second is similar.
The idea for the 
proof is easy: simply pass to a Koszul 
resolutions at the places where the tensor products take place.  
Formally, let us consider the case when $r=2$ now,
as the general proof differs only in requiring more cumbersome notation. 
The left-hand side of the first isomorphism in \eqref{eq:slide} is
\begin{gather}
\pfunctor_{\stuple^{-}}\big( 
\pfunctor_{\stuple}( 
\psplit{\stuple[J]}{\stuple}
\otimes_{\rring[{\stuple[J]}]} 
\morstuff{B}) \big) 
\cong 
\apar^{k_{1}+k_{2}}\qpar^{\tm k_{1}^{2}\tm k_{1}\tm k_{2}^{2}\tm k_{2}} 
\big( 
\psplit{\stuple[J]}{\stuple} 
\otimes_{\rring[{\stuple[J]}]} {_{\stuple[J]}\morstuff{B}_{\stuple^{\prime}}} 
\big) 
\otimes_{\K}\extalg\{\xi_{r},\zeta_{s}\},
\end{gather}
where $1\leq r\leq k_{1}$, $1\leq s\leq k_{2}$ and with differential given by 
$d(\xi_{r})=\efunc_{r}(\ualph[1]) 
-\efunc_{r}(\ualph[1]^{\prime})$, 
$d(\zeta_{s})=\efunc_{s}(\ualph[2]) 
-\efunc_{s}(\ualph[2]^{\prime})$
for alphabets of size 
$|\ualph[1]|=k_{1}=|\ualph[1]^{\prime}|$ 
and $|\ualph[2]|=k_{2}=|\ualph[2]^{\prime}|$, respectively. 
Here, polynomials in the relevant
alphabets act as indicated by the subscripts on the bimodules. 
Passing to a Koszul resolution of the diagonal 
$\rring[{\stuple[J]}]$-bimodule, we see this is quasi-isomorphic to
\begin{gather}\label{eq:topS}
\apar^{k_{1}+k_{2}}\qpar^{\tm k_{1}^{2}\tm k_{1}\tm k_{2}^{2}\tm k_{2}} 
\big( 
\psplit{\stuple[J]}{\stuple}\otimes_{\K}{_{\stuple[J]^{\prime}}\morstuff{B}_{\stuple^{\prime}}} 
\big)\otimes_{\K}\extalg\{\xi_{r},\zeta_{s},\theta_{t}\},
\end{gather}
where here (additionally) $1\leq t \leq k_{1}+k_{2}$, $d(\theta_{t})=\efunc_{t}(\ualph[]) 
-\efunc_{t}(\ualph[]^{\prime})$, and $|\ualph[]|=k_{1}+k_{2}=|\ualph[]^{\prime}|$.

Similarly, the right-hand side is
\begin{gather}
\qpar^{2k_{1}k_{2}}\pfunctor_{\stuple[J]}
\big(\morstuff{B} 
\otimes_{\rring[\stuple]} 
\psplit{\stuple[J]}{\stuple}\big) 
\cong 
\apar^{k_{1}+k_{2}}\qpar^{2k_{1}k_{2}}
\qpar^{\tm(k_{1}+k_{2})(k_{1}+k_{2}+1)} 
\big( 
{_{\stuple[J^{\prime}]}\morstuff{B}_{\stuple}} 
\otimes_{\rring[\stuple]}\psplit{\stuple[J]}{\stuple} 
\big) 
\otimes_{\K}\extalg\{\Theta_{t}\},
\end{gather}
where $1\leq t \leq k_{1}+k_{2}$, 
$d(\Theta_{t})=\efunc_{t}(\yalph[]^{\prime}) 
-\efunc_{t}(\yalph[])$, and $|\yalph[]|=k_{1}+k_{2}=|\yalph[]^{\prime}|$.
Passing to a Koszul resolution of $\rring[\stuple]$ gives that this is quasi-isomorphic to
\begin{gather}\label{eq:bottomS}
\apar^{k_{1}+k_{2}}\qpar^{\tm(k_{1}+k_{2})(k_{1}+k_{2}+1)+2k_{1}k_{2}} 
\big(
{_{\stuple[J^{\prime}]}\morstuff{B}_{\stuple^{\prime}}} 
\otimes_{\K} 
\psplit{\stuple[J]}{\stuple}\big) 
\otimes_{\K}\extalg\{\Theta_{t},\Xi_{r},Z_{s}\},
\end{gather}
with $1\leq r \leq k_{1}$, $1\leq s \leq k_{2}$ and differential given by 
$d(\Xi_{r})=\efunc_{r}(\yalph[1]^{\prime}) 
-\efunc_{r}(\yalph[1])$, 
$d(Z_{s}) 
=\efunc_{s}(\yalph[2]^{\prime})-\efunc_{s}(\yalph[2])$
for alphabets of size $|\yalph[1]|=k_{1}=|\yalph[1]^{\prime}|$ 
and $|\yalph[2]|=k_{2}=|\yalph[2]^{\prime}|$. 
The result now follows by comparing \eqref{eq:topS} with \eqref{eq:bottomS}.
\end{proof}

\begin{lemma}\label{lemma:skein-II}
For $k\geq 0$, there are $\apar\tpar\qpar$-degree $0$ isomorphisms
\begin{gather}\label{eq:skein-a}
\begin{tikzpicture}[anchorbase,scale=.5,smallnodes]
	\draw[usualbs,crossline] (1,0) to[out=90,in=270] (0,1.5) node[above,yshift=-2pt]{$k$};
	\draw[usualbs,crossline] (0,0) node[below]{$k$} to[out=90,in=270] (1,1.5);
	\draw[cusual] (1,0) to[out=270,in=180] (1.5,-.5) to[out=0,in=270] (2,0) 
	to[out=90,in=270] (2,1.5) to[out=90,in=0] (1.5,2) to[out=180,in=90] (1,1.5);
\end{tikzpicture}
\ouriso[\simeq]
\tpar^{k}\qpar^{\tm k}
\begin{tikzpicture}[anchorbase,scale=.5,smallnodes]
	\draw[usualbs,crossline] (2,0) node[below]{$k$} to[out=90,in=270] (2,1.5) node[above,yshift=-2pt]{$k$};
\end{tikzpicture}
\quad\&\quad
\begin{tikzpicture}[anchorbase,scale=.5,smallnodes]
	\draw[usualbs,crossline] (0,0) node[below]{$k$} to[out=90,in=270] (1,1.5);
	\draw[usualbs,crossline] (1,0) to[out=90,in=270] (0,1.5) node[above,yshift=-2pt]{$k$};
	\draw[cusual] (1,0) to[out=270,in=180] (1.5,-.5) to[out=0,in=270] (2,0) 
	to[out=90,in=270] (2,1.5) to[out=90,in=0] (1.5,2) to[out=180,in=90] (1,1.5);
\end{tikzpicture}
\ouriso[\simeq]
\apar^{k}\qpar^{\tm 2k^{2}\tm k}
\begin{tikzpicture}[anchorbase,scale=.5,smallnodes]
	\draw[usualbs,crossline] (2,0) node[below]{$k$} to[out=90,in=270] (2,1.5) node[above,yshift=-2pt]{$k$};
\end{tikzpicture}
\end{gather}
\end{lemma}

This result, which implies the invariance of the usual colored triply-graded link homology under stabilization, 
is well-known, and follows from the equivalence of the definition in terms of singular Soergel bimodules with the 
constructions in \cite{WeWi-colored-homfly} and \cite{Ca-remarks-triply-graded}.
We give the (well-known) argument for the sake of completion, 
and to determine the exact degree shifts (given our grading conventions for the Rickard--Rouquier complexes) 
so that we may be precise in \eqref{eq:colored-homfly} below.

\begin{proof}
We induct on $k$, starting with $k=1$.
By \fullref{example:uncolored} and \eqref{eq:preskein-1}, we have
\begin{gather}
\begin{tikzpicture}[anchorbase,scale=.5,smallnodes]
	\draw[usualbs,crossline] (1,0) to[out=90,in=270] (0,1.5) node[above,yshift=-2pt]{$1$};
	\draw[usualbs,crossline] (0,0) node[below]{$1$} to[out=90,in=270] (1,1.5);
	\draw[cusual] (1,0) to[out=270,in=180] (1.5,-.5) to[out=0,in=270] (2,0) 
	to[out=90,in=270] (2,1.5) to[out=90,in=0] (1.5,2) to[out=180,in=90] (1,1.5);
\end{tikzpicture}
\ouriso[\simeq]
\tfrac{\qpar+\apar\qpar^{\tm 3}}{1-\qpar^{2}} 
\begin{tikzpicture}[anchorbase,scale=.5,tinynodes]
	\draw[usualbs] (0,0) node[below]{$1$} to (0,2) node[above,yshift=-2pt]{$1$};
\end{tikzpicture}
\longrightarrow
\tfrac{\tpar\qpar^{\tm 1}+\apar\tpar\qpar^{\tm 3}}{1-\qpar^{2}} 
\begin{tikzpicture}[anchorbase,scale=.5,tinynodes]
	\draw[usualbs] (0,0) node[below]{$1$} to (0,2) node[above,yshift=-2pt]{$1$};
\end{tikzpicture}
\quad\&\quad
\begin{tikzpicture}[anchorbase,scale=.5,tinynodes]
	\draw[usualbs,crossline] (0,0) node[below]{$1$} to[out=90,in=270] (1,1.5);
	\draw[usualbs,crossline] (1,0) to[out=90,in=270] (0,1.5) node[above,yshift=-2pt]{$1$};
	\draw[cusual] (1,0) to[out=270,in=180] (1.5,-.5) to[out=0,in=270] (2,0) 
	to[out=90,in=270] (2,1.5) to[out=90,in=0] (1.5,2) to[out=180,in=90] (1,1.5);
\end{tikzpicture}
\ouriso[\simeq]
\tfrac{\tpar^{\tm 1}\qpar+\apar\tpar^{\tm 1}\qpar^{\tm 1}}{1-\qpar^{2}} 
\begin{tikzpicture}[anchorbase,scale=.5,tinynodes]
	\draw[usualbs] (0,0) node[below]{$1$} to (0,2) node[above,yshift=-2pt]{$1$};
\end{tikzpicture}
\longrightarrow
\tfrac{\qpar+\apar\qpar^{\tm 3}}{1-\qpar^{2}} 
\begin{tikzpicture}[anchorbase,scale=.5,tinynodes]
	\draw[usualbs] (0,0) node[below]{$1$} to (0,2) node[above,yshift=-2pt]{$1$};
\end{tikzpicture}
\end{gather}
The proof of \cite[Proposition 3.10]{Ho-young-torus} identifies the differentials in these complexes, 
giving homotopy equivalences
\begin{gather}
\begin{tikzpicture}[anchorbase,scale=.5,smallnodes]
	\draw[usualbs,crossline] (1,0) to[out=90,in=270] (0,1.5) node[above,yshift=-2pt]{$1$};
	\draw[usualbs,crossline] (0,0) node[below]{$1$} to[out=90,in=270] (1,1.5);
	\draw[cusual] (1,0) to[out=270,in=180] (1.5,-.5) to[out=0,in=270] (2,0) 
	to[out=90,in=270] (2,1.5) to[out=90,in=0] (1.5,2) to[out=180,in=90] (1,1.5);
\end{tikzpicture}
\ouriso[\simeq]
\tpar\qpar^{\tm 1} 
\begin{tikzpicture}[anchorbase,scale=.5,tinynodes]
	\draw[usualbs] (0,0) node[below]{$1$} to (0,2) node[above,yshift=-2pt]{$1$};
\end{tikzpicture}
\quad\&\quad
\begin{tikzpicture}[anchorbase,scale=.5,smallnodes]
	\draw[usualbs,crossline] (0,0) node[below]{$1$} to[out=90,in=270] (1,1.5);
	\draw[usualbs,crossline] (1,0) to[out=90,in=270] (0,1.5) node[above,yshift=-2pt]{$1$};
	\draw[cusual] (1,0) to[out=270,in=180] (1.5,-.5) to[out=0,in=270] (2,0) 
	to[out=90,in=270] (2,1.5) to[out=90,in=0] (1.5,2) to[out=180,in=90] (1,1.5);
\end{tikzpicture}
\ouriso[\simeq]
\apar\qpar^{\tm 3}
\begin{tikzpicture}[anchorbase,scale=.5,tinynodes]
	\draw[usualbs] (0,0) node[below]{$1$} to (0,2) node[above,yshift=-2pt]{$1$};
\end{tikzpicture}
\end{gather}
that follow from ``Gaussian elimination'' of all terms for which the $\apar\qpar$-degrees coincide.
For the inductive step, we compute, using \fullref{lemma:vertextrace} and \eqref{eq:pitchfork}, that
\begin{gather}
\begin{aligned}
\tfrac{\qpar^{k}-\qpar^{\tm k}}{\qpar-\qpar^{\tm 1}}
\begin{tikzpicture}[anchorbase,scale=.5,smallnodes]
	\draw[usualbs,crossline] (0,0) to[out=90,in=270] (1,1.5);
	\draw[usualbs,crossline] (1,0) to[out=90,in=270] (0,1.5) node[above,yshift=-2pt]{$k$};
	\draw[cusual] (1,0) to[out=270,in=180] (1.5,-.5) to[out=0,in=270] (2,0) 
	to[out=90,in=270] (2,1.5) to[out=90,in=0] (1.5,2) to[out=180,in=90] (1,1.5);
\end{tikzpicture} 
&\ouriso[\simeq]
\begin{tikzpicture}[anchorbase,scale=.5,tinynodes]
	\draw[usualbs,crossline] (0,0) to[out=90,in=270] (1,1.5);
	\draw[usualbs,crossline] (1,0) to[out=90,in=270] (0,1.5) node[above,yshift=-2pt]{$k$};
	\draw[usualbs] (0,-1.5) node[below]{$k$} to (0,-1) to [out=180,in=180] node[left,xshift=2pt]{$1$} (0,-.25) to (0,0);
	\draw[usualbs] (0,-1) to [out=0,in=0] (0,-.25);
	\draw[cusual] (1,0) to[out=270,in=180] (1.5,-.5) to[out=0,in=270] (2,0) 
	to[out=90,in=270] (2,1.5) to[out=90,in=0] (1.5,2) to[out=180,in=90] (1,1.5);
\end{tikzpicture}
\ouriso[\simeq]
\qpar^{\tm 2(k\tm 1)}
\begin{tikzpicture}[anchorbase,scale=.5,tinynodes]
	\draw[usualbs] (0,4) to[out=90,in=180] (.5,4.5) to (.5,5) node[above,yshift=-2pt]{$k$};
	\draw[usualbs] (1,4) to[out=90,in=0] (.5,4.5);
	\draw[usualbs] (2,4) to (2,5);
	\draw[usualbs] (3,4) to (3,5);
	\draw[usualbs,crossline] (0,1) to[out=90,in=270] (2,4);
	\draw[usualbs,crossline] (1,1) to[out=90,in=270] (3,4);
	\draw[usualbs,crossline] (2,1) to[out=90,in=270] (0,4);
	\draw[usualbs,crossline] (3,1) to[out=90,in=270] (1,4);
	\draw[usualbs] (.5,0) node[below]{$k$} to (.5,.5) to[out=180,in=270] (0,1);
	\draw[usualbs] (.5,.5) to[out=0,in=270] (1,1);
	\draw[usualbs] (2,0) to (2,1);
	\draw[usualbs] (3,0) to (3,1);
	\draw[cusual] (2,0) to[out=270,in=180] (3.5,-1) to[out=0,in=270] (5,0) to (5,2.5) node[left,xshift=2pt]{$1$}
	to[out=90,in=270] (5,5) to[out=90,in=0] (3.5,6) to[out=180,in=90] (2,5);
	\draw[cusual] (3,0) to[out=270,in=180] (3.5,-.5) to[out=0,in=270] (4,0) to (4,2.5) node[left,xshift=2pt]{$k{-}1$} 
	to[out=90,in=270] (4,5) to[out=90,in=0] (3.5,5.5) to[out=180,in=90] (3,5);
\end{tikzpicture}
\ouriso[\simeq]
\apar^{k\tm 1}\qpar^{\tm 2k^{2}\tp k\tp 1}
\begin{tikzpicture}[anchorbase,scale=.5,tinynodes]
	\draw[usualbs] (0,4) to[out=90,in=180] (.5,4.5) to (.5,5) node[above,yshift=-2pt]{$k$};
	\draw[usualbs] (1,4) to[out=90,in=0] (.5,4.5);
	\draw[usualbs] (2,4) to (2,5);
	\draw[usualbs] (0,3) to[out=90,in=270] (0,4);
	\draw[usualbs] (1,3) to[out=90,in=270] (2,4);
	\draw[usualbs,crossline] (2,3) to[out=90,in=270] (1,4);
	\draw[usualbs] (0,2) to[out=90,in=270] (1,3);
	\draw[usualbs,crossline] (1,2) to[out=90,in=270] (0,3);
	\draw[usualbs] (2,2) to[out=90,in=270] (2,3);
	\draw[usualbs] (0,1) to[out=90,in=270] (0,2);
	\draw[usualbs] (1,1) to[out=90,in=270] (2,2);
	\draw[usualbs,crossline] (2,1) to[out=90,in=270] (1,2);
	\draw[usualbs] (.5,0) node[below]{$k$} to (.5,.5) to[out=180,in=270] (0,1);
	\draw[usualbs] (.5,.5) to[out=0,in=270] (1,1);
	\draw[usualbs] (2,0) to (2,1);
	\draw[cusual] (2,0) to[out=270,in=180] (2.5,-.5) to[out=0,in=270] (3,0) to (3,2.5) node[left,xshift=2pt]{$1$} 
	to[out=90,in=270] (3,5) to[out=90,in=0] (2.5,5.5) to[out=180,in=90] (2,5);
\end{tikzpicture}
\\
&\ouriso[\simeq]
\apar^{k\tm 1}\qpar^{\tm 2k^{2}\tp k\tp 1}
\begin{tikzpicture}[anchorbase,scale=.5,tinynodes]
	\draw[usualbs] (0,4) to[out=90,in=180] (.5,4.5) to (.5,5) node[above,yshift=-2pt]{$k$};
	\draw[usualbs] (1,4) to[out=90,in=0] (.5,4.5);
	\draw[usualbs] (2,4) to (2,5);
	\draw[usualbs] (0,3) to[out=90,in=270] (1,4);
	\draw[usualbs,crossline] (1,3) to[out=90,in=270] (0,4);	
	\draw[usualbs] (2,3) to[out=90,in=270] (2,4);
	\draw[usualbs] (0,2) to[out=90,in=270] (0,3);
	\draw[usualbs] (1,2) to[out=90,in=270] (2,3);
	\draw[usualbs,crossline] (2,2) to[out=90,in=270] (1,3);	
	\draw[usualbs] (0,1) to[out=90,in=270] (1,2);
	\draw[usualbs,crossline] (1,1) to[out=90,in=270] (0,2);	
	\draw[usualbs] (2,1) to[out=90,in=270] (2,2);
	\draw[usualbs] (.5,0) node[below]{$k$} to (.5,.5) to[out=180,in=270] (0,1);
	\draw[usualbs] (.5,.5) to[out=0,in=270] (1,1);
	\draw[usualbs] (2,0) to (2,1);
	\draw[cusual] (2,0) to[out=270,in=180] (2.5,-.5) to[out=0,in=270] (3,0) to (3,2.5) node[left,xshift=2pt]{$1$} 
	to[out=90,in=270] (3,5) to[out=90,in=0] (2.5,5.5) to[out=180,in=90] (2,5);
\end{tikzpicture}
\ouriso[\simeq]
\apar^{k}\qpar^{\tm 2k^{2}\tp k\tm 2}
\begin{tikzpicture}[anchorbase,scale=.5,tinynodes]
	\draw[usualbs] (0,4) to[out=90,in=180] (.5,4.5) to (.5,5) node[above,yshift=-2pt]{$k$};
	\draw[usualbs] (1,4) to[out=90,in=0] (.5,4.5);
	\draw[usualbs] (0,3) to[out=90,in=270] (1,4);
	\draw[usualbs,crossline] (1,3) to[out=90,in=270] (0,4);
	\draw[usualbs] (0,2) to[out=90,in=270] (0,3);
	\draw[usualbs] (1,2) to[out=90,in=270] (1,3);
	\draw[usualbs] (0,1) to[out=90,in=270] (1,2);
	\draw[usualbs,crossline] (1,1) to[out=90,in=270] (0,2);
	\draw[usualbs] (.5,0) node[below]{$k$} to (.5,.5) to[out=180,in=270] (0,1) node[left,xshift=2pt]{$1$};
	\draw[usualbs] (.5,.5) to[out=0,in=270] (1,1);
\end{tikzpicture}
\ouriso[\simeq]
\tfrac{\qpar^{k}-\qpar^{\tm k}}{\qpar-\qpar^{\tm 1}}
\apar^{k}\qpar^{\tm 2k^{2}\tm k}
\begin{tikzpicture}[anchorbase,scale=.5,tinynodes]
	\draw[usualbs] (0,0) node[below]{$k$} to (0,4) node[above,yshift=-2pt]{$k$};
\end{tikzpicture}
\end{aligned}
\end{gather}
and the result follows for the negative crossing using the Krull--Schmidt property of 
the derived category, 
see \eg \cite[Lemma 4.20]{Wu-colored-homology}. 
The case of the positive crossing follows from an analogous computation.
\end{proof}

Together with \eqref{eq:tracelocality}, \fullref{lemma:skein-II} proves stabilization 
invariance of $\cHH_{\hbody}(\braid)$ (up to grading shift), 
and consequently completes the proof of \fullref{theorem:homfly}.

Hence, given a balanced, coloring $(\braid,\mtuple)$ of a handlebody braid,
we define
\begin{gather}
w_{(\braid,\mtuple)}:=
\sum_{k=1}^\infty \,
k\cdot\bigg(
\#\Big(\!
\begin{tikzpicture}[anchorbase,scale=.5,tinynodes]
	\draw[usualbs,crossline] (1,0) node[right,xshift=-2pt,yshift=2pt]{$k$} to[out=90,in=270] (0,1.5);
	\draw[usualbs,crossline] (0,0) node[left,xshift=2pt,yshift=2pt]{$k$} to[out=90,in=270] (1,1.5);
\end{tikzpicture}\!
\Big)
-
\#\Big(
\begin{tikzpicture}[anchorbase,scale=.5,tinynodes]
	\draw[usualbs,crossline] (0,0) node[left,xshift=2pt,yshift=2pt]{$k$} to[out=90,in=270] (1,1.5);
	\draw[usualbs,crossline] (1,0) node[right,xshift=-2pt,yshift=2pt]{$k$} to[out=90,in=270] (0,1.5);
\end{tikzpicture}
\Big)
\bigg),
\end{gather}
\ie it is a weighted sum of the difference between the number of purely $k$-colored positive and negative crossings.
Similarly, define
\begin{gather}
W_{(\braid,\mtuple)}:=
\sum_{k=1}^\infty \,
k^{2}
\cdot\bigg(
\#\Big(\!
\begin{tikzpicture}[anchorbase,scale=.5,tinynodes]
	\draw[usualbs,crossline] (1,0) node[right,xshift=-2pt,yshift=2pt]{$k$} to[out=90,in=270] (0,1.5);
	\draw[usualbs,crossline] (0,0) node[left,xshift=2pt,yshift=2pt]{$k$} to[out=90,in=270] (1,1.5);
\end{tikzpicture}\!
\Big)
-
\#\Big(
\begin{tikzpicture}[anchorbase,scale=.5,tinynodes]
	\draw[usualbs,crossline] (0,0) node[left,xshift=2pt,yshift=2pt]{$k$} to[out=90,in=270] (1,1.5);
	\draw[usualbs,crossline] (1,0) node[right,xshift=-2pt,yshift=2pt]{$k$} to[out=90,in=270] (0,1.5);
\end{tikzpicture}
\Big)
\bigg).
\end{gather}


Passing to half-integral values of the $\apar\tpar$-gradings,
we set
\begin{gather}\label{eq:shifts}
\xpar(\braid,\mtuple) 
:=\apar^{\frac{1}{2}(w_{(\braid,\mtuple)}-\sum_{i=1}^{n}l_{i})} 
\tpar^{\frac{1}{2}(-w_{(\braid,\mtuple)}-\sum_{i=1}^{n}l_{i})} 
\qpar^{-W_{(\braid,\mtuple)}+\sum_{i=1}^{n}l_{i}^{2}+l_{i}} 
\end{gather}
and define
\begin{gather}\label{eq:colored-homfly}
\cHHH{\braid,\mtuple}{\hbody}:=
H^{\bullet}(\xpar(\braid,\mtuple)
\cHH_{\hbody}(\braid)),
\end{gather}
where $H^{\bullet}(\placeholder)$ 
denotes taking homology. 

\begin{corollary}\label{corollary:colored-homfly}
For balanced, colored $(\braid,\mtuple)\in\braidg[g,n]$, the triply-graded vector space
$\cHHH{\braid,\mtuple}{\hbody}\in\vecatq$ is an invariant 
of the handlebody link $\bclosure\subset\hbody$.
In general, $\cHHH{\braid,\mtuple}{\hbody}$ is not an 
invariant of the link corresponding to the closure in $\thesphere$ of the 
non-core strands in $\braid$.
\end{corollary}

\begin{proof}
First observe that our normalization factor \eqref{eq:shifts} is invariant under the relations in the handlebody braid group, 
\ie if $\braid,\braid^{\prime}\in\braidg[g,n]$ are colored handlebody braids related by 
\eqref{eq:braid-rels-typeA}, \eqref{eq:braid-rels-typeC}, 
\eqref{eq:braid-rels-special}, then $\xpar(\braid,\mtuple)=\xpar(\braid^{\prime},\mtuple)$. 
Thus, since
$\cHH_{\hbody}(\placeholder)$ is an invariant of handlebody braids, 
the same is true for $\cHHH{\braid,\mtuple}{\placeholder}$.
Conjugation invariance follows from the conjugation invariance of $\cHH_{\hbody}(\braid)$, 
up to homotopy, given in \fullref{theorem:homfly}, together 
with the observation that 
$\xpar(\braid,\mtuple)=\xpar(\topstuff{s}\braid\topstuff{s}^{-1},\topstuff{s}\cdot\mtuple)$ 
(here $\topstuff{s} \cdot \mtuple$ is obtained from $\mtuple$ 
by applying the permutation corresponding to $\topstuff{s}$).

Invariance under stabilization follows from \eqref{eq:tracelocality} 
and \eqref{eq:skein-a}, together with a careful inspection of \eqref{eq:shifts}.

Finally, the second statement follows from (the proof of) \fullref{proposition:links-get-stuck}, 
since this shows that the homology of $\cHH_{\hbody}$ for the braids therein are not isomorphic up to a degree shift. 
\end{proof}


\end{document}